\documentclass[12pt,reqno]{amsart}
\usepackage[cp1251]{inputenc}
\usepackage[T2A]{fontenc}
\usepackage[english]{babel}
\usepackage{amsmath,amsfonts,amssymb}
\usepackage{geometry}
\usepackage{amsmath, amsthm, amscd, amsfonts, amssymb, graphicx, color}
\usepackage[bookmarksnumbered, colorlinks, plainpages]{hyperref}

\numberwithin{equation}{section}

\newtheorem{theorem}{Theorem}[section]
\newtheorem{corollary}[theorem]{Corollary}
\newtheorem{lemma}[theorem]{Lemma}
\newtheorem{proposition}[theorem]{Proposition}

\theoremstyle{remark}
\newtheorem{remark}[theorem]{Remark}
\newtheorem{example}{Example}[subsection]

\theoremstyle{definition}
\newtheorem{definition}[theorem]{Definition}
\newtheorem*{main-definition}{Main Definition}

\begin{document}

\title{An extended Hilbert scale and its applications}


\author[V. Mikhailets]{Vladimir Mikhailets}

\address{Institute of Mathematics of the National Academy of Sciences of Ukraine, 3 Tereshchen\-kivs'ka, Kyiv, 01024, Ukraine}

\email{mikhailets@imath.kiev.ua}


\author[A. Murach]{Aleksandr Murach}

\address{Institute of Mathematics of the National Academy of Sciences of Ukraine, 3 Tereshchen\-kivs'ka, Kyiv, 01024, Ukraine}

\email{murach@imath.kiev.ua}


\author[T. Zinchenko]{Tetiana Zinchenko}

\address{5d Mittelstr., Oranienburg, 16515, Germany}

\email{zinchenkotat@ukr.net}

\subjclass[2010]{46B70, 46E35, 47A40}


\keywords{Hilbert scale, interpolation space, interpolation with function parameter, interpolational inequality, generalized Sobolev space, spectral expansion}

\thanks{This work is supported by the European Union’s Horizon 2020 research and innovation programme under the Marie Sk{\l}odowska-Curie grant agreement No 873071 (SOMPATY: Spectral Optimization: From Mathematics to Physics and Advanced Technology).}

\begin{abstract}
We propose a new viewpoint on Hilbert scales extending them by means of all Hilbert spaces that are interpolation ones between spaces on the scale. We prove that this extension admits an explicit description with the help of $\mathrm{OR}$-varying functions of the operator generating the scale. We also show that this extended Hilbert scale is obtained by the quadratic interpolation (with function parameter) between the above spaces and is closed with respect to the quadratic interpolation between Hilbert spaces. We give applications of the extended Hilbert scale to interpolational inequalities, generalized Sobolev spaces, and spectral expansions induced by abstract and elliptic operators.
\end{abstract}

\maketitle

\section{Introduction}\label{sec1}

Hilbert scales (above all, the Sobolev scale) play an important role in mathematical analysis and the theory of differential equations; see, e.g., the classical monographs \cite{Berezansky68, Hermander85iii, Lax06, LionsMagenes72}, surveys \cite{Agranovich94, Agranovich97, Eidelman94}, and recent book \cite{KoshmanenkoDudkin16}. Such scales are built with respect to an arbitrarily chosen Hilbert space $H$ and a positive definite self-adjoint unbounded operator $A$ acting in this space. As a result, we obtain the Hilbert scale $\{H^{s}_{A}:s\in\mathbb{R}\}$, where $H^{s}_{A}$ is the completion of the domain of $A^{s}$ in the norm $\|A^{s}u\|_{H}$ of a vector~$u$. This scale has the following fundamental property: if $0<\theta<1$, then the mapping  $\{H^{s_0}_{A},H^{s_1}_{A}\}\mapsto H^{s}_{A}$, with $s_0<s_1$ and $s:=(1-\theta)s_0+\theta s_1$, is an exact interpolation functor of type $\theta$ \cite[Theorem~9.1]{KreinPetunin66}. Concerning a linear operator $T$ bounded on both spaces $H^{s_0}_{A}$ and $H^{s_1}_{A}$, this means that $T$ is also bounded on $H^{s}_{A}$ and that the norms of $T$ on these spaces satisfy the inequality
\begin{equation*}
\|T:H^{s}_{A}\to H^{s}_{A}\|\leq
\|T:H^{s_0}_{A}\to H^{s_0}_{A}\|^{1-\theta}\,
\|T:H^{s_1}_{A}\to H^{s_1}_{A}\|^{\theta}.
\end{equation*}
(An analogous property is fulfilled for bounded linear operators that act on pairs of different spaces belonging to two Hilbert scales.) Hence, every space $H^{s}_{A}$ subject to $s_0<s<s_1$ is an interpolation space between $H^{s_0}_{A}$ and $H^{s_1}_{A}$. However, the class of such interpolation Hilbert spaces is far broader than the section $\{H^{s}_{A}:s_0\leq s\leq s_1\}$ of the Hilbert scale.

It is therefore natural to consider the extension of this scale by means of all Hilbert spaces that are interpolation ones between some spaces $H^{s_0}_{A}$ and $H^{s_1}_{A}$, where the numbers $s_0<s_1$ range over $\mathbb{R}$.  Such an extended Hilbert scale is an object of our investigation. We will show that this scale admits a simple explicit description with the help of $\mathrm{OR}$-varying functions of $A$, is obtained by the quadratic interpolation (with function parameter) between the spaces $H^{s_0}_{A}$ and $H^{s_1}_{A}$, and is closed with respect to the quadratic interpolation between Hilbert spaces. These and some other properties of the extended Hilbert scale are considered in Section~\ref{sec2} of this paper; they are proved in Section~\ref{sec3}. Note that the above interpolation and interpolational properties of Hilbert scales are studied in articles \cite{Ameur04, Ameur19, Donoghue67, Fan11, FoiasLions61, Krein60a, KreinPetunin66, Lions58, MikhailetsMurach08MFAT1, Ovchinnikov84, Pustylnik82} (see also monographs \cite[Chapter~1]{LionsMagenes72},
\cite[Section 1.1]{MikhailetsMurach14}, and \cite[Chapters 15 and 30]{Simon19}). Among them, of fundamental importance for our investigation is Ovchinnikov's result \cite[Theorem 11.4.1]{Ovchinnikov84} on an explicit description (with respect to equivalence of norms) of all Hilbert spaces that are interpolation ones between arbitrarily chosen compatible  Hilbert spaces.

The next sections are devoted to various applications of the extended Hilbert scale. Section~\ref{sec3b} considers interpolational inequalities that connect the norms in spaces on the scale to each other, as well as the norms of linear operators acting between extended Hilbert scales. From the viewpoint of inequalities for norms of vectors, this scale can be interpreted as a variable Hilbert scale investigated in \cite{Hegland95, Hegland10, HeglandAnderssen11, MatheTautenhahn06}; the latter appears naturally in the theory of ill-posed problems (see, e.g., \cite{HeglandHofmann11, JinTautenhahn11, MathePereverzev03, TautenhahnHamarikHofmannShao13}). Section~\ref{sec4} gives applications of the extended Hilbert scale to function or distribution spaces, which are used specifically in the theory of pseudodifferential operators. We show that  the extended Hilbert scale generated by some elliptic operators consists of generalized Sobolev spaces whose regularity order is a function $\mathrm{OR}$-varying at infinity. These spaces form the extended Sobolev scale considered in \cite{MikhailetsMurach13UMJ3, MikhailetsMurach15ResMath1} and \cite[Section~2.4.2]{MikhailetsMurach14}. It has important applications to elliptic operators \cite{Murach09UMJ3, MurachZinchenko13MFAT1, ZinchenkoMurach12UMJ11, ZinchenkoMurach14JMathSci} and elliptic boundary-value problems \cite{AnopDenkMurach20arxiv, AnopKasirenko16MFAT, AnopMurach14MFAT, AnopMurach14UMJ, KasirenkoMurach18UMJ11}. Among them are applications to the investigation of various types of convergence of spectral expansions induced by elliptic operators. This topic is examined in the last Section~\ref{sec6}. Its results are based on theorems on the convergence---in a space with two norms---of the spectral expansion induced by an abstract normal operator and on the degree of this convergence. These theorems are proved in Section~\ref{sec6a}.

\section{Basic results}\label{sec2}

Let $H$ be a separable infinite-dimensional complex Hilbert space, with  $(\cdot,\cdot)$ and $\|\cdot\|$ respectively denoting the inner product and the corresponding norm in~$H$. Let $A$ be a positive definite self-adjoint unbounded linear operator in~$H$. The positive definiteness of $A$ means that there exists a number $r>0$ such that $(Au,u)\geq r(u,u)$ for every $u\in\mathrm{Dom}\,A$. As usual, $\mathrm{Dom}\,A$ denotes the domain of $A$. Without loss of generality we suppose that the lower bound $r=1$.

For every $s\in\mathbb{R}$, the self-adjoint operator $A^{s}$ in $H$ is well defined with the help of the spectral decomposition of~$A$. The domain $\mathrm{Dom}\,A^{s}$ of $A^{s}$ is dense in $H$; moreover, $\mathrm{Dom}\,A^{s}=H$ whenever $s\leq0$. Let $H^{s}_{A}$ denote the completion of $\mathrm{Dom}\,A^{s}$ with respect to the norm $\|u\|_{s}:=\|A^{s}u\|$ and the corresponding inner product $(u_{1},u_{2})_{s}:=(A^{s}u_{1},A^{s}u_{2})$, with $u,u_{1},u_{2}\in\mathrm{Dom}\,A^{s}$. The Hilbert space $H^{s}_{A}$ is separable. As usual, we retain designations  $(\cdot,\cdot)_{s}$ and $\|\cdot\|_{s}$ for the inner product and the corresponding norm in this space. Note that the linear manifold $H^{s}_{A}$ coincides with $\mathrm{Dom}\,A^{s}$ whenever $s\geq0$ and that $H^{s}_{A}\supset H$ whenever $s<0$. The set $H^{\infty}_{A}:=\bigcap_{\lambda>0}H^{\lambda}_{A}$
is dense in every space $H^{s}_{A}$, with $s\in\mathbb{R}$.

The class $\{H^{s}_{A}:s\in\mathbb{R}\}$ is called the Hilbert scale generated by $A$ or, simply, $A$-scale (see., e.g., \cite[Section~9, Subsection~1]{KreinPetunin66}). If $s_{0},s_{1}\in\mathbb{R}$ and $s_{0}<s_{1}$, then the identity mapping on $\mathrm{Dom}\,A^{s_{1}}$ extends uniquely to a continuous embedding operator $H^{s_{1}}_{A}\hookrightarrow H^{s_{0}}_{A}$, the embedding being normal. Therefore, interpreting $H^{s_{1}}_{A}$ as a linear manifold in $H^{s_{0}}_{A}$, we obtain the normal pair $[H^{s_{0}}_{A},H^{s_{1}}_{A}]$ of Hilbert spaces. This means that $H^{s_{1}}_{A}$ is dense in $H^{s_{0}}_{A}$ and that $\|u\|_{s_{0}}\leq\|u\|_{s_{1}}$ for every $u\in H^{s_{1}}_{A}$.

\begin{main-definition}
\emph{The extended Hilbert scale generated by $A$} or, simply, \emph{the extended $A$-scale} consists of all Hilbert spaces each of which is an interpolation space for a certain pair $[H^{s_{0}}_{A},H^{s_{1}}_{A}]$ where $s_{0}<s_{1}$ (the real numbers $s_{0}$ and $s_{1}$ may depend on the interpolation Hilbert space).
\end{main-definition}

We will give an explicit description of this scale and prove its important interpolation properties.

Beforehand, let us recall the definition of an interpolation space in the case considered. Suppose $H_{0}$ and $H_{1}$ are Hilbert spaces such that $H_{1}$ is a linear manifold in $H_{0}$ and that the embedding $H_{1}\hookrightarrow H_{0}$ is continuous. A~Hilbert space $X$ is called an interpolation space for the pair $[H_{0},H_{1}]$ (or, in other words, an interpolation space between $H_{0}$ and $H_{1}$) if $X$ satisfies the following two conditions:
\begin{enumerate}
\item [(i)] $X$ is an intermediate space for this pair, i.e. $X$ is a linear manifold in $H_{0}$ and the continuous embeddings $H_{1}\hookrightarrow X\hookrightarrow H_{0}$ hold;
\item [(ii)] for every linear operator $T$ given on $H_{0}$, the following implication is true: if the restriction of $T$ to $H_{j}$ is a bounded operator on $H_{j}$ for each $j\in\{0,1\}$, then the restriction of $T$ to $X$ is a bounded operator on $X$.
\end{enumerate}

Property (ii) implies the following inequality for norms of operators:
\begin{equation*}
\|T:X\to X\|\leq c\,\max\bigl\{\,\|T:H_{0}\to H_{0}\|,\,
\|T:H_{1}\to H_{1} \|\,\bigr\},
\end{equation*}
where $c$ is a certain positive number which does not depend on $T$ (see, e.g., \cite[Theorem 2.4.2]{BerghLefstrem76}). If $c=1$, the interpolation space $X$ is called exact.

Both properties (i) and (ii) are invariant with respect to the change of the norm in $X$ for an equivalent norm. Therefore, it makes sense to describe the interpolation spaces for the pair $[H_{0},H_{1}]$ up to equivalence of norms.

As is known \cite[Theorem 9.1]{KreinPetunin66}, every space $H^{s}_{A}$ is an interpolation one for the pair $[H^{s_{0}}_{A},H^{s_{1}}_{A}]$ whenever $s_{0}\leq s\leq s_{1}$. To give a description of all interpolation Hilbert spaces for this pair, we need more general functions of $A$ than power functions used in the definition of $H^{s}_{A}$.

Choosing a Borel measurable function $\varphi:[1,\infty)\to(0,\infty)$ arbitrarily and using the spectral decomposition of~$A$, we define the self-adjoint operator $\varphi(A)>0$ which acts in~$H$. Recall that $\mathrm{Spec}\,A\subseteq[1,\infty)$ according to our assumption. Here and below, $\mathrm{Spec}\,A$ denotes the spectrum of $A$, and $\varphi(A)>0$ means that $(\varphi(A)u,u)>\nobreak0$ for every $u\in\mathrm{Dom}\,\varphi(A)\setminus\{0\}$. Let $H^{\varphi}_{A}$ denote the completion of the domain $\mathrm{Dom}\,\varphi(A)$ of $\varphi(A)$ with respect to the norm $\|u\|_{\varphi}:=\|\varphi(A)u\|$ of $u\in\mathrm{Dom}\,\varphi(A)$.

The space $H^{\varphi}_{A}$ is Hilbert and separable. Indeed, this norm is induced by the inner product $(u_{1},u_{2})_{\varphi}:=(\varphi(A)u_{1},\varphi(A)u_{2})$ of $u_{1},u_{2}\in\mathrm{Dom}\,\varphi(A)$. Besides, endowing the linear space $\mathrm{Dom}\,\varphi(A)$ with the norm $\|\cdot\|_{\varphi}$ and considering the isometric operator
\begin{equation}\label{f2.1}
\varphi(A):\mathrm{Dom}\,\varphi(A)\to H,
\end{equation}
we infer the separability of $\mathrm{Dom}\,\varphi(A)$ (in the norm $\|\cdot\|_{\varphi}$) from the separability of~$H$. Therefore, the space $H_{A}^{\varphi}$ is separable as well. In the sequel we use the same designations  $(\cdot,\cdot)_{\varphi}$ and $\|\cdot\|_{\varphi}$ for the inner product and the corresponding norm in the whole Hilbert space~$H^{\varphi}_{A}$.

Operator \eqref{f2.1}  extends uniquely (by continuity) to an isometric isomorphism
\begin{equation}\label{f2.2}
B:H_{A}^{\varphi}\leftrightarrow H.
\end{equation}
The equality $B(H_{A}^{\varphi})=H$ follows from the fact that the range of $\varphi(A)$ coincides with $H$ whenever $0\not\in\mathrm{Spec}\,\varphi(A)$ and that the range is narrower than $H$ but is dense in~$H$ whenever $0\in\mathrm{Spec}\,\varphi(A)$. Hence, $(u_{1},u_{2})_{\varphi}=(Bu_{1},Bu_{2})$ for every $u_{1},u_{2}\in H_{A}^{\varphi}$. Besides, $H_{A}^{\varphi}=\mathrm{Dom}\,\varphi(A)$ if and only if $0\not\in\mathrm{Spec}\,\varphi(A)$.

Remark that we use the same designation $H^{\varphi}_{A}$ both in the case where $\varphi$ is a function and in the case where $\varphi$ is a number. This will not lead to ambiguity because we will always specify what $\varphi$ means, a function or number. Of course, this remark also concerns the designations of the norm and inner product in $H^{\varphi}_{A}$.

We need the Hilbert spaces $H^{\varphi}_{A}$ such that $\varphi$ ranges over a certain function class $\mathrm{OR}$. By definition, this class consists of all Borel measurable functions
$\varphi:\nobreak[1,\infty)\rightarrow(0,\infty)$ for which there exist numbers $a>1$ and $c\geq1$ such that $c^{-1}\leq\varphi(\lambda t)/\varphi(t)\leq c$ for all $t\geq1$ and $\lambda\in[1,a]$ (the numbers $a$ and $c$ may depend on $\varphi$). Such functions were introduced by V.~G.~Avakumovi\'c \cite{Avakumovic36} in 1936, are called OR-varying (or O-regularly varying) at infinity and have been well investigated \cite{BinghamGoldieTeugels89, BuldyginIndlekoferKlesovSteinebach18, Seneta76}.

The class $\mathrm{OR}$ admits the following simple description \cite[Theorem~A.1]{Seneta76}: $\varphi\in\mathrm{OR}$ if and only if
\begin{equation*}
\varphi(t)=
\exp\Biggl(\beta(t)+
\int\limits_{1}^{t}\frac{\gamma(\tau)}{\tau}\;d\tau\Biggr),
\quad t\geq1,
\end{equation*}
for some bounded Borel measurable functions $\beta,\gamma:[1,\infty)\to\mathbb{R}$.

This class has the following important property \cite[Theorem~A.2(a)]{Seneta76}: for every $\varphi\in\mathrm{OR}$ there exist real numbers $s_{0}$ and $s_{1}$, with $s_{0}\leq s_{1}$, and positive numbers $c_{0}$ and $c_{1}$ such that
\begin{equation}\label{f2.3}
c_{0}\lambda^{s_{0}}\leq\frac{\varphi(\lambda t)}{\varphi(t)}\leq
c_{1}\lambda^{s_{1}}\quad\mbox{for all}\quad t\geq1\quad\mbox{and}\quad\lambda\geq1.
\end{equation}
Let $\varphi\in\mathrm{OR}$; considering the left-hand side of the inequality \eqref{f2.3} in the $t=1$ case, we conclude that $\varphi(\lambda)\geq\mathrm{const}\cdot e^{-\lambda}$ whenever $\lambda\geq1$. Hence, the identity mapping on $\mathrm{Dom}\,\varphi(A)$ extends uniquely to a continuous embedding operator $H^{\varphi}_{A}\hookrightarrow H^{1/\exp}_{A}$. This will be shown in the first two paragraphs of the proof of Theorem~\ref{th2.6}, in which we put $\varphi_{1}(t):=\varphi(t)$ and $\varphi_{2}(t):=e^{-t}$.
Here, of course, $H^{1/\exp}_{A}$ denotes the Hilbert space $H^{\chi}_{A}$ parametrized with the function $\chi(t):=e^{-t}$ of $t\geq1$. Therefore, we will interpret $H^{\varphi}_{A}$ as a linear manifold in $H^{1/\exp}_{A}$.

Thus, all the spaces $H^{\varphi}_{A}$ parametrized with $\varphi\in\mathrm{OR}$ and, hence, all the spaces from the extended $A$-scale lie in the same space $H^{1/\exp}_{A}$, which enables us to compare them.

\begin{theorem}\label{th2.1}
A Hilbert space $X$ belongs to the extended $A$-scale if and only if $X=\nobreak H^{\varphi}_{A}$ up to equivalence of norms for certain $\varphi\in\mathrm{OR}$.
\end{theorem}

\begin{remark}\label{rem2.2}
We cannot transfer from the extended $A$-scale to a wider class of spaces by means of interpolation Hilbert spaces between any spaces from this scale. Namely, suppose that certain Hilbert spaces $H_{0}$ and $H_{1}$ belong to the extended $A$-scale and satisfy the continuous embedding $H_{1}\hookrightarrow H_{0}$. Then every Hilbert space $X$ which is an interpolation one for the pair $[H_{0},H_{1}]$ belongs to this scale as well. Indeed, for each $j\in\{0,1\}$, the space $H_{j}$ is an interpolation one for a certain pair $[H^{s_{j,0}}_{A},H^{s_{j,1}}_{A}]$, where $s_{j,0}<s_{j,1}$. Besides, both $H^{s_{j,0}}_{A}$ and  $H^{s_{j,1}}_{A}$ are interpolation spaces for the pair $[H^{s_{0}}_{A},H^{s_{1}}_{A}]$ provided that $s_{0}:=\min\{s_{0,0},s_{1,0}\}$ and $s_{1}:=\max\{s_{0,1},s_{1,1}\}$. Hence, the above-mentioned space $X$ is an interpolation one for the latter pair, which follows directly from the given definition of an interpolation space. Thus, $X$ belongs to the extended $A$-scale.
\end{remark}

We will also give an explicit description (up to equivalence of norms) of all Hilbert spaces that are interpolation ones for the given pair
$[H^{s_{0}}_{A},H^{s_{1}}_{A}]$, where $s_{0}<s_{1}$.
Considering $\varphi\in\mathrm{OR}$, we put
\begin{gather}\label{f2.4}
\sigma_{0}(\varphi):=\sup\{s_{0}\in\mathbb{R}\mid\mbox{the left-hand inequality in \eqref{f2.3} holds}\},\\ \label{f2.5}
\sigma_{1}(\varphi):=\inf\{s_{1}\in\mathbb{R}\mid\mbox{the right-hand inequality in \eqref{f2.3} holds}\}.
\end{gather}
Evidently, $-\infty<\sigma_{0}(\varphi)\leq\sigma_{1}(\varphi)<\infty$. The numbers $\sigma_{0}(\varphi)$ and $\sigma_{1}(\varphi)$ are equal to the lower and the upper Matuszewska indices of $\varphi$, respectively (see \cite{Matuszewska64} and \cite[Theorem~2.2.2]{BinghamGoldieTeugels89}).

\begin{theorem}\label{th2.3}
Let $s_{0},s_{1}\in\mathbb{R}$ and $s_{0}<s_{1}$. A Hilbert space $X$ is an interpolation space for the pair $[H^{s_{0}}_{A},H^{s_{1}}_{A}]$
if and only if $X=H^{\varphi}_{A}$ up to equivalence of norms for a certain function parameter $\varphi\in\mathrm{OR}$ that satisfies condition~\eqref{f2.3}.
\end{theorem}

\begin{remark}\label{rem2.4}
Of course, we mean in Theorem~\ref{th2.3} that the positive numbers $c_{0}$ and $c_{1}$ in condition \eqref{f2.3} depend neither on $t$ nor on $\lambda$. Evidently, this condition is equivalent to the following pair of conditions:
\begin{enumerate}
\item [$\mathrm{(i)}$] $s_{0}\leq\sigma_{0}(\varphi)$ and, moreover,
$s_{0}<\sigma_{0}(\varphi)$ if the supremum in $\eqref{f2.4}$ is not attained;
\item [$\mathrm{(ii)}$] $\sigma_{1}(\varphi)\leq s_{1}$ and, moreover, $\sigma_{1}(\varphi)<s_{1}$ if the infimum in $\eqref{f2.5}$ is not attained.
\end{enumerate}
\end{remark}

It is important for applications that the extended $A$-scale can be obtained by means of the quadratic interpolation (with function parameter) between spaces from $A$-scale. Before we formulate a relevant theorem, we will recall the definition of the quadratic interpolation between Hilbert spaces. This interpolation is a natural generalization of the classical interpolation method by J.-L.~Lions \cite{Lions58} and S.~G.~Krein \cite{Krein60a} (see also the book \cite[Chapter~1, Sections 2 and~5]{LionsMagenes72} and survey \cite[Section~9]{KreinPetunin66}) to the case where a general enough function is used, instead of the number $\theta\in(0,\,1)$, as an interpolation parameter. The generalization first appeared in C.~Foia\c{s} and J.-L.~Lions' paper \cite[Section~3.4]{FoiasLions61}. We mainly follow monograph \cite[Section~1.1]{MikhailetsMurach14} (see also \cite[Section~2.1]{MikhailetsMurach08MFAT1}).

Let $\mathcal{B}$ denote the set of all Borel measurable functions $\psi:(0,\infty)\rightarrow(0,\infty)$ such that $\psi$ is bounded on each compact interval $[a,b]$, with $0<a<b<\infty$, and that $1/\psi$ is bounded on every set $[r,\infty)$, with $r>0$. We arbitrarily choose a function $\psi\in\mathcal{B}$ and a regular pair $\mathcal{H}:=[H_{0},H_{1}]$ of separable complex Hilbert spaces. The regularity of this pair means that $H_{1}$ is a dense linear manifold in $H_{0}$ and that the embedding $H_{1}\hookrightarrow H_{0}$ is continuous. For $\mathcal{H}$ there exists a positive definite self-adjoint linear operator $J$ in $H_{0}$ such that $\mathrm{Dom}\,J=H_{1}$ and that $\|Ju\|_{H_{0}}=\|u\|_{H_{1}}$ for every $u\in H_{1}$. The operator $J$ is uniquely determined by the pair $\mathcal{H}$ and is called the generating operator for this pair.

Using the spectral decomposition of $J$, we define the self-adjoint operator $\psi(J)$ in $H_{0}$. Let $[H_{0},H_{1}]_{\psi}$ or, simply, $\mathcal{H}_{\psi}$ denote the domain of $\psi(J)$ endowed with the inner product $(u_{1},u_{2})_{\mathcal{H}_{\psi}}:=(\psi(J)u_{1},\psi(J)u_{2})_{H_{0}}$ and the corresponding norm $\|u\|_{\mathcal{H}_{\psi}}=\|\psi(J)u\|_{H_{0}}$, with $u,u_{1},u_{2}\in\mathcal{H}_{\psi}$. The space $\mathcal{H}_{\psi}$ is Hilbert and separable.

A function $\psi\in\mathcal{B}$ is called an interpolation parameter if the following condition is fulfilled for all regular pairs $\mathcal{H}=[H_{0},H_{1}]$ and
$\mathcal{G}=[G_{0},G_{1}]$ of separable complex Hilbert spaces and for an arbitrary linear mapping $T$ given on $H_{0}$: if the restriction of $T$ to $H_{j}$ is a bounded operator $T:H_{j}\rightarrow G_{j}$ for each $j\in\{0,1\}$, then the restriction of $T$ to
$\mathcal{H}_{\psi}$ is also a bounded operator $T:\mathcal{H}_{\psi}\to\mathcal{G}_{\psi}$. If $\psi$ is an interpolation parameter, we will say that the Hilbert space
$\mathcal{H}_{\psi}$ is obtained by the quadratic interpolation with the function parameter~$\psi$ of the pair $\mathcal{H}$ (or, in other words, between the spaces $H_{0}$ and $H_{1}$). In this case, the dense continuous embeddings $H_{1}\hookrightarrow\mathcal{H}_{\psi}\hookrightarrow H_{0}$ hold true.

A function $\psi\in\mathcal{B}$ is an interpolation parameter if and only if $\psi$ is pseudoconcave in a neighbourhood of infinity. The latter property means that there exists a number $r>0$ and a concave function $\psi_{1}:(r,\infty)\rightarrow(0,\infty)$ that both functions $\psi/\psi_{1}$ and $\psi_{1}/\psi$ are bounded on $(r,\infty)$. This key fact follows from J.~Peetre's \cite{Peetre66, Peetre68} description of all interpolation functions for the weighted $L_{p}(\mathbb{R}^{n})$-type spaces (the description is also set forth in monograph \cite[Theorem 5.4.4]{BerghLefstrem76}).

The above-mentioned interpolation property of the extended $A$-scale is formulated as follows:

\begin{theorem}\label{th2.5}
Let $\varphi\in\mathrm{OR}$, and let real numbers $s_{0}<s_{1}$ be taken from condition~\eqref{f2.3}. Put
\begin{equation}\label{f2.6}
\psi(\tau):=
\begin{cases}
\;\tau^{-s_{0}/(s_{1}-s_{0})}\,\varphi(\tau^{1/(s_{1}-s_{0})}) &\text{whenever}\quad\tau\geq1, \\
\;\varphi(1) & \text{whenever}\quad0<\tau<1.
\end{cases}
\end{equation}
Then the function $\psi$ belongs to $\mathcal{B}$ and is an interpolation parameter, and
\begin{equation}\label{f2.7}
\bigl[H^{s_{0}}_{A},H^{s_{1}}_{A}\bigr]_{\psi}=H^{\varphi}_{A}
\quad\mbox{with equality of norms}.
\end{equation}
\end{theorem}

For instance, considering the function $\varphi(t):=1+\log t$ of $t\geq1$ from the class $\mathrm{OR}$, we can take $s_{0}:=0$ and $s_{1}:=\varepsilon$ for every $\varepsilon>0$ and put $\psi(\tau):=1+\varepsilon^{-1}\log\tau$ whenever $\tau\geq1$ in the interpolation formula \eqref{f2.7}.

Note that, if $s_{0}<\sigma_{0}(\varphi)$ and $s_{1}>\sigma_{1}(\varphi)$, the numbers $s_{0}$ and $s_{1}$ will satisfy the condition of Theorem~\ref{th2.5} whatever $\varphi\in\mathrm{OR}$.

The extended $A$-scale is closed with respect to the quadratic interpolation (with function parameter). This follows directly from the next two results.

\begin{theorem}\label{th2.6}
Let $\varphi_{0},\varphi_{1}:[1,\infty)\to(0,\infty)$ be Borel measurable functions. Suppose that the function $\varphi_{0}/\varphi_{1}$ is bounded on $[1,\infty)$. Then the pair $[H^{\varphi_{0}}_{A},H^{\varphi_{1}}_{A}\bigr]$ is regular. Let $\psi\in\mathcal{B}$, and put
\begin{equation}\label{f2.8}
\varphi(t):=\varphi_{0}(t)\,\psi
\biggl(\frac{\varphi_{1}(t)}{\varphi_{0}(t)}\biggr)
\quad\mbox{whenever}\quad t\geq1.
\end{equation}
Then
\begin{equation}\label{f2.9}
\bigl[H^{\varphi_{0}}_{A},H^{\varphi_{1}}_{A}\bigr]_{\psi}=H^{\varphi}_{A}
\quad\mbox{with equality of norms}.
\end{equation}
\end{theorem}

\begin{proposition}\label{prop2.7}
Let $\varphi_{0},\varphi_{1}\in\mathrm{OR}$ and $\psi\in\mathcal{B}$. Suppose that the function $\varphi_{0}/\varphi_{1}$ is bounded in a neighbourhood of infinity and that $\psi$ is an interpolation parameter. Then the function \eqref{f2.8} belongs to the class $\mathrm{OR}$.
\end{proposition}

This proposition is contained in \cite[Theorem~5.2]{MikhailetsMurach15ResMath1}.

As to Theorem~\ref{th2.6}, it is necessary to note that its hypothesis allows us to consider $H^{\varphi_{1}}_{A}$ and $H^{\varphi}_{A}$ as linear manifolds in $H^{\varphi_0}_{A}$. Indeed, since the functions $\varphi_{0}/\varphi_{1}$ and $\varphi_{0}/\varphi$ are bounded on $[1,\infty)$, the identity mappings on $\mathrm{Dom}\,\varphi_{1}(A)$ and on $\mathrm{Dom}\,\varphi(A)$ extend uniquely to some continuous
embedding operators $H^{\varphi_{1}}_{A}\hookrightarrow H^{\varphi_{0}}_{A}$ and
$H^{\varphi}_{A}\hookrightarrow H^{\varphi_{0}}_{A}$ respectively (see the first two paragraphs of the proof of Theorem~\ref{th2.6}). Thus, we may
say about the regularity of the pair $[H^{\varphi_{0}}_{A},H^{\varphi_{1}}_{A}]$ and compare the spaces $[H^{\varphi_{0}}_{A},H^{\varphi_{1}}_{A}]_{\psi}$ and $H^{\varphi}_{A}$ in \eqref{f2.9}.

\section{Proofs of basic results}\label{sec3}

We will prove Theorems \ref{th2.1}, \ref{th2.3}, \ref{th2.5}, and \ref{th2.6} in the reverse order, which is stipulated by a remarkable result by Ovchinnikov \cite[Theorem 11.4.1]{Ovchinnikov84}. This result explicitly describes (up to equivalence of norms) all the Hilbert spaces that are interpolation ones for an arbitrary compatible pair of Hilbert spaces.  As to our consideration, Ovchinnikov's theorem can be formulated as follows:

\begin{proposition}\label{prop3.1}
Let $\mathcal{H}:=[H_{0},H_{1}]$ be a regular pair of separable complex Hilbert spaces. A Hilbert space $X$ is an interpolation space for $\mathcal{H}$ if and only if $X=\mathcal{H}_{\psi}$ up to equivalence of norms for a certain interpolation parameter $\psi\in\mathcal{B}$.
\end{proposition}

Note that all exact interpolation Hilbert spaces for $\mathcal{H}$ were characterized (isometrically) by Donoghue \cite{Donoghue67}.

Let us turn to the proofs of the theorems formulated in Section~\ref{sec2}.

\begin{proof}[Proof of Theorem $\ref{th2.6}$]
We first show that the pair $[H^{\varphi_{0}}_{A},H^{\varphi_{1}}_{A}]$ is regular. It follows from the hypothesis of the theorem that  $\mathrm{Dom}\,\varphi_{1}(A)\subseteq\mathrm{Dom}\,\varphi_{0}(A)$ and that $\|u\|_{\varphi_{0}}\leq\varkappa^{-1}\|u\|_{\varphi_{1}}$ for every $u\in\mathrm{Dom}\,\varphi_{1}(A)$, with
\begin{equation}\label{f3.2}
\varkappa:=\inf_{t\geq1}
\frac{\varphi_{1}(t)}{\varphi_{0}(t)}>0.
\end{equation}
Hence, the identity mapping on $\mathrm{Dom}\,\varphi_{1}(A)$ extends uniquely to a continuous linear operator
\begin{equation}\label{f3.3}
I:H^{\varphi_{1}}_{A}\to H^{\varphi_{0}}_{A}.
\end{equation}
Let us prove that this operator is injective.

Suppose that $Iu=0$ for certain $u\in H^{\varphi_{1}}_{A}$. We must prove the equality $u=0$. Choose a sequence $(u_{k})_{k=1}^{\infty}\subset\mathrm{Dom}\,\varphi_{1}(A)$ such that $u_{k}\to u$ in $H^{\varphi_{1}}_{A}$ as $k\to\infty$. Since operator \eqref{f3.3} is bounded, we have the convergence $u_{k}=Iu_{k}\to Iu=0$ in $H^{\varphi_{0}}_{A}$. Hence, the sequence $(u_{k})_{k=1}^{\infty}$ is a Cauchy one in $\mathrm{Dom}\,\varphi_{1}(A)$, and $u_{k}\to0$ in $\mathrm{Dom}\,\varphi_{0}(A)$. Here and below in the proof, the linear space $\mathrm{Dom}\,\varphi_{j}(A)$ is endowed with the norm $\|\cdot\|_{\varphi_{j}}$ for each $j\in\{0,1\}$. Thus, there exists a vector $v\in H$ such that $\varphi_{1}(A)u_{k}\to v$ in $H$, and $\varphi_{0}(A)u_{k}\to0$ in $H$. Besides,
\begin{equation*}
\varphi_{0}(A)u_{k}=
\frac{\varphi_{0}}{\varphi_{1}}(A)\varphi_{1}(A)u_{k}\to
\frac{\varphi_{0}}{\varphi_{1}}(A)v\quad\mbox{in}\quad H
\end{equation*}
because the function $\varphi_{0}/\varphi_{1}$ is bounded on $[1,\infty)$. Therefore, $(\varphi_{0}/\varphi_{1})(A)v=0$. Hence, $v=0$ as a vector from $H$ because the function $\varphi_{0}/\varphi_{1}$ is positive on $[1,\infty)$. Thus, $\varphi_{1}(A)u_{k}\to0$ in~$H$. Therefore, $\|u_{k}\|_{\varphi_{1}}=\|\varphi_{1}(A)u_{k}\|\to0$, i.e. $u=\lim_{k\to\infty}u_{k}=0$ in $H^{\varphi_{1}}_{A}$. We have proved that the operator \eqref{f3.3} is injective.

Hence, it realizes a continuous embedding $H^{\varphi_{1}}_{A}\hookrightarrow H^{\varphi_{0}}_{A}$. The density of this embedding follows directly from the density of $\mathrm{Dom}\,\varphi_{1}(A)$ in the normed space $\mathrm{Dom}\,\varphi_{0}(A)$. Let us prove the latter density. Choose a vector $u\in\mathrm{Dom}\,\varphi_{0}(A)$ arbitrarily. The domain of the operator $\varphi_{1}(A)(1/\varphi_{0})(A)$ is dense in $H$ because the closure of this operator coincides with the operator $(\varphi_{1}/\varphi_{0})(A)$, whose domain is dense in~$H$. Hence, there exists a sequence
\begin{equation*}
(v_{k})_{k=1}^{\infty}\subset
\mathrm{Dom}\bigl(\varphi_{1}(A)(1/\varphi_{0})(A)\bigr)
\end{equation*}
such that $v_{k}\to\varphi_{0}(A)u$ in $H$ as $k\to\infty$. Putting
\begin{equation*}
u_{k}:=(1/\varphi_{0})(A)v_{k}\in\mathrm{Dom}\,\varphi_{1}(A)
\end{equation*}
for every integer $k\geq1$, we conclude that $\varphi_{0}(A)u_{k}=v_{k}\to\varphi_{0}(A)u$ in $H$. Therefore, $\mathrm{Dom}\,\varphi_{1}(A)\ni u_{k}\to u$ in $\mathrm{Dom}\,\varphi_{0}(A)$. Hence, the set $\mathrm{Dom}\,\varphi_{1}(A)$ is dense in  $\mathrm{Dom}\,\varphi_{0}(A)$. Thus, the continuous embedding $H^{\varphi_{1}}_{A}\hookrightarrow H^{\varphi_{0}}_{A}$ is dense, i.e. the pair $[H^{\varphi_{0}}_{A},H^{\varphi_{1}}_{A}\bigr]$ is regular.

Let us build the generating operator for this pair. Choosing  $j\in\{0,1\}$ arbitrarily, we have the isometric linear operator
\begin{equation*}
\varphi_{j}(A):\mathrm{Dom}\,\varphi_{j}(A)\to H.
\end{equation*}
This operator extends uniquely (by continuity) to an isometric isomorphism
\begin{equation}\label{f3.4}
B_{j}:H_{A}^{\varphi_{j}}\leftrightarrow H,
\quad\mbox{with}\quad j\in\{0,1\}
\end{equation}
(see the explanation for \eqref{f2.2}). Define the linear operator $J$ in $H_{A}^{\varphi_{0}}$ by the formula $Ju:=B_{0}^{-1}B_{1}u$ for every $u\in\mathrm{Dom}\,J:=H_{A}^{\varphi_{1}}$. Let us prove that $J$ is the generating operator for the pair $[H^{\varphi_{0}}_{A},H^{\varphi_{1}}_{A}]$.

Note first that $J$ sets an isometric isomorphism
\begin{equation}\label{f3.5}
J=B_{0}^{-1}B_{1}:H_{A}^{\varphi_{1}}\leftrightarrow H_{A}^{\varphi_{0}}.
\end{equation}
Hence, the operator $J$ is closed in $H_{A}^{\varphi_{0}}$. Besides, $J$ is a positive definite operator in $H_{A}^{\varphi_{0}}$. Indeed, choosing $u\in\mathrm{Dom}\,\varphi_{1}(A)$ arbitrarily, we write the following:
\begin{align*}
(Ju,u)_{\varphi_{0}}&=(B_{0}^{-1}B_{1}u,u)_{\varphi_{0}}=(B_{1}u,B_{0}u)=
(\varphi_{1}(A)u,\varphi_{0}(A)u)\\
&=\Bigl(\frac{\varphi_{1}}{\varphi_{0}}(A)\varphi_{0}(A)u,
\varphi_{0}(A)u\Bigl)\geq\varkappa(\varphi_{0}(A)u,\varphi_{0}(A)u)=
\varkappa(u,u)_{\varphi_{0}},
\end{align*}
the inequality being due to~\eqref{f3.2}. Passing here to the limit and using \eqref{f3.5}, we conclude that
\begin{equation}\label{f3.5bis}
(Ju,u)_{\varphi_{0}}\geq\varkappa(u,u)_{\varphi_{0}}
\quad\mbox{for every}\quad u\in H_{A}^{\varphi_{1}}.
\end{equation}
Thus, $J$ is a positive definite closed operator in $H_{A}^{\varphi_{0}}$. Moreover, since $0\notin\mathrm{Spec}\,J$ by \eqref{f3.5}, this operator is self-adjoint on $H_{A}^{\varphi_{0}}$. Regarding \eqref{f3.5} again, we conclude that $J$ is the generating operator for the pair $[H^{\varphi_{0}}_{A},H^{\varphi_{1}}_{A}]$.

Let us reduce the self-adjoint operator $J$ in $H_{A}^{\varphi_{0}}$ to an operator of multiplication by function. Since the operator $A$ is self-adjoint in $H$ and since $A\geq1$, there exists a space $R$ with a finite measure $\mu$, a measurable function $\alpha:R\to[1,\infty)$, and an isometric isomorphism
\begin{equation}\label{f3.6}
\mathcal{I}:H\leftrightarrow L_{2}(R,d\mu)
\end{equation}
such that
\begin{equation*}
\mathrm{Dom}\,A=\{u\in H:\alpha\cdot\mathcal{I}u\in L_{2}(R,d\mu)\}
\end{equation*}
and that $\mathcal{I}Au=\alpha\cdot\mathcal{I}u$ for every $u\in\mathrm{Dom}\,A$; see, e.g, \cite[Theorem VIII.4]{ReedSimon72}. Otherwise speaking, $\mathcal{I}$ reduces $A$ to the operator of multiplication by~$\alpha$.

Using \eqref{f3.4} and \eqref{f3.6}, we introduce the isometric isomorphism
\begin{equation}\label{f3.9}
\mathcal{I}_{0}:=\mathcal{I}B_{0}:H^{\varphi_{0}}_{A}\leftrightarrow L_{2}(R,d\mu).
\end{equation}
Let us show that $\mathcal{I}_{0}$ reduces $J$ to an operator of multiplication by function. Given $u\in\mathrm{Dom}\,\varphi_{1}(A)$, we write the following:
\begin{align*}
\mathcal{I}_{0}Ju&=(\mathcal{I}B_{0})(B_{0}^{-1}B_{1})u=
\mathcal{I}B_{1}u=\mathcal{I}\varphi_{1}(A)u\\
&=\mathcal{I}\Bigl(\frac{\varphi_{1}}{\varphi_{0}}\Bigr)(A)\varphi_{0}(A)u
=\Bigl(\frac{\varphi_{1}}{\varphi_{0}}\circ\alpha\Bigr)
\mathcal{I}\varphi_{0}(A)u=
\Bigl(\frac{\varphi_{1}}{\varphi_{0}}\circ\alpha\Bigr)\mathcal{I}_{0}u.
\end{align*}
Thus,
\begin{equation}\label{f3.10}
\mathcal{I}_{0}Ju=
\Bigl(\frac{\varphi_{1}}{\varphi_{0}}\circ\alpha\Bigr)\mathcal{I}_{0}u
\quad\mbox{for every}\quad u\in\mathrm{Dom}\,\varphi_{1}(A).
\end{equation}
Let us prove that this equality holds true for every $u\in H^{\varphi_{1}}_{A}$.

Choose $u\in H^{\varphi_{1}}_{A}$ arbitrarily, and consider a sequence $(u_{k})_{k=1}^{\infty}\subset\mathrm{Dom}\,\varphi_{1}(A)$ such that $u_{k}\to u$ in $H^{\varphi_{1}}_{A}$ as $k\to\infty$. Owing to \eqref{f3.10}, we have the equality
\begin{equation}\label{f3.11}
\mathcal{I}_{0}Ju_{k}=
\Bigl(\frac{\varphi_{1}}{\varphi_{0}}\circ\alpha\Bigr)\mathcal{I}_{0}u_{k}
\quad\mbox{whenever}\quad 1\leq k\in\mathbb{Z}.
\end{equation}
Here,
\begin{equation}\label{f3.12}
\mathcal{I}_{0}Ju_{k}\to\mathcal{I}_{0}Ju\quad\mbox{and}\quad
\mathcal{I}_{0}u_{k}\to\mathcal{I}_{0}u
\quad\mbox{in}\quad L_{2}(R,d\mu)\quad\mbox{as}\quad k\to\infty
\end{equation}
due to the isometric isomorphisms \eqref{f3.5} and \eqref{f3.9}. Since convergence in $L_{2}(R,d\mu)$ implies convergence in the measure $\mu$, it follows from \eqref{f3.12} by the Riesz theorem that
\begin{equation}\label{f3.13}
\mathcal{I}_{0}Ju_{k_l}\to\mathcal{I}_{0}Ju
\quad\mbox{and}\quad \mathcal{I}_{0}u_{k_l}\to\mathcal{I}_{0}u
\quad\mu\mbox{-a.e. on}\;R\quad\mbox{as}\quad l\to\infty
\end{equation}
for a certain subsequence $(u_{k_l})_{l=1}^{\infty}$ of $(u_{k})_{k=1}^{\infty}$. The latter convergence implies that
\begin{equation}\label{f3.14}
\Bigl(\frac{\varphi_{1}}{\varphi_{0}}\circ\alpha\Bigr)
\mathcal{I}_{0}u_{k_l}\to
\Bigl(\frac{\varphi_{1}}{\varphi_{0}}\circ\alpha\Bigr)\mathcal{I}_{0}u
\quad\mu\mbox{-a.e. on}\;R\quad\mbox{as}\quad l\to\infty.
\end{equation}
Now formulas \eqref{f3.11}, \eqref{f3.13}, and \eqref{f3.14} yield the required equality
\begin{equation}\label{f3.15}
\mathcal{I}_{0}Ju=
\Bigl(\frac{\varphi_{1}}{\varphi_{0}}\circ\alpha\Bigr)\mathcal{I}_{0}u
\quad\mbox{for every}\quad u\in H^{\varphi_{1}}_{A}.
\end{equation}

It follows from \eqref{f3.9} and \eqref{f3.15} that
\begin{equation}\label{f3.16}
H^{\varphi_{1}}_{A}\subseteq\Bigl\{u\in H^{\varphi_{0}}_{A}:
\Bigl(\frac{\varphi_{1}}{\varphi_{0}}\circ\alpha\Bigr)\mathcal{I}_{0}u\in
L_{2}(R,d\mu)\Bigr\}=:Q,
\end{equation}
we recalling $H^{\varphi_{1}}_{A}=\mathrm{Dom}\,J$. Let us prove that $H^{\varphi_{1}}_{A}=Q$ in fact. We endow the linear space $Q$ with the norm
\begin{equation*}
\|u\|_{Q}:=\Bigl\|\Bigl(\frac{\varphi_{1}}{\varphi_{0}}\circ\alpha\Bigr)
\mathcal{I}_{0}u\Bigr\|_{L_{2}(R,d\mu)}.
\end{equation*}
Owing to \eqref{f3.15} and \eqref{f3.16}, we have the normal embedding
\begin{equation}\label{f3.17}
H^{\varphi_{1}}_{A}\subseteq Q\quad\mbox{with}\quad
\|u\|_{Q}=\|u\|_{\varphi_{1}}\;\;\mbox{for every}\;\;
u\in H^{\varphi_{1}}_{A}.
\end{equation}

Consider the linear mapping
\begin{equation}\label{f3.18}
L:u\mapsto B_{1}^{-1}\mathcal{I}^{-1}\Bigl[\Bigl
(\frac{\varphi_{1}}{\varphi_{0}}\circ\alpha\Bigr)\cdot
\mathcal{I}_{0}u\Bigr]\quad\mbox{where}\quad u\in Q.
\end{equation}
According to the isometric isomorphisms \eqref{f3.4} and \eqref{f3.6}, we have
\begin{equation}\label{f3.19}
\mbox{the isometric operator}\quad L:Q\to H^{\varphi_{1}}_{A}.
\end{equation}
If $Lu=u$ for every $u\in H^{\varphi_{1}}_{A}$, the required equality $H^{\varphi_{1}}_{A}=Q$ will follow plainly from \eqref{f3.16} and the injectivity of \eqref{f3.19}. Let us prove that $Lu=u$ for every $u\in H^{\varphi_{1}}_{A}$.

Given $u\in\mathrm{Dom}\,\varphi_{1}(A)$, we write the following:
\begin{align*}
Lu&=B_{1}^{-1}\mathcal{I}^{-1}\Bigl[\Bigl
(\frac{\varphi_{1}}{\varphi_{0}}\circ\alpha\Bigr)
\mathcal{I}\varphi_{0}(A)u\Bigr]=
B_{1}^{-1}\mathcal{I}^{-1}\Bigl[\Bigl
(\frac{\varphi_{1}}{\varphi_{0}}\circ\alpha\Bigr)
(\varphi_{0}\circ\alpha)\mathcal{I}u\Bigr]\\
&=B_{1}^{-1}\mathcal{I}^{-1}
\bigl[(\varphi_{1}\circ\alpha)\mathcal{I}u\bigr]=
B_{1}^{-1}\varphi_{1}(A)u=B_{1}^{-1}B_{1}u=u.
\end{align*}
Thus,
\begin{equation}\label{f3.20}
Lu=u\quad\mbox{for every}\quad u\in\mathrm{Dom}\,\varphi_{1}(A).
\end{equation}
Choose now $u\in H^{\varphi_{1}}_{A}$ arbitrarily, and let a sequence $(u_{k})_{k=1}^{\infty}\subset\mathrm{Dom}\,\varphi_{1}(A)$ converge to  $u$ in $H^{\varphi_{1}}_{A}$. Since $u_{k}\to u$ in $Q$ by \eqref{f3.17}, we write
\begin{equation*}
Lu=\lim_{k\to\infty}Lu_{k}=\lim_{k\to\infty}u_{k}=u
\quad\mbox{in}\quad H^{\varphi_{1}}_{A}
\end{equation*}
in view of \eqref{f3.19} and \eqref{f3.20}. Thus, $Lu=u$ for every $u\in H^{\varphi_{1}}_{A}$, and we have proved the required equality
\begin{equation}\label{f3.21}
\mathrm{Dom}\,J=\Bigl\{u\in H^{\varphi_{0}}_{A}:
\Bigl(\frac{\varphi_{1}}{\varphi_{0}}\circ\alpha\Bigr)\mathcal{I}_{0}u\in
L_{2}(R,d\mu)\Bigr\}.
\end{equation}

Formulas \eqref{f3.15} and \eqref{f3.21} mean that the operator $J$ is reduced by the isometric isomorphism \eqref{f3.9} to the operator of multiplication by the function $(\varphi_{1}/\varphi_{0})\circ\alpha$.
Using this fact, we will prove the required formula \eqref{f2.9}.

Since $\psi\in\mathcal{B}$, the function $1/\psi$ is bounded on $[\varkappa,\infty)$. Hence, the function
\begin{equation*}
\frac{\varphi_{0}(t)}{\varphi(t)}=
\frac{1}{\psi}\Bigl(\frac{\varphi_{1}(t)}{\varphi_{0}(t)}\Bigr)
\quad\mbox{of}\quad t\geq1
\end{equation*}
is bounded due to \eqref{f3.2}. Therefore, $\mathrm{Dom}\,\varphi(A)\subseteq\mathrm{Dom}\,\varphi_{0}(A)$.

Choosing $u\in\mathrm{Dom}\,\varphi(A)$ arbitrarily and using the above-mentioned reductions of $A$ and $J$ to operators of multiplication by function, we write the following:
\begin{align*}
L_2(R,d\mu)\ni\mathcal{I}\varphi(A)u&=(\varphi\circ\alpha)\mathcal{I}u=
\Bigl(\psi\circ\frac{\varphi_1}{\varphi_0}\circ\alpha\Bigr)
(\varphi_{0}\circ\alpha)\mathcal{I}u\\
&=\Bigl(\psi\circ\frac{\varphi_1}{\varphi_0}\circ\alpha\Bigr)
\mathcal{I}\varphi_{0}(A)u=
\Bigl(\psi\circ\frac{\varphi_1}{\varphi_0}\circ\alpha\Bigr)
\mathcal{I}_{0}u=\mathcal{I}_{0}\psi(J)u.
\end{align*}
Hence,
\begin{equation*}
\|u\|_{\varphi}=\|\varphi(A)u\|=\|\mathcal{I}\varphi(A)u\|_{L_2(R,d\mu)}=
\|\mathcal{I}_{0}\psi(J)u\|_{L_2(R,d\mu)}=\|\psi(J)u\|_{\varphi_0}.
\end{equation*}
Therefore, $\mathrm{Dom}\,\varphi(A)\subseteq\mathrm{Dom}\,\psi(J)$,
and $\|u\|_{\varphi}=\|u\|_{X}$ for every $u\in\mathrm{Dom}\,\varphi(A)$, where  $X:=[H^{\varphi_{0}}_{A},H^{\varphi_{1}}_{A}]_{\psi}=
\mathrm{Dom}\,\psi(J)$. Passing here to the limit, we infer the normal embedding
\begin{equation}\label{f3.22}
H^{\varphi}_{A}\subseteq X\quad\mbox{with}\quad
\|u\|_{\varphi}=\|u\|_{X}\;\;\mbox{for every}\;\;
u\in H^{\varphi}_{A}.
\end{equation}
Besides, as we have just shown,
\begin{equation}\label{f3.23}
\mathcal{I}\varphi(A)u=\mathcal{I}_{0}\psi(J)u
\quad\mbox{whenever}\quad u\in\mathrm{Dom}\,\varphi(A).
\end{equation}

Let us deduce the equality $H^{\varphi}_{A}=X$ from \eqref{f3.22} and \eqref{f3.23}. Using the isometric isomorphisms \eqref{f2.2}, \eqref{f3.6}, and \eqref{f3.9}, we get
\begin{equation}\label{f3.24}
\mbox{the isometric operator}\quad M:=B^{-1}\mathcal{I}^{-1}\mathcal{I}_{0}\psi(J):X\to H^{\varphi}_{A}.
\end{equation}
Owing to \eqref{f3.23}, we write $Mu=B^{-1}\mathcal{I}^{-1}\mathcal{I}\varphi(A)u=u$ for every $u\in\mathrm{Dom}\,\varphi(A)$. Therefore, choosing $u\in H^{\varphi}_{A}$ arbitrarily and considering a sequence $(u_{k})_{k=1}^{\infty}\subset\mathrm{Dom}\,\varphi(A)$ such that $u_{k}\to u$ in $H^{\varphi}_{A}$, we get
\begin{equation*}
Mu=\lim_{k\to\infty}Mu_{k}=\lim_{k\to\infty}u_{k}=u
\quad\mbox{in}\quad H^{\varphi}_{A}
\end{equation*}
due to \eqref{f3.22}. Now the required equality $H^{\varphi}_{A}=X$ follows from the property $Mu=u$ whenever $u\in H^{\varphi}_{A}$, the inclusion $H^{\varphi}_{A}\subseteq X$, and the injectivity of the operator \eqref{f3.24}. In view of \eqref{f3.22}, we have proved \eqref{f2.9} and, hence, Theorem~\ref{th2.6}.
\end{proof}

We will deduce Theorem~\ref{th2.5} from Theorem~\ref{th2.6} with the help of the following result \cite[Theorem~4.2]{MikhailetsMurach15ResMath1}:

\begin{proposition}\label{prop3.2}
Let $s_{0},s_{1}\in\mathbb{R}$, $s_{0}<s_{1}$, and $\psi\in\mathcal{B}$. Put $\varphi(t):=t^{s_{0}}\psi(t^{s_{1}-s_{0}})$ for every $t\geq1$. Then the function $\psi$ is an interpolation parameter if and only if the function $\varphi$ satisfies \eqref{f2.3} with some positive numbers $c_{0}$ and $c_{1}$ that are independent of $t$ and $\lambda$.
\end{proposition}

\begin{proof}[Proof of Theorem $\ref{th2.5}$]
Let us show first that $\psi\in\mathcal{B}$. Evidently, the function $\psi$ is Borel measurable. Putting $t:=1$ in \eqref{f2.3}, we write $c_{0}\varphi(1)\lambda^{s_{0}}\leq\varphi(\lambda)\leq c_{1}\varphi(1)\lambda^{s_{1}}$ for arbitrary $\lambda\geq1$. Hence, the function $\varphi$ is bounded on every compact subset of $[1,\infty)$, which yields the boundedness of $\psi$ on every interval $(0,b]$ with $b>1$. Besides,
\begin{equation*}
\psi(\tau):=\tau^{-s_{0}/(s_{1}-s_{0})}\,\varphi(\tau^{1/(s_{1}-s_{0})})
\geq c_{0}\varphi(1)\quad\mbox{whenever}\quad\tau\geq1.
\end{equation*}
Therefore, the function $1/\psi$ is bounded on $(0,\infty)$. Thus, $\psi\in\mathcal{B}$ by the definition of~$\mathcal{B}$.

It follows from the definition of $\psi$ that $\varphi(t)=t^{s_{0}}\psi(t^{s_{1}-s_{0}})$ for every $t\geq1$. Hence, $\psi$ is an interpolation parameter according to Proposition~\ref{prop3.2}, whereas the interpolation property \eqref{f2.7} is due to Theorem~\ref{th2.6}, in which we put $\varphi_{0}(t)\equiv t^{s_0}$ and $\varphi_{1}(t)\equiv t^{s_1}$.
\end{proof}

\begin{proof}[Proof of Theorem $\ref{th2.3}$.]
\emph{Necessity.} Suppose that a Hilbert space $X$ is an interpolation space for the pair $[H^{s_{0}}_{A},H^{s_{1}}_{A}]$. Then, owing to Proposition~\ref{prop3.1}, there exists an interpolation parameter $\psi\in\mathcal{B}$ such that $X=[H^{s_{0}}_{A},H^{s_{1}}_{A}]_{\psi}$ up to equivalence of norms. According to Theorem~\ref{th2.6}, we get $[H^{s_{0}}_{A},H^{s_{1}}_{A}]_{\psi}=H^{\varphi}_{A}$ with equality of norms; here, $\varphi(t):=t^{s_{0}}\psi(t^{s_{1}-s_{0}})$ for every $t\geq1$. Thus, $X=H^{\varphi}_{A}$ up to equivalence of norms. Note that $\varphi\in\mathrm{OR}$ due to Proposition~\ref{prop2.7} considered in the case of $\varphi_{0}(t)\equiv t^{s_0}$ and $\varphi_{1}(t)\equiv t^{s_1}$. Moreover, according to Proposition~\ref{prop3.2}, the function $\varphi$ satisfies condition~\eqref{f2.3}. The necessity is proved.

\emph{Sufficiency.} Suppose now that a Hilbert space $X$ coincides with $H^{\varphi}_{A}$ up to equivalence of norms for a certain function parameter $\varphi\in\mathrm{OR}$ that satisfies condition~\eqref{f2.3}. Then, owing to Theorem~\ref{th2.5}, we get $H^{\varphi}_{A}=[H^{s_{0}}_{A},H^{s_{1}}_{A}]_{\psi}$ with equality of norms, where the interpolation parameter $\psi\in\mathcal{B}$ is defined by formula~\eqref{f2.6}. Thus, $X=[H^{s_{0}}_{A},H^{s_{1}}_{A}]_{\psi}$ up to equivalence of norms. This implies that $X$ is an interpolation space for the pair $[H^{s_{0}}_{A},H^{s_{1}}_{A}]$ in view of the definition of an interpolation parameter (or by Proposition~\ref{prop3.1}). The sufficiency is also proved.
\end{proof}

\begin{proof}[Proof of Theorem $\ref{th2.1}$.]
\emph{Necessity.} Suppose that a Hilbert space $X$ belongs to the extended $A$-scale. Then $X$ is an interpolation space for a certain pair $[H^{s_{0}}_{A},H^{s_{1}}_{A}]$ where $s_{0}<s_{1}$. Hence, we conclude by Theorem~\ref{th2.3} that $X=H^{\varphi}(\Omega)$ up to equivalence of norms for a certain function parameter $\varphi\in\mathrm{OR}$. The necessity is proved.

\emph{Sufficiency.} Suppose now that $X=H^{\varphi}_{A}$ up to equivalence of norms for certain $\varphi\in\mathrm{OR}$. The function $\varphi$ satisfies condition~\eqref{f2.3} for the numbers $s_{0}:=\sigma_{0}(\varphi)-1$ and $s_{1}:=\sigma_{1}(\varphi)+1$, for example. Therefore, $X$ is an interpolation space for the pair $[H^{s_{0}}_{A},H^{s_{1}}_{A}]$ due to Theorem~\ref{th2.3}; i.e., $X$ belongs to the extended $A$-scale. The sufficiency is also proved.
\end{proof}

\section{Interpolational inequalities}\label{sec3b}

We assume in this section that functions $\varphi_{0},\varphi_{1}:[1,\infty)\to(0,\infty)$ and $\psi:(0,\infty)\to(0,\infty)$ satisfy the hypothesis of Theorem~\ref{th2.6}; i.e., $\varphi_{0}$ and $\varphi_{1}$ are Borel measurable, and $\varphi_{0}/\varphi_{1}$ is bounded on $[1,\infty)$, and $\psi$ belongs to $\mathcal{B}$. Moreover, suppose that $\psi$ is pseudoconcave in a neighbourhood of infinity; then $\psi$ is an interpolation parameter (see, e.g., \cite[Theorem~1.9]{MikhailetsMurach14}). Owing to Theorem~\ref{th2.6}, we have the dense continuous embedding $H^{\varphi_{1}}_{A}\hookrightarrow H^{\varphi_{0}}_{A}$ and the interpolation formula $[H^{\varphi_{0}}_{A},H^{\varphi_{1}}_{A}]_{\psi}=H^{\varphi}_{A}$ with equality of norms. Here, the Borel measurable function $\varphi:[1,\infty)\to(0,\infty)$ is defined by \eqref{f2.8}. Hence, $H^{\varphi}_{A}$ is an interpolation space between $H^{\varphi_{0}}_{A}$ and~$H^{\varphi_{1}}_{A}$.

We will obtain some inequalities that estimate (from above) the norm in the interpolation space $H^{\varphi}_{A}$ via the norms in the marginal spaces $H^{\varphi_{0}}_{A}$ and $H^{\varphi_{1}}_{A}$ with the help of the interpolation parameter $\psi$. Such inequalities are naturally called interpolational. Specifically, if $\varphi_{0},\varphi_{1}\in\mathrm{OR}$, then $\varphi\in\mathrm{OR}$ as well, due to Proposition~\ref{prop2.7}. In this case, these interpolational inequalities deal with norms in spaces belonging to the extended Hilbert scale.

We denote the number $\varkappa>0$ by formula \eqref{f3.2}. Owing to \cite[Lemma 1.1]{MikhailetsMurach14}, the function $\psi$ is pseudoconcave on $(\varepsilon,0)$ whenever $\varepsilon>0$. Hence, according to \cite[Lemma 1.2]{MikhailetsMurach14}, there exists a number $c_{\psi,\varkappa}>0$ such that
\begin{equation}\label{f3b.2}
\frac{\psi(t)}{\psi(\tau)}\leq c_{\psi,\varkappa}\max\biggl\{1,\frac{t}{\tau}\biggr\}\quad\mbox{for all}\;\; t,\tau\in[\varkappa,\infty).
\end{equation}

\begin{theorem}\label{th3b.1}
Let $\tau\geq\varkappa$ and $u\in H^{\varphi_{1}}_{A}$; then
\begin{equation}\label{f3b.3}
\|u\|_{\varphi}\leq c_{\psi,\varkappa}\,\psi(\tau)
\bigl(\|u\|_{\varphi_0}^{2}+\tau^{-2}\|u\|_{\varphi_1}^{2}\bigr)^{1/2}.
\end{equation}
\end{theorem}

Before we prove this theorem, let us comment formula \eqref{f3b.3}. It follows from \eqref{f3b.2} that $\psi(t)\leq c_{\psi,\varkappa}\psi(\tau)$ whenever $\varkappa\leq t\leq\tau$ and that $\psi(t)/t\leq c_{\psi,\varkappa}\psi(\tau)/\tau$ whenever $\varkappa\leq\tau\leq t$. Hence, $\psi$ is slowly equivalent to an increasing function on the set $[\varkappa,\infty)$, and the function $\psi(\tau)/\tau$ is slowly equivalent to a decreasing function on the same set. Of the main interest is the case where $\psi(\tau)\to\infty$ and $\psi(\tau)/\tau\to0$ as $\tau\to\infty$. In this case, it is useful to rewrite inequality \eqref{f3b.3} in the form
\begin{equation}\label{f3b.4}
\|u\|_{\varphi}\leq c_{\psi,\varkappa}
\biggl(\frac{\psi^{2}(\tau)}{\tau^{2}}\|u\|_{\varphi_1}^{2}+
\psi^{2}(\tau)\|u\|_{\varphi_0}^{2}\biggr)^{1/2}.
\end{equation}
Restricting ourselves to the Hilbert scale $\{H^{s}_{A}:s\in\mathbb{R}\}$, we conclude that this inequality becomes
\begin{equation}\label{f3b.5}
\|u\|_{s}\leq\bigl(\tau^{2(\theta-1)}\|u\|_{s_1}^{2}+
\tau^{2\theta}\|u\|_{s_0}^{2}\bigr)^{1/2}
\quad\mbox{whenever}\;\;u\in H^{s_1}_{A};
\end{equation}
here, the real numbers $s$, $s_0$, $s_1$, $\theta$, and $\tau$ satisfy the conditions
\begin{equation}\label{f3b.6}
s_0<s_1,\quad 0<\theta<1,\quad s=(1-\theta)s_0+\theta s_1,
\end{equation}
and $\tau\geq1$. Indeed, we only need to take
\begin{equation}\label{f3b.7}
\varphi_0(t)\equiv t^{s_0},\quad\varphi_1(t)\equiv t^{s_1},\quad \psi(\tau)\equiv\tau^{\theta},\quad\mbox{and}\quad\varphi(t)\equiv t^{s}\;\;\mbox{by \eqref{f2.8}}
\end{equation}
in \eqref{f3b.4} and observe that $c_{\psi,\varkappa}=1$ for $\psi$ taken. Interpolational inequalities of type \eqref{f3b.5} for Sobolev scales are used in the theory of partial differential operators (see, e.g., \cite[Section~1, Subsection~6]{AgranovichVishik64}).

\begin{proof}[Proof of Theorem $\ref{th3b.1}$.] Let $J$ denote the generating operator for the pair $[H^{\varphi_{0}}_{A},H^{\varphi_{1}}_{A}]$. Recall that $J$ is a positive definite self-adjoint operator in the Hilbert space $H^{\varphi_{0}}_{A}$ and that $\|Ju\|_{\varphi_{0}}=\|u\|_{\varphi_{1}}$ for every $u\in H^{\varphi_{1}}_{A}=\mathrm{Dom}\,J$. According to  \eqref{f3.5bis}, we have $\mathrm{Spec}\,J\subseteq[\varkappa,\infty)$.

Let $E_{t}$, $t\geq\varkappa$, be the resolution of the identity associated with the self-adjoint operator $J$. Choosing  $\tau\geq\varkappa$ and $u\in H^{\varphi_{1}}_{A}$ arbitrarily, we get
\begin{align*}
\|u\|_{\varphi}^{2}&=\|\psi(J)u\|_{\varphi_{0}}^{2}=
\int\limits_{\varkappa}^{\infty}\psi^{2}(t)\,d(E_{t}u,u)_{\varphi_{0}}\\
&\leq c_{\psi,\varkappa}^{2}\,\psi^{2}(\tau)\int\limits_{\varkappa}^{\infty}
\max\biggl\{1,\frac{t^{2}}{\tau^{2}}\biggr\}d(E_{t}u,u)_{\varphi_{0}}\\
&\leq c_{\psi,\varkappa}^{2}\,\psi^{2}(\tau)\int\limits_{\varkappa}^{\infty}
\biggl(1+\frac{t^{2}}{\tau^{2}}\biggr)d(E_{t}u,u)_{\varphi_{0}}\\
&=c_{\psi,\varkappa}^{2}\,\psi^{2}(\tau)\bigl(\|u\|_{\varphi_0}^{2}+
\tau^{-2}\,\|Ju\|_{\varphi_0}^{2}\bigr)\\
&=c_{\psi,\varkappa}^{2}\,\psi^{2}(\tau)\bigl(\|u\|_{\varphi_0}^{2}+
\tau^{-2}\,\|u\|_{\varphi_1}^{2}\bigr).
\end{align*}
Here, we use the interpolation formula $[H^{\varphi_{0}}_{A},H^{\varphi_{1}}_{A}]_{\psi}=H^{\varphi}_{A}$ and inequality \eqref{f3b.2}. Thus, the required inequality \eqref{f3b.3} is   proved.
\end{proof}

Let us consider an application of Theorem~\ref{th3b.1}. We arbitrarily choose $u\in H^{\varphi_{1}}_{A}$ such that $u\neq0$. Put $\tau:=\|u\|_{\varphi_1}/\|u\|_{\varphi_0}$ in \eqref{f3b.3} and note that $\tau\geq\varkappa$ in view of \eqref{f3.2}. We then obtain the interpolational inequality
\begin{equation}\label{f3b.9new}
\|u\|_{\varphi}\leq c_{\psi,\varkappa}\sqrt{2}\,\|u\|_{\varphi_0}\,
\psi\biggl(\frac{\|u\|_{\varphi_1}}{\|u\|_{\varphi_0}}\biggr).
\end{equation}
If the function
\begin{equation}\label{f3b.10}
\chi(\tau):=\psi^{2}(\sqrt{\tau})\quad\mbox{is concave on}\quad
[\varkappa^{2},\infty),
\end{equation}
then \eqref{f3b.9new} holds true without the factor $c_{\psi,\varkappa}\sqrt{2}$. Indeed, choosing $v\in H^{\varphi_{1}}_{A}$ with $\|v\|_{\varphi_0}=1$ arbitrarily, we get
\begin{align*}
\|v\|_{\varphi}^2&=
\int\limits_{\varkappa}^{\infty}\psi^{2}(t)\,d(E_{t}v,v)_{\varphi_{0}}=
\int\limits_{\varkappa}^{\infty}\chi(t^2)\,d(E_{t}v,v)_{\varphi_{0}}\\
&\leq\chi\Biggl(\,\int\limits_{\varkappa}^{\infty}
t^2\,d(E_{t}v,v)_{\varphi_{0}}\Biggr)=
\chi\bigl(\|Jv\|_{\varphi_0}^2\bigr)=
\chi\bigl(\|v\|_{\varphi_1}^2\bigr)=
\psi^2\bigl(\|v\|_{\varphi_1}\bigr)
\end{align*}
due to the Jensen inequality applied to the concave function $\chi$. Here, $E_{t}$ is the same as that in the proof of Theorem~\ref{th3b.1}, and
\begin{equation*}
\int\limits_{\varkappa}^{\infty}d(E_{t}v,v)_{\varphi_{0}}=
\|v\|_{\varphi_0}^2=1.
\end{equation*}
Putting $v:=u/\|u\|_{\varphi_{0}}$ in the inequality $\|v\|_{\varphi}\leq\psi(\|v\|_{\varphi_1})$ just obtained, we conclude that
\begin{equation}\label{f3b.11}
\|u\|_{\varphi}\leq\|u\|_{\varphi_0}\,
\psi\biggl(\frac{\|u\|_{\varphi_1}}{\|u\|_{\varphi_0}}\biggr)
\quad\mbox{under condition \eqref{f3b.10}}.
\end{equation}

This interpolational inequality is equivalent to Variable Hilbert Scale Inequality \cite[Theorem~1, formula (9)]{HeglandAnderssen11} on the supplementary assumption that both functions $\varphi_0$ and $\varphi_1$ are continuous. (Note that the norm $\|\cdot\|_{\varphi}$ used in the cited article \cite{HeglandAnderssen11} means the norm $\|\cdot\|_{\sqrt{\varphi}}$ used by us. Besides, there is no assumption in this article that the function $\varphi_0/\varphi_1$ is bounded.)

In the case of the Hilbert scale $\{H^{s}_{A}:s\in\mathbb{R}\}$, inequality \eqref{f3b.11} becomes
\begin{equation}\label{f3b.12}
\|u\|_{s}\leq\|u\|_{s_0}^{1-\theta}\,\|u\|_{s_1}^{\theta}
\end{equation}
provided that the real numbers $s$, $s_0$, $s_1$, and $\theta$ satisfy \eqref{f3b.6} (we use the power functions \eqref{f3b.7}). The interpolational inequality \eqref{f3b.12} is well known \cite[Section~9, Subsection~1]{KreinPetunin66} and means that the Hilbert scale is a normal scale of spaces.

The interpolational inequalities just considered deal with norms of vectors. Now we focus our attention on interpolational inequalities that involve norms of linear operators acting continuously between appropriate Hilbert spaces $H^{\varphi}_{A}$ and~$G^{\eta}_{Q}$. Here, $G$ (just as $H$) is a separable infinite-dimensional complex Hilbert space, and $Q$ is a counterpart of $A$ for $G$. Namely, $Q$ is a self-adjoint unbounded linear operator in $G$ such that $\mathrm{Spec}\,Q\subseteq[1,\infty)$. We suppose that functions $\eta_{0},\eta_{1},\eta:[1,\infty)\to(0,\infty)$ satisfy analogous conditions to those imposed on $\varphi_{0}$, $\varphi_{1}$, and $\varphi$ at the beginning of this section. Namely, these functions are Borel measurable, and the function $\eta_{0}/\eta_{1}$ is bounded, and
\begin{equation}\label{f3b.13}
\eta(t)=\eta_{0}(t)\,\psi\biggl(\frac{\eta_{1}(t)}{\eta_{0}(t)}\biggr)
\quad\mbox{whenever}\quad t\geq1.
\end{equation}

We suppose that a linear mapping $T$ is given on $H^{\varphi_0}_{A}$ and satisfies the following condition: the restriction of $T$ to the space $H^{\varphi_j}_{A}$ is a bounded operator
\begin{equation}\label{f3b.14}
T:H^{\varphi_j}_{A}\to G^{\eta_j}_{Q}\quad\mbox{for each}\quad j\in\{0,1\}.
\end{equation}
Then the restriction of $T$ to $H^{\varphi}_{A}$ is a bounded operator
\begin{equation}\label{f3b.15}
T:H^{\varphi}_{A}=\bigl[H^{\varphi_0}_{A},H^{\varphi_1}_{A}\bigr]_{\psi}\to \bigl[G^{\eta_0}_{Q},G^{\eta_1}_{Q}\bigr]_{\psi}=G^{\eta}_{Q}
\end{equation}
according to Theorem~\ref{th2.6} and because $\psi$ is an interpolation parameter. Let $\|T\|_{j}$ and $\|T\|$ denote the norms of operators \eqref{f3b.14} and \eqref{f3b.15} respectively. Then
\begin{equation}\label{f3b.16}
\|T\|\leq c\,\max\{\|T\|_0,\|T\|_1\}
\end{equation}
for some number $c>0$ that does not depend on $T$ but may depend on $\psi$ and the spaces $H^{\varphi_j}_{A}$ and $G^{\eta_j}_{Q}$ (see, e.g., \cite[Theorem 2.4.2]{BerghLefstrem76}). This is an interpolational inequality for operator norms, which means that the method of quadratic  interpolation is uniform.

We will consider a more precise interpolational inequality than \eqref{f3b.16}; it involves $\psi$ in some way. Put
\begin{equation*}
\nu:=\min\biggl\{\inf_{t\geq1}\frac{\varphi_{1}(t)}{\varphi_{0}(t)},\,
\inf_{t\geq1}\frac{\eta_{1}(t)}{\eta_{0}(t)}\biggr\}>0,
\end{equation*}
and let $c_{\psi,\nu}$ denote a positive number such that inequality \eqref{f3b.2} holds true with $\nu$ taken instead of $\varkappa$. Without loss of generality we suppose that
\begin{equation}\label{f3b.17}
\frac{\psi(t)}{\psi(\tau)}\leq c_{\psi,\nu}\max\biggl\{1,\frac{t}{\tau}\biggr\}\quad\mbox{for all}\;\; t,\tau>0.
\end{equation}
(Hence, $\psi$ is pseudoconcave on $(0,\infty)$ according to
\cite[Lemma 5.4.3]{BerghLefstrem76}.) We can achieve this by
redefining $\psi$ properly on $(0,\nu)$, e.g., by the formula $\psi(t):=\psi(\nu)\nu^{-1}t$ whenever $0<t<\nu$. This does not change $\varphi$ and $\eta$ in view of \eqref{f2.8} and \eqref{f3b.13}. Let $\widetilde{\psi}$ denote the dilation function for $\psi$, i.e.
\begin{equation}\label{f3b.18}
\widetilde{\psi}(\lambda):=\sup_{t>0}\frac{\psi(\lambda t)}{\psi(t)}
\leq c_{\psi,\nu}\max\{1,\lambda\}
\quad\mbox{whenever}\;\;\lambda>0.
\end{equation}

\begin{theorem}\label{th3b.2}
The following interpolational inequality holds true:
\begin{equation}\label{f3b.19}
\|T\|\leq c_{\psi,\nu}^{2}\sqrt{8}\,\|T\|_{0}\,
\widetilde{\psi}\biggl(\frac{\|T\|_{1}}{\|T\|_{0}}\biggr).
\end{equation}
\end{theorem}

\begin{proof}
It follows directly from Theorem~\ref{th2.6} (namely, from the equalities in \eqref{f3b.15}) and the result by Fan \cite[formula (2.3)]{Fan11} that inequality \eqref{f3b.19} holds true with a certain number $c>0$ written instead of $c_{\psi,\nu}^{2}\sqrt{8}$. Note that Fan \cite{Fan11} denotes the interpolation space $[H_{0},H_{1}]_{\psi}$ by $\overline{\mathcal{H}}_{\chi}$ where $\chi(t)\equiv\psi^{2}(\sqrt{t})$ and that $\psi$ is pseudoconcave on $(0,\infty)$ if and only if so is $\chi$, with the dilation function $\widetilde{\chi}(t)\equiv\widetilde{\psi}^{2}(\sqrt{t})$.
(As in Section~\ref{sec2}, $[H_{0},H_{1}]$ is a regular pair of separable complex Hilbert spaces.) Besides, if $\chi$ is concave on $(0,\infty)$, then $c=\sqrt{2}$. This is a direct consequence of Theorem~\ref{th2.6} and the inequality \cite[formula (2.2)]{Fan11}.
Let us show that we may take $c=c_{\psi,\nu}^{2}\sqrt{8}$ in \eqref{f3b.19} in the general case where $\psi$ is pseudoconcave on $(0,\infty)$.

Considering the function $\chi(t):=\psi^2(\sqrt{t})$ of $t>0$ and using \eqref{f3b.17}, we have
\begin{equation*}
\frac{\chi(t)}{\chi(\tau)}\leq
c_{\psi,\nu}^2\max\biggl\{1,\frac{t}{\tau}\biggr\}
\quad\mbox{for all}\;\; t,\tau>0.
\end{equation*}
It follows from this that
\begin{equation*}
\frac{1}{2c_{\psi,\nu}^2}\,\chi_1(t)\leq
\chi(t)\leq\chi_1(t)\quad\mbox{whenever}\quad t>0,
\end{equation*}
with $\chi_1:(0,\infty)\to(0,\infty)$ being the least concave majorant of $\chi$ (see \cite[p.~91]{Peetre68}). Hence,
\begin{equation}\label{est-psi}
\frac{1}{\sqrt{2}\,c_{\psi,\nu}}\,\psi_{\ast}(t)\leq
\psi(t)\leq\psi_{\ast}(t)\quad\mbox{whenever}\quad t>0,
\end{equation}
where $\psi_{\ast}(t):=\sqrt{\chi_1(t^2)}$ of $t>0$. Since the function
$\psi_{\ast}^{2}(\sqrt{t})\equiv\chi_{1}(t)$ is concave on $(0,\infty)$, we conclude by \cite[formula (2.2)]{Fan11} that
\begin{equation}\label{est-T-ast}
\|T\|_{\ast}\leq \sqrt{2}\,\|T\|_{0}\,
\widetilde{\psi_{\ast}}\biggl(\frac{\|T\|_{1}}{\|T\|_{0}}\biggr).
\end{equation}
Here, $\|T\|_{\ast}$ denotes the norm of the bounded operator
\begin{equation*}
T:\bigl[H^{\varphi_0}_{A},H^{\varphi_1}_{A}\bigr]_{\psi_{\ast}}\to \bigl[G^{\eta_0}_{Q},G^{\eta_1}_{Q}\bigr]_{\psi_{\ast}},
\end{equation*}
and $\widetilde{\psi_{\ast}}$ is the dilation function for $\psi_{\ast}$.
It follows from \eqref{f3b.15} and \eqref{est-psi} that
\begin{equation}\label{f3b.20}
\|T\|\leq \sqrt{2}\,c_{\psi,\nu}\,\|T\|_{\ast}.
\end{equation}
Besides,
\begin{equation}\label{f3b.21}
\widetilde{\psi_{\ast}}(\lambda)=
\sup_{t>0}\frac{\psi_{\ast}(\lambda t)}{\psi_{\ast}(t)}\leq
\sqrt{2}\,c_{\psi,\nu}\widetilde{\psi}(\lambda)
\quad\mbox{whenever}\;\;\lambda>0.
\end{equation}
Now \eqref{est-T-ast}, \eqref{f3b.20}, and \eqref{f3b.21} yield the required inequality \eqref{f3b.19}.
\end{proof}

The inequality \eqref{f3b.19} is more precise than \eqref{f3b.16} in view of \eqref{f3b.18}.

\begin{remark}
If the function $\psi$ is concave on $(0,\infty)$, then $c_{\psi,\nu}=1$ in inequality \eqref{f3b.19}. Besides, we may write $\sqrt{2}$ instead of $c_{\psi,\nu}^{2}\sqrt{8}$ in this inequality provided that the function
$\psi^{2}(\sqrt{t})$ of $t>0$ is concave on $(0,\infty)$, as we have noted in the proof of Theorem~\ref{th3b.2}.
\end{remark}

\section{Applications to function spaces}\label{sec4}

In this section, we will show how the concept of the extended Hilbert scale allows us to introduce and investigate wide classes of Hilbert function (or distribution) spaces related to Sobolev spaces on manifolds.

Let $\Gamma$ be a separable paracompact infinitely smooth Riemannian
manifold without boundary. Consider the separable complex Hilbert space  $H:=L_{2}(\Gamma)$ of all functions $f:\Gamma\to\mathbb{C}$ which are square integrable over $\Gamma$ with respect to the Riemann measure. Let $\Delta_{\Gamma}$ be the Laplace\,--\,Beltrami operator on $\Gamma$; it is defined on the linear manifold $C^{\infty}_{0}(\Gamma)$ of all compactly supported functions $f\in C^{\infty}(\Gamma)$. Assume that the closure of this operator is self-adjoint in $H$, and denote this closure by $\Delta_{\Gamma}$ too. (Specifically, this self-adjointness follows from the completeness of $\Gamma$ under the Riemann metric \cite[p.~140]{Gaffney54}. For incomplete Riemannian manifolds,  sufficient conditions for the self-adjointness are given, e.g., in \cite{BravermanMilatovicShubin02, MilatovicTruc16}). Then the operator $A:=(1-\Delta_{\Gamma})^{1/2}$ is self-adjoint and positive definite with the lower bound $r=1$. Therefore, the separable Hilbert space $H_{A}^{\varphi}$ is defined for every Borel measurable function $\varphi:\nobreak[1,\infty)\to(0,\infty)$; we denote this space by $H_{A}^{\varphi}(\Gamma)$. If $\varphi(t)\equiv t^{s}$ for certain $s\in\mathbb{R}$, then $H_{A}^{\varphi}(\Gamma)$ becomes the Sobolev space $H^{s}(\Gamma)$ of order~$s$. According to Theorem~\ref{th2.1}, the extended $A$-scale consists (up to equivalence of norms) of all Hilbert spaces $H_{A}^{\varphi}(\Gamma)$ where $\varphi\in\mathrm{OR}$. In other words, the class $\{H_{A}^{\varphi}(\Gamma):\varphi\in\mathrm{OR}\}$ consists of all interpolation Hilbert spaces between inner product Sobolev spaces over $\Gamma$. Therefore, it is naturally to call this class the extended Sobolev scale over $\Gamma$.

Now we focus our attention on two important cases where $\Gamma$ is the
Euclidean space $\mathbb{R}^{n}$ and where $\Gamma$ is a compact boundaryless manifold. Generalizing the above consideration, we use some elliptic operators as $A$. We will prove that the extended Hilbert scales generated by these operators consist of some distribution spaces, admit explicit description with the help of the Fourier transform and local charts on $\Gamma$, and do not depend on the choice of the elliptic operators and these charts. Since we consider complex linear spaces formed by functions or distributions, all functions and distributions are supposed to be complex-valued unless otherwise stated.

\subsection{}\label{sec4.1}
In this subsection, we consider the extended Hilbert scale generated by a uniformly elliptic operator on $\mathbb{R}^{n}$, with $n\geq1$. Here, we put $H:=L_{2}(\mathbb{R}^{n})$ and note that the Riemann measure on $\mathbb{R}^{n}$ is the Lebesgue measure. Let $(\cdot,\cdot)_{\mathbb{R}^{n}}$ and $\|\cdot\|_{\mathbb{R}^{n}}$ stand respectively for the inner product and norm in $L_{2}(\mathbb{R}^{n})$.

Following \cite[Section 1.1]{Agranovich94}, we let $\Psi^{m}(\mathbb{R}^{n})$, where $m\in\mathbb{R}$, denote the class of all pseudodifferential operator (PsDO) $\mathcal{A}$ on $\mathbb{R}^{n}$ whose symbols $a$ belong to $C^{\infty}(\mathbb{R}^{2n})$ and satisfy the following condition: for every multi-indices $\alpha,\beta\in\mathbb{Z}_{+}^{n}$ there exists a number $c_{\alpha,\beta}>0$ such that
\begin{equation*}
|\partial^{\alpha}_{x}\partial^{\beta}_{\xi}a(x,\xi)|\leq c_{\alpha,\beta}(1+|\xi|)^{m-|\beta|}\quad\mbox{for all}\quad x,\xi\in\mathbb{R}^{n},
\end{equation*}
with $x$ and $\xi$ being considered respectively as spatial and frequency variables. If $\mathcal{A}\in\Psi^{m}(\mathbb{R}^{n})$, we say that the (formal) order of $\mathcal{A}$ is $m$. A PsDO $\mathcal{A}\in\Psi^{m}(\mathbb{R}^{n})$ is called uniformly elliptic on $\mathbb{R}^{n}$ if there exist positive numbers $c_{1}$ and $c_{2}$ such that $|a(x,\xi)|\geq c_{1}|\xi|^{m}$ whenever $x,\xi\in\mathbb{R}^{n}$ and $|\xi|\geq c_{2}$ (see \cite[Section 3.1~b]{Agranovich94}).

We suppose henceforth in this subsection that
\begin{itemize}
\item[(a)] $\mathcal{A}$ is a PsDO of class $\Psi^{1}(\mathbb{R}^{n})$;
\item[(b)] $\mathcal{A}$ is uniformly elliptic on $\mathbb{R}^{n}$;
\item[(c)] the inequality $(\mathcal{A}w,w)_{\mathbb{R}^{n}}\geq\|w\|_{\mathbb{R}^{n}}^{2}$ holds true for every $w\in C^{\infty}_{0}(\mathbb{R}^{n})$.
\end{itemize}
Here, as usual, $C^{\infty}_{0}(\mathbb{R}^{n})$ denotes the set of all compactly supported functions $u\in C^{\infty}(\mathbb{R}^{n})$.

Consider the mapping
\begin{equation}\label{f4.1}
w\mapsto\mathcal{A}w,\quad\mbox{with}\quad w\in C^{\infty}_{0}(\mathbb{R}^{n}),
\end{equation}
as an unbounded linear operator in the Hilbert space $H=L_{2}(\mathbb{R}^{n})$. This operator is closable because the PsDO $\mathcal{A}\in\Psi^{1}(\mathbb{R}^{n})$ acts continuously from $L_{2}(\mathbb{R}^{n})$ to $H^{-1}(\mathbb{R}^{n})$ (see \cite[Theorem 1.1.2]{Agranovich94}). Here and below, $H^{s}(\mathbb{R}^{n})$ denotes the inner product Sobolev space of order $s\in\mathbb{R}$ over~$\mathbb{R}^{n}$.

Let $A$ denote the closure of the operator \eqref{f4.1} in $L_{2}(\mathbb{R}^{n})$. It follows from (a) and (b) that $\mathrm{Dom}\,A=H^{1}(\mathbb{R}^{n})$ (see \cite[Sections 2.3~c and 3.1~b]{Agranovich94}). Owing to (c), the operator $A$ is positive definite with the lower bound $r=1$. This implies due to (b) that $A$ is self-adjoint \cite[Sections 2.3~d and 3.1~b]{Agranovich94}, with $\mathrm{Spec}\,A\subseteq[1,\infty)$.

Thus, $A$ is the operator considered in Section~\ref{sec2}. Therefore, the separable Hilbert space $H_{A}^{\varphi}$ is defined for every Borel measurable function $\varphi:\nobreak[1,\infty)\to(0,\infty)$. We denote this space by $H_{A}^{\varphi}(\mathbb{R}^{n})$. An important example of $A$ is the operator $(1-\Delta)^{1/2}$, with $\Delta$ denoting the Laplace operator in $\mathbb{R}^{n}$. In this case, the PsDO $\mathcal{A}$ has the symbol $a(x,\xi)\equiv(1+|\xi|^{2})^{1/2}$.

If $\varphi(t)\equiv t^{s}$ for certain $s\in\mathbb{R}$, then it is possible to show that the space $H_{A}^{s}(\mathbb{R}^{n}):=H_{A}^{\varphi}(\mathbb{R}^{n})$ coincides with the Sobolev space $H^{s}(\mathbb{R}^{n})$ up to equivalence norms. Thus, the $A$-scale $\{H_{A}^{s}(\mathbb{R}^{n}):s\in\mathbb{R}\}$ is the Sobolev Hilbert scale. Let us show that the extended $A$-scale consists of some generalized Sobolev spaces, namely the spaces $H^{\varphi}(\mathbb{R}^{n})$ with $\varphi\in\mathrm{OR}$.

Let $\varphi\in\mathrm{OR}$. By definition, the linear space $H^{\varphi}(\mathbb{R}^{n})$ consists of all distributions $w\in\mathcal{S}'(\mathbb{R}^{n})$ that their Fourier transform $\widehat{w}$ is locally Lebesgue integrable over $\mathbb{R}^{n}$ and satisfies the condition
$$
\int\limits_{\mathbb{R}^{n}}\varphi^{2}(\langle\xi\rangle)\,
|\widehat{w}(\xi)|^{2}\,d\xi<\infty.
$$
Here, as usual, $\mathcal{S}'(\mathbb{R}^{n})$ is the linear topological space of all tempered distributions on $\mathbb{R}^{n}$, and
$\langle\xi\rangle:=(1+|\xi|^{2})^{1/2}$ is the smoothed modulus of $\xi\in\mathbb{R}^{n}$. The space $H^{\varphi}(\mathbb{R}^{n})$ is
endowed with the inner product
$$
(w_{1},w_{2})_{\varphi,\mathbb{R}^{n}}:=
\int\limits_{\mathbb{R}^{n}}\varphi^{2}(\langle\xi\rangle)\,
\widehat{w_{1}}(\xi)\,\overline{\widehat{w_{2}}(\xi)}\,d\xi
$$
and the corresponding norm $\|w\|_{\varphi,\mathbb{R}^{n}}:=
(w,w)_{\varphi,\mathbb{R}^{n}}^{1/2}$. This space is complete and separable with respect to this norm and is embedded continuously in $\mathcal{S}'(\mathbb{R}^{n})$; the set  $C^{\infty}_{0}(\mathbb{R}^{n})$ is dense in $H^{\varphi}(\mathbb{R}^{n})$ \cite[Theorem 2.2.1]{Hermander63}.

If $\varphi(t)\equiv t^{s}$ for some $s\in\mathbb{R}$, then $H^{\varphi}(\mathbb{R}^{n})$ becomes the Sobolev space $H^{s}(\mathbb{R}^{n})$. The Hilbert space $H^{\varphi}(\mathbb{R}^{n})$ is an isotropic case of the spaces introduced and investigated by Malgrange \cite{Malgrange57}, H\"ormander \cite[Sec. 2.2]{Hermander63} (see also \cite[Section 10.1]{Hermander83}), and Volevich and Paneah \cite[Section~2]{VolevichPaneah65}. These spaces are generalizations of Sobolev spaces to the case where a general enough function of frequency variables serves as an order of distribution spaces.

\begin{theorem}\label{th4.1}
Let $\varphi\in\mathrm{OR}$. Then the spaces $H^{\varphi}_{A}(\mathbb{R}^{n})$ and $H^{\varphi}(\mathbb{R}^{n})$ coincide as completions of $C^{\infty}_{0}(\mathbb{R}^{n})$ with respect to equivalent norms.
\end{theorem}


It is worth-wile to note that the norm of $w\in C^{\infty}_{0}(\mathbb{R}^{n})$ in $H^{\varphi}_{A}(\mathbb{R}^{n})$ is $\|\varphi(A)w\|_{\mathbb{R}^{n}}$ because
\begin{equation*}
\mathrm{Dom}\,\varphi(A)\supset\mathrm{Dom}\,A^{s_1}=
H^{s_1}_{A}(\mathbb{R}^{n})\supset C^{\infty}_{0}(\mathbb{R}^{n});
\end{equation*}
here $s_1$ is a positive number that satisfies \eqref{f2.3}.

According to Theorem~\ref{th4.1}, the space $H^{\varphi}_{A}(\mathbb{R}^{n})$, with $\varphi\in\mathrm{OR}$, does not depend on~$A$ up to equivalence of norms. Theorems \ref{th2.1} and \ref{th4.1} yield the following explicit description of the extended Hilbert scale generated by the considered operator $A$:

\begin{corollary}\label{cor4.2}
The extended $A$-scale consists (up to equivalence of norms) of all the spaces $H^{\varphi}(\mathbb{R}^{n})$ with $\varphi\in\mathrm{OR}$.
\end{corollary}

Thus, the class $\{H^{\varphi}(\mathbb{R}^{n}):\varphi\in\mathrm{OR}\}$ consists (up to equivalence if norms) of all Hilbert spaces each of which is an interpolation space between some Sobolev spaces $H^{s_0}(\mathbb{R}^{n})$ and $H^{s_1}(\mathbb{R}^{n})$ with $s_0<s_1$.
As we have noted, this class is called the extended Sobolev scale over~$\mathbb{R}^{n}$.  It was considered in \cite{MikhailetsMurach13UMJ3, MikhailetsMurach15ResMath1} and \cite[Section~2.4.2]{MikhailetsMurach14}.

\begin{proof}[Proof of Theorem $\ref{th4.1}$.]
Choose an integer $k\gg1$ such that $-k<\sigma_{0}(\varphi)$ and $k>\sigma_{1}(\varphi)$, and define the interpolation parameter $\psi$ by formula \eqref{f2.6} in which $s_{0}:=-k$ and $s_{1}:=k$. According to Theorem~\ref{th2.5}, we have the equality
\begin{equation}\label{f4.2}
H^{\varphi}_{A}(\mathbb{R}^{n})=
\bigl[H^{-k}_{A}(\mathbb{R}^{n}),H^{k}_{A}(\mathbb{R}^{n})\bigr]_{\psi}.
\end{equation}

Note that each space $H^{\pm k}_{A}(\mathbb{R}^{n})$ coincides with $H^{\pm k}(\mathbb{R}^{n})$ up to equivalence of norms. Indeed, $A$ sets an isomorphism between $H^{1}(\mathbb{R}^{n})$ and $L_{2}(\mathbb{R}^{n})$ because $\mathrm{Dom}\,A=H^{1}(\mathbb{R}^{n})$ and $0\notin\mathrm{Spec}\,A$. Besides, since the PsDO $\mathcal{A}$ of the first order is uniformly elliptic on $\mathbb{R}^{n}$, the operator $A$ has the following lifting property: if $u\in H^{1}(\mathbb{R}^{n})$ and if $Au\in H^{s-1}(\mathbb{R}^{n})$ for some $s>1$, then $u\in H^{s}(\mathbb{R}^{n})$ (see \cite[Sections 1.8 and 3.1~b]{Agranovich94}). Hence, $A^{k}$ sets an isomorphism between $H^{k}(\mathbb{R}^{n})$ and $L_{2}(\mathbb{R}^{n})$. Thus, $H^{k}_{A}(\mathbb{R}^{n})=H^{k}(\mathbb{R}^{n})$ up to equivalence of norms. Passing here to dual spaces with respect to $L_{2}(\mathbb{R}^{n})$, we conclude that $H^{-k}_{A}(\mathbb{R}^{n})=H^{-k}(\mathbb{R}^{n})$ up to equivalence of norms (see \cite[Section~9, Subsection~1]{KreinPetunin66}).

Thus, it follows from \eqref{f4.2} that
\begin{equation}\label{f4.3}
H^{\varphi}_{A}(\mathbb{R}^{n})=
\bigl[H^{-k}(\mathbb{R}^{n}),H^{k}(\mathbb{R}^{n})\bigr]_{\psi}
\end{equation}
up to equivalence of norms. Indeed, since the identity mapping sets the isomorphisms
\begin{equation*}
I:H^{\mp k}_{A}(\mathbb{R}^{n})\leftrightarrow H^{\mp k}(\mathbb{R}^{n})
\end{equation*}
and since $\psi$ is an interpolation parameter, the identity mapping realizes the isomorphism
\begin{equation*}
I:\bigl[H^{-k}_{A}(\mathbb{R}^{n}),H^{k}_{A}(\mathbb{R}^{n})\bigr]_{\psi}
\leftrightarrow
\bigl[H^{-k}(\mathbb{R}^{n}),H^{k}(\mathbb{R}^{n})\bigr]_{\psi}.
\end{equation*}
This yields \eqref{f4.3} due to \eqref{f4.2}.

Thus, Theorem~\ref{th4.1} follows directly from \eqref{f4.3} and
\begin{equation}\label{f4.4}
H^{\varphi}(\mathbb{R}^{n})=
\bigl[H^{-k}(\mathbb{R}^{n}),H^{k}(\mathbb{R}^{n})\bigr]_{\psi}.
\end{equation}
The letter equality is proved in \cite[Theorem~2.19]{MikhailetsMurach14}. Besides, \eqref{f4.4} is a special case of \eqref{f4.3} because $H^{\varphi}(\mathbb{R}^{n})=H_{A}^{\varphi}(\mathbb{R}^{n})$ if $A=(1-\Delta)^{1/2}$. Indeed, since the Fourier transform reduces
$A=(1-\Delta)^{1/2}$ to the operator of multiplication by $\langle\xi\rangle$, we conclude that $\mathrm{Dom}\,\varphi(A)\subseteq H^{\varphi}(\mathbb{R}^{n})$ and $\|\varphi(A)u\|_{\mathbb{R}^{n}}=\|u\|_{\varphi,\mathbb{R}^{n}}$ for every $u\in\mathrm{Dom}\,\varphi(A)$. Hence, $H_{A}^{\varphi}(\mathbb{R}^{n})$ is a subspace of $H^{\varphi}(\mathbb{R}^{n})$. But $C^{\infty}_{0}(\mathbb{R}^{n})$ and then $H_{A}^{\varphi}(\mathbb{R}^{n})$ are dense in $H^{\varphi}(\mathbb{R}^{n})$. Thus, $H_{A}^{\varphi}(\mathbb{R}^{n})$ coincides with $H^{\varphi}(\mathbb{R}^{n})$ if $A=(1-\Delta)^{1/2}$.
\end{proof}

Ending this subsection, we note the following: in contrast to the spaces $H^{s}(\mathbb{R}^{n})$, the inner product Sobolev spaces
\begin{equation*}
H^{s}(\Omega):=\bigl\{w\!\upharpoonright\!\Omega:w\in H^{s}(\mathbb{R}^{n})\bigr\},\quad\mbox{with}\quad s\in\mathbb{R},
\end{equation*}
over a domain $\Omega\subset\mathbb{R}^{n}$ do not form a Hilbert scale even if $\Omega$ is a bounded domain with infinitely smooth boundary
and if we restrict ourselves to the spaces of order $s\geq0$ \cite[Corollary~2.3]{Neubauer88}. An explicit description of all Hilbert spaces that are interpolation ones for an arbitrary chosen pair of Sobolev spaces $H^{s_{0}}(\Omega)$ and $H^{s_{1}}(\Omega)$, where $-\infty<s_{0}<s_{1}<\infty$, is given in \cite[Theorem~2.4]{MikhailetsMurach15ResMath1} provided that $\Omega$ is a bounded domain with Lipschitz boundary. These interpolation spaces are (up to equivalence of norms) the generalized Sobolev spaces
\begin{gather*}
H^{\varphi}(\Omega):=\bigl\{v:=w\!\upharpoonright\!\Omega:w\in H^{\varphi}(\mathbb{R}^{n})\bigr\}
\end{gather*}
such that $\varphi$ belongs to $\mathrm{OR}$ and satisfies~\eqref{f2.3}, the Hilbert norm in $H^{\varphi}(\Omega)$ being naturally defined by the formula
\begin{equation*}
\|v\|_{\varphi,\Omega}:=\inf\bigl\{\|w\|_{\varphi,\mathbb{R}^{n}}:
w\in H^{\varphi}(\mathbb{R}^{n}),v=w\!\upharpoonright\!\Omega\bigr\}.
\end{equation*}

\subsection{}\label{sec4.2}
Here, we consider the extended Hilbert scale generated by an elliptic operator given on a closed manifold. Let $\Gamma$ be an arbitrary closed (i.e. compact and boundaryless) infinitely smooth
manifold of dimension $n\geq1$. We suppose that a certain positive $C^{\infty}$-density $dx$ is given on~$\Gamma$. We put $H:=L_{2}(\Gamma)$, where $L_{2}(\Gamma)$ is the complex Hilbert space of all square integrable functions over $\Gamma$ with respect to the measure induced by this density. Let $(\cdot,\cdot)_{\Gamma}$ and $\|\cdot\|_{\Gamma}$ stand respectively for the inner product and norm in $L_{2}(\Gamma)$.

Following \cite[Section 2.1]{Agranovich94}, we let $\Psi^{m}(\Gamma)$, where $m\in\mathbb{R}$, denote the class of all PsDOs on $\Gamma$ whose
representations in every local chart on $\Gamma$ belong to $\Psi^{m}(\mathbb{R}^{n})$. If $\mathcal{A}\in\Psi^{m}(\Gamma)$, we say that the (formal) order of $\mathcal{A}$ is $m$. A~PsDO $\mathcal{A}\in\Psi^{m}(\Gamma)$ is called elliptic on $\Gamma$ if for every point $x_{0}\in\Gamma$ there exist positive numbers $c_{1}$ and $c_{2}$ such that $|a_{x_{0}}(x,\xi)|\geq c_{1}|\xi|^{m}$ whenever $x\in U(x_{0})$ and $\xi\in\mathbb{R}^{n}$ and $|\xi|\geq\nobreak c_{2}$, with $a_{x_{0}}(x,\xi)$ be the local symbol of $\mathcal{A}$ corresponding to a certain coordinate neighbourhood $U(x_{0})$ of $x_{0}$ (see \cite[Section 3.1 b]{Agranovich94}).

We suppose in this subsection that
\begin{itemize}
\item[(a)] $\mathcal{A}$ is a PsDO of  class $\Psi^{1}(\Gamma)$;
\item[(b)] $\mathcal{A}$ is elliptic on $\Gamma$;
\item[(c)] the inequality $(\mathcal{A}f,f)_{\Gamma}\geq\|f\|_{\Gamma}^{2}$ holds true for every $f\in C^{\infty}(\Gamma)$.
\end{itemize}

Let $A$ denote the closure, in $H=L_{2}(\Gamma)$, of the linear operator $f\mapsto\mathcal{A}f$, with $f\in\nobreak C^{\infty}(\Gamma)$.
Note that this operator is closable in $L_{2}(\Gamma)$ because the PsDO $\mathcal{A}\in\nobreak\Psi^{1}(\Gamma)$ acts continuously from $L_{2}(\Gamma)$ to $H^{-1}(\Gamma)$ \cite[Theorem 2.1.2]{Agranovich94}. Here and below, $H^{s}(\Gamma)$ stands for the inner product Sobolev space of order $s\in\mathbb{R}$ over $\Gamma$. It follows from (a)--(c) that the operator $A$ is positive definite and self-adjoint in $L_{2}(\Gamma)$, with $\mathrm{Dom}\,A=H^{1}(\Gamma)$ and $\mathrm{Spec}\,A\subseteq[1,\infty)$
(see \cite[Sections 2.3 c, d and 3.1~b]{Agranovich94}).

Thus, $A$ is the operator considered in Section~\ref{sec2}, and the separable Hilbert space $H_{A}^{\varphi}$ is defined for every Borel measurable function $\varphi:\nobreak[1,\infty)\to(0,\infty)$. We denote this space by $H_{A}^{\varphi}(\Gamma)$. An important example of $A$ is the operator $(1-\Delta_{\Gamma})^{1/2}$, where $\Gamma$ is endowed with a Riemann metric (then the density $dx$ is induced by this metric).

If $\varphi(t)\equiv t^{s}$ for some $s\in\mathbb{R}$, then the space $H_{A}^{s}(\Gamma):=H_{A}^{\varphi}(\Gamma)$ coincides with the Sobolev space $H^{s}(\Gamma)$ up to equivalence norms \cite[Corollary 5.3.2]{Agranovich94}. Thus, the $A$-scale $\{H_{A}^{s}(\Gamma):s\in\nobreak\mathbb{R}\}$ is the Sobolev Hilbert scale over~$\Gamma$. We will prove that the extended $A$-scale consists of the generalized Sobolev spaces $H^{\varphi}(\Gamma)$ with $\varphi\in\mathrm{OR}$. Let us give their definition with the help of local charts on~$\Gamma$.

We arbitrarily choose a finite atlas from the $C^{\infty}$-structure on $\Gamma$; let this atlas be formed by $\varkappa$ local charts $\pi_j: \mathbb{R}^{n}\leftrightarrow \Gamma_{j}$, with $j=1,\ldots,\varkappa$.
Here, the open sets $\Gamma_{1},\ldots,\Gamma_{\varkappa}$ form a covering of $\Gamma$. We also arbitrarily choose functions $\chi_j\in C^{\infty}(\Gamma)$, with $j=1,\ldots,\varkappa$, that form a partition of unity on $\Gamma$ such that $\mathrm{supp}\,\chi_j\subset \Gamma_j$.

Let $\varphi\in\mathrm{OR}$. By definition, the linear space $H^{\varphi}(\Gamma)$ is the completion of the linear manifold $C^{\infty}(\Gamma)$ with respect to the inner product
\begin{equation}\label{f4.5}
(f_{1},f_{2})_{\varphi,\Gamma}:=
\sum_{j=1}^{\varkappa}\,((\chi_{j}f_{1})\circ\pi_{j},
(\chi_{j}f_{2})\circ\pi_{j})_{\varphi,\mathbb{R}^{n}}
\end{equation}
of functions $f_{1},f_{2}\in C^{\infty}(\Gamma)$. Thus, $H^{\varphi}(\Gamma)$ is a Hilbert space. Let $\|\cdot\|_{\varphi,\Gamma}$ denote the norm induced by the inner product \eqref{f4.5}. If $\varphi(t)\equiv t^{s}$ for certain $s\in\mathbb{R}$, then $H^{\varphi}(\Gamma)$ becomes the Sobolev space $H^{s}(\Gamma)$.

\begin{theorem}\label{th4.3}
Let $\varphi\in\mathrm{OR}$. Then the spaces $H^{\varphi}_{A}(\Gamma)$ and $H^{\varphi}(\Gamma)$ coincide as completions of $C^{\infty}(\Gamma)$ with respect to equivalent norms.
\end{theorem}

Note that the norm of $f\in C^{\infty}(\Gamma)$ in $H^{\varphi}_{A}(\Gamma)$ is $\|\varphi(A)f\|_{\Gamma}$ because
\begin{equation*}
\mathrm{Dom}\,\varphi(A)\supset\mathrm{Dom}\,A^{s_1}
\supset C^{\infty}(\Gamma);
\end{equation*}
here $s_1$ is a positive integer that satisfies \eqref{f2.3}.

Theorem \ref{th4.3} specifically entails

\begin{corollary}\label{cor4.4}
Let $\varphi\in\mathrm{OR}$. Then the space $H^{\varphi}_{A}(\Gamma)$ does not depend on~$A$ up to equivalence of norms. Besides, the space $H^{\varphi}(\Gamma)$ does not depend (up to equivalence of norms) on our choice of the atlas and partition of unity on $\Gamma$.
\end{corollary}

Owing to Theorems \ref{th2.1} and \ref{th4.3}, we obtain the following explicit description of the extended Hilbert scale generated by the considered PsDO $A$ on $\Gamma$:

\begin{corollary}\label{cor4.5}
The extended $A$-scale consists (up to equivalence if norms) of all the spaces $H^{\varphi}(\Gamma)$ with $\varphi\in\mathrm{OR}$.
\end{corollary}

Thus, the class $\{H^{\varphi}(\Gamma):\varphi\in\mathrm{OR}\}$ consists (up to equivalence if norms) of all Hilbert spaces each of which is an interpolation space between some Sobolev inner-product spaces $H^{s_0}(\Gamma)$ and $H^{s_1}(\Gamma)$ with $s_0<s_1$.
This class is called the extended Sobolev scale over~$\Gamma$.

\begin{theorem}\label{th4.6}
Suppose that the manifold $\Gamma$ is endowed with a Riemann metric, and let $\varphi\in\mathrm{OR}$. Then the space $H^{\varphi}(\Gamma)$ admits the following three equivalent definitions in the sense that they introduce the same Hilbert space up to equivalence of norms:
\begin{itemize}
\item[(i)] \textbf{Operational definition.} The Hilbert space $H^{\varphi}(\Gamma)$ is the completion of $C^{\infty}(\Gamma)$ with respect to the norm $\|\varphi((1-\Delta_{\Gamma})^{1/2})f\|_{\Gamma}$ of $f\in C^{\infty}(\Gamma)$.
\item[(ii)] \textbf{Local definition.} The Hilbert space $H^{\varphi}(\Gamma)$ consists of all distributions $f\in\mathcal{D}'(\Gamma)$ such that $(\chi_{j}f)\circ\pi_{j}\in
H^{\varphi}(\mathbb{R}^{n})$ for every $j\in\{1,\ldots,\varkappa\}$ and is endowed with the inner product \eqref{f4.5} of distributions $f_{1},f_{2}\in H^{\varphi}(\Gamma)$.
\item[(iii)] \textbf{Interpolational definition.} Let integers $s_{0}$ and $s_{1}$ satisfy the conditions $s_{0}<\sigma_{0}(\varphi)$ and $s_{1}>\sigma_{1}(\varphi)$, and let $\psi$ be the interpolation parameter defined by \eqref{f2.6}. Then
\begin{equation*}
H^{\varphi}(\Gamma):=
\bigl[H^{s_{0}}(\Gamma),H^{s_{1}}(\Gamma)\bigr]_{\psi}.
\end{equation*}
\end{itemize}
\end{theorem}

Here, as usual, $\mathcal{D}'(\Gamma)$ is the linear topological space of all distributions on $\Gamma$, and $(\chi_{j}f)\circ\pi_{j}$ stands for the representation of the distribution $\chi_{j}f\in\mathcal{D}'(\Gamma)$ in the local chart $\pi_{j}$. We naturally interpret $\mathcal{D}'(\Gamma)$ as the dual of $C^{\infty}(\Gamma)$ with respect to the extension of the inner product in $L_{2}(\Gamma)$. This extension is denoted by $(\cdot,\cdot)_{\Gamma}$ as well. Of course, the $C^{\infty}$-density $dx$ is now induced by the Riemann metric.

\begin{remark}\label{rem4.7}
It follows directly from Theorem~\ref{th4.3} that in the operational definition we may change $(1-\Delta_{\Gamma})^{1/2}$ for the more general PsDO $A$ considered in this subsection.
\end{remark}

Let us prove Theorems \ref{th4.3} and \ref{th4.6}.

\begin{proof}[Proof of Theorem $\ref{th4.3}$.]
Let the integer $k\gg1$ and the interpolation parameter $\psi$ be the same as those in the proof of Theorem~$\ref{th4.1}$. Then
\begin{equation}\label{f4.7}
H^{\varphi}_{A}(\Gamma)=
\bigl[H^{-k}_{A}(\Gamma),H^{k}_{A}(\Gamma)\bigr]_{\psi}
\end{equation}
due to Theorem~\ref{th2.5}. Here, $H^{\pm k}_{A}(\Gamma)=H^{\pm k}(\Gamma)$ up to equivalence of norms, which is demonstrated in the same way as that in the proof of Theorem~$\ref{th4.1}$. Therefore, \eqref{f4.7} implies that
\begin{equation}\label{f4.8}
H^{\varphi}_{A}(\Gamma)=
\bigl[H^{-k}(\Gamma),H^{k}(\Gamma)\bigr]_{\psi}
\end{equation}
up to equivalence of norms. Hence, we have the dense continuous embedding $H^{k}(\Gamma)\hookrightarrow H^{\varphi}_{A}(\Gamma)$, which entails the density of $C^{\infty}(\Gamma)$ in $H^{\varphi}_{A}(\Gamma)$.

Owing to \eqref{f4.8}, it remains to show that
\begin{equation}\label{f4.9}
\bigl[H^{-k}(\Gamma),H^{k}(\Gamma)\bigr]_{\psi}=H^{\varphi}(\Gamma)
\end{equation}
up to equivalence of norms. We will deduce this formula from \eqref{f4.4}
with the help of certain operators of flattening and sewing of the manifold $\Gamma$.

Let us define the flattening operator by the formula
\begin{equation}\label{f4.10}
T:f\mapsto ((\chi_1 f)\circ\pi_1,\ldots,
(\chi_{\varkappa}f)\circ\pi_{\varkappa})
\;\;\mbox{for every}\;\;f\in C^{\infty}(\Gamma).
\end{equation}
The mapping \eqref{f4.10} extends by continuity to isometric linear operators
\begin{equation}\label{f4.11}
T:H^{\varphi}(\Gamma)\rightarrow (H^{\varphi}(\mathbb{R}^n))^{\varkappa}
\end{equation}
and
\begin{equation}\label{f4.12}
T:H^{\mp k}(\Gamma)\rightarrow (H^{\mp k}(\mathbb{R}^n))^{\varkappa}.
\end{equation}
Since $\psi$ is an interpolation parameter, it follows from the boundedness of the operators \eqref{f4.12} that a restriction of the first operator acts continuously
\begin{equation}\label{f4.13}
T:\bigl[H^{-k}(\Gamma),H^{k}(\Gamma)\bigr]_{\psi}\to
\bigl[(H^{-k}(\mathbb{R}^n))^{\varkappa},
(H^{k}(\mathbb{R}^n))^{\varkappa}\bigr]_{\psi}.
\end{equation}
Here, the target space equals $(H^{\varphi}(\mathbb{R}^n))^{\varkappa}$ due to \eqref{f4.4} and the definition of the interpolation with the parameter $\psi$. Thus, the operator \eqref{f4.13} acts continuously
\begin{equation}\label{f4.14}
T:\bigl[H^{-k}(\Gamma),H^{k}(\Gamma)\bigr]_{\psi}\to
\bigl(H^{\varphi}(\mathbb{R}^n)\bigr)^{\varkappa}.
\end{equation}

We define the sewing operator by the formula
\begin{equation}\label{f4.15}
\begin{gathered}
K:\mathbf{w}\mapsto\sum_{j=1}^{\varkappa}
\Theta_j\bigl((\eta_j w_j)\circ\pi_j^{-1}\bigr)\\
\mbox{for every}\quad\mathbf{w}:=(w_{1},\ldots, w_{\varkappa})\in \bigl(C^{\infty}_{0}(\mathbb{R}^n)\bigr)^{\varkappa}.
\end{gathered}
\end{equation}
Here, for each $j\in\{1,\ldots,\varkappa\}$, the function $\eta_j \in C^\infty_0(\mathbb{R}^n)$ is chosen  such that $\eta_j=\nobreak1$ in a neighbourhood of $\pi^{-1}_j(\mathrm{supp}\,\chi_j)$.
Besides, for every function $\omega:\Gamma_{j}\to\mathbb{C}$, we put $(\Theta_j\omega)(x):=\omega(x)$ whenever $x\in\Gamma_{j}$ and put
$(\Theta_j\omega)(x):=0$ whenever $x\in\Gamma\setminus\Gamma_{j}$. Thus,
$K\mathbf{w}\in C^{\infty}(\Gamma)$ for every $\mathbf{w}\in (C^{\infty}_{0}(\mathbb{R}^n))^{\varkappa}$.

The mapping $K$ is left inverse to the flattening operator \eqref{f4.10}. Indeed, given $f\in C^{\infty}(\Gamma)$, we have the following equalities:
\begin{align*}
KTf=&\sum_{j=1}^\varkappa\Theta_j\Bigl(\bigl(\eta_j ((\chi_j f)\circ\pi_j)\bigr)\circ\pi_j^{-1}\Bigr)\\
=&\sum_{j=1}^\varkappa
\Theta_j\bigl((\eta_j\circ\pi_j^{-1})(\chi_j f)\bigr)
=\sum_{j=1}^\varkappa \Theta_j (\chi_j f) =
\sum_{j=1}^\varkappa \chi_j f = f.
\end{align*}
Thus,
\begin{equation}\label{f4.16}
KTf=f\quad\mbox{for every}\quad f\in C^{\infty}(\Gamma).
\end{equation}

There exists a number $c>0$ such that
\begin{equation}\label{f4.17}
\|K\mathbf{w}\|_{\varphi,\Gamma}^{2}\leq c  \sum_{l=1}^{\varkappa}\|w_l\|_{\varphi,\mathbb{R}^n}^{2}
\quad\mbox{whenever}\quad\mathbf{w}\in \bigl(C^{\infty}_{0}(\mathbb{R}^n)\bigr)^{\varkappa}.
\end{equation}
Indeed,
\begin{align*}
\| K\mathbf{w} \|_{\varphi,\Gamma}^{2}&=
\sum_{j=1}^{\varkappa}\|(\chi_jK\mathbf{w})\circ\pi_j\|_
{\varphi,\mathbb{R}^n}^2 \\
&= \sum_{j=1}^{\varkappa}\,\Bigl\|\sum_{l=1}^{\varkappa} \bigl(\chi_j\Theta_l((\eta_lw_l)\circ\pi_l^{-1})\bigr)
\circ\pi_j\Bigr\|_{\varphi,\mathbb{R}^n}^2\\
&= \sum_{j=1}^{\varkappa}\,\Bigl\|\sum_{l=1}^{\varkappa} (\eta_{l,j} w_l)\circ\beta_{l,j} \Bigr\|_{\varphi,\mathbb{R}^n}^2\leq
c \sum_{l=1}^{\varkappa}\|w_l\|_{\varphi,\mathbb{R}^n}^2.
\end{align*}
Here, $\eta_{l,j}:=(\chi_j\circ\pi_l)\eta_l\in C^{\infty}_{0}(\mathbb{R}^n)$,
and $\beta_{l,j}:\mathbb{R}^n\rightarrow\mathbb{R}^n$ is a $C^\infty$-diffeomorphism such that $\beta_{l,j}=\pi_l^{-1}\circ\pi_j$ in a neighbourhood of $\mathrm{supp}\,\eta_{l,j}$ and that $\beta_{l,j}(t)= t$ whenever $|t|$ is sufficiently large. The last inequality is a consequence of the fact that the operator of the multiplication by a function from  $C^{\infty}_{0}(\mathbb{R}^n)$ and the operator $v\mapsto v\circ\beta_{l,j}$ of change of variables are bounded on the space $H^\varphi(\mathbb{R}^n)$. These properties of $H^\varphi(\mathbb{R}^n)$ follow by \eqref{f4.4} from their known analogs for the Sobolev spaces $H^{\mp k}(\mathbb{R}^n)$.

According to \eqref{f4.17}, the mapping \eqref{f4.15} extends by continuity to a bounded linear operator
\begin{equation}\label{f4.18}
K:(H^{\varphi}(\mathbb{R}^n))^{\varkappa}\rightarrow H^{\varphi}(\Gamma)
\end{equation}
and, specifically, to bounded linear operators
\begin{equation}\label{f4.19}
K:(H^{\mp k}(\mathbb{R}^n))^{\varkappa}\rightarrow H^{\mp k}(\Gamma).
\end{equation}
Hence, a restriction of the first operator \eqref{f4.19} acts continuously
\begin{equation}\label{f4.19b}
K:\bigl(H^{\varphi}(\mathbb{R}^n)\bigr)^{\varkappa}=
\bigl[(H^{-k}(\mathbb{R}^n))^{\varkappa},
(H^{k}(\mathbb{R}^n))^{\varkappa}\bigr]_{\psi}\to
\bigl[H^{-k}(\Gamma),H^{k}(\Gamma)\bigr]_{\psi}
\end{equation}
in view of \eqref{f4.4}.

According to \eqref{f4.11} and \eqref{f4.19b}, we have the bounded operator
\begin{equation*}
KT:H^{\varphi}(\Gamma)\to[H^{-k}(\Gamma),H^{k}(\Gamma)]_{\psi}.
\end{equation*}
Besides, owing to \eqref{f4.14} and \eqref{f4.18}, we get the bounded operator
\begin{equation*}
KT:[H^{-k}(\Gamma),H^{k}(\Gamma)]_{\psi}\to H^{\varphi}(\Gamma).
\end{equation*}
These operators are identical mappings in view of \eqref{f4.16} and the density of $C^{\infty}(\Gamma)$ in their target spaces. Thus, the required equality \eqref{f4.9} holds true up to equivalence of norms.
\end{proof}

\begin{proof}[Proof of Theorem $\ref{th4.6}$.]
Let us prove that the initial definition of $H^{\varphi}(\Gamma)$ as the completion of $C^{\infty}(\Gamma)$ with respect to the inner product \eqref{f4.5} is equivalent to each of definitions (i)\,--\,(iii). The initial definition is tantamount to (i) due to Theorem~\ref{th4.3} in the $A=(1-\Delta_{\Gamma})^{1/2}$ case. Hence, the initial definition is equivalent to (iii) in view of Theorem~\ref{th2.5}.

To prove the equivalence of this definition and (ii), it suffices to show  that $C^{\infty}(\Gamma)$ is dense in the space defined by (ii). We arbitrarily choose a distribution $f\in\mathcal{D}'(\Gamma)$ such that $(\chi_{j}f)\circ\pi_{j}\in
H^{\varphi}(\mathbb{R}^{n})$ for every $j\in\{1,\ldots,\varkappa\}$. Given such $j$, we take a sequence $(w^{r}_{j})_{r=1}^{\infty}\subset C^{\infty}_{0}(\mathbb{R}^{n})$ such that $w^{(r)}_{j}\to(\chi_{j}f)\circ\pi_{j}$ in $H^{\varphi}(\mathbb{R}^{n})$ as $r\to\infty$. Let $T$ and $K$ be the flattening and sewing mappings used in the proof of Theorem~\ref{th4.3}. These mappings are well defined respectively on $\mathcal{D}'(\Gamma)$ and $(\mathcal{S}'(\mathbb{R}^{n}))^{\varkappa}$, with the formulas $KTf=f$ and \eqref{f4.17} being valid whenever $f\in\mathcal{D}'(\Gamma)$ and $w\in(H^{\varphi}(\mathbb{R}^{n}))^{\varkappa}$. Therefore, putting $\mathbf{w}^{(r)}:=(w^{(r)}_{1},\ldots,w^{(r)}_{\varkappa})$,
we conclude that $K\mathbf{w}^{(r)}\in C^{\infty}(\Gamma)$ and that
\begin{align*}
\|K\mathbf{w}^{(r)}-f\|_{\varphi,\Gamma}^{2}&=
\|K(\mathbf{w}^{(r)}-Tf)\|_{\varphi,\Gamma}^{2}\\
&\leq c\sum_{l=1}^{\varkappa}
\|w^{(r)}_{l}-(\chi_{l}f)\circ\pi_{l}\|_{\varphi,\mathbb{R}^n}^{2}
\to0\quad\mbox{as}\quad r\to\infty.
\end{align*}
Thus, $C^{\infty}(\Gamma)$ is dense in the space defined by (ii), and the initial definition is then equivalent to (ii).
\end{proof}

At the end of this subsection, we give a description of the space $H^{\varphi}(\Gamma)$ in terms of sequences induced by the spectral decomposition of the self-adjoint operator $A$. Since this operator is positive definite and since $\mathrm{Dom}\,A=H^{1}(\Gamma)$, its inverse $A^{-1}$ is a compact self-adjoint operator on $L_{2}(\Gamma)$ (recall that $H^{1}(\Gamma)$ is compactly embedded in $L_{2}(\Gamma)$). Hence, the Hilbert space $L_{2}(\Gamma)$ has an orthonormal basis $\mathcal{E}:=(e_{j})_{j=1}^{\infty}$ formed by eigenvectors of $A$.
Let $\lambda_{j}\geq1$ be the corresponding eigenvalue of $A$, i.e. $Ae_j=\lambda_{j}e_j$. We may and will enumerate the eigenvectors $e_{j}$ so that $\lambda_{j}\leq\lambda_{j+1}$ whenever $j\geq1$, with $\lambda_{j}\to\infty$ as $j\to\infty$. Since $\mathcal{A}$ is elliptic on $\Gamma$, each $e_{j}\in C^{\infty}(\Gamma)$. We suppose that
the PsDO $\mathcal{A}$ is classical (i.e. polyhomogeneous); see, e.g., \cite[Definitions 1.5.1 and 2.1.3]{Agranovich94}. Then
\begin{equation}\label{f5.20}
\lambda_{j}\sim\widetilde{c}\,j^{1/n}\quad\mbox{as}\quad j\to\infty,
\end{equation}
where $\widetilde{c}$ is a positive number that does not depend on $j$ \cite[Section 6.1~b]{Agranovich94}. Every distribution $f\in\mathcal{D}'(\Gamma)$ expands into the series
\begin{equation}\label{f5.21}
f=\sum_{j=1}^{\infty}\varkappa_{j}(f)e_j\quad\mbox{in}\;\;\mathcal{D}'(\Gamma);
\end{equation}
here, $\varkappa_{j}(f):=(f,e_j)_{\Gamma}$ is the value of the distribution $f$ at the test function $e_{j}$ \cite[Section 6.1~a]{Agranovich94}.

\begin{theorem}\label{th5.8}
Let $\varphi\in\mathrm{OR}$. Then the space $H^{\varphi}(\Gamma)$ consists of all distributions $f\in\mathcal{D}'(\Gamma)$ such that
\begin{equation}\label{f5.22}
\|f\|_{\varphi,\Gamma,\mathcal{E}}^{2}:=
\sum_{j=1}^{\infty}\varphi^{2}(j^{1/n})|\varkappa_{j}(f)|^{2}<\infty,
\end{equation}
and the norm in $H^{\varphi}(\Gamma)$ is equivalent to the (Hilbert) norm $\|\cdot\|_{\varphi,\Gamma,\mathcal{E}}$. If $f\in H^{\varphi}(\Gamma)$, then the series \eqref{f5.21} converges in $H^{\varphi}(\Gamma)$.
\end{theorem}

\begin{proof}
It follows from \eqref{f5.20} and $\varphi\in\mathrm{OR}$ that there exists a number $c\geq1$ such that
\begin{equation}\label{f5.23}
c^{-1}\varphi(\lambda_{j})\leq\varphi(j^{1/n})\leq c\,\varphi(\lambda_{j})\quad\mbox{whenever}\quad 1\leq j\in\mathbb{Z}.
\end{equation}
Since $\mathrm{Spec}\,A=\{\lambda_{j}:j\geq1\}$, we have
\begin{equation*}
\|\varphi(A)f\|_{\Gamma}^{2}=
\sum_{j=1}^{\infty}\varphi^{2}(\lambda_{j})|\varkappa_{j}(f)|^{2}<\infty
\end{equation*}
for every $f\in\mathrm{Dom}\,\varphi(A)$. Hence, the norm $\|\cdot\|_{\varphi,\Gamma,\mathcal{E}}$ is equivalent to the norm in $H^{\varphi}(\Gamma)$ on $\mathrm{Dom}\,\varphi(A)\supset C^{\infty}(\Gamma)$ (see Theorem~\ref{th4.3}).

If $f\in H^{\varphi}(\Gamma)$, we consider a sequence $(f_k)_{k=1}^{\infty}\subset C^{\infty}(\Gamma)$ such that $f_{k}\rightarrow f$ in $H^{\varphi}(\Gamma)$ as $k\rightarrow\infty$. There exist positive numbers $c_1$ and $c_2$ such that
\begin{equation*}
\sum_{j=1}^{\infty}\varphi^{2}(j^{1/n})|\varkappa_{j}(f_k)|^{2}=
\|f_k\|_{\varphi,\Gamma,\mathcal{E}}^{2}\leq
c_1\|f_k\|_{\varphi,\Gamma}^{2}\leq c_2<\infty
\end{equation*}
for every integer $k\geq1$. Passing here to the limit as $k\to\infty$ and taking $\varkappa_{j}(f_k)\to\varkappa_{j}(f)$ into account, we conclude by  Fatou's lemma that every distribution $f\in H^{\varphi}(\Gamma)$ satisfies \eqref{f5.22}.

Assume now that a distribution $f\in\mathcal{D}'(\Gamma)$ satisfies \eqref{f5.22}, and prove that $f\in H^{\varphi}(\Gamma)$. Owing to our assumption and \eqref{f5.23}, we have the convergent orthogonal series
\begin{equation}\label{f5.26}
\sum_{j=1}^{\infty}\varphi(\lambda_{j})\varkappa_{j}(f)e_{j}=:h
\quad\mbox{in}\quad L_{2}(\Gamma).
\end{equation}
Consider its partial sum
\begin{equation*}
h_k:=\sum_{j=1}^{k}\varphi(\lambda_{j})\varkappa_{j}(f)e_{j}
\end{equation*}
for each $k$, and note that
\begin{equation*}
\varphi^{-1}(A)h_k=\sum_{j=1}^{k}\varkappa_{j}(f)e_{j}\in C^{\infty}(\Gamma).
\end{equation*}
Since $h_k\to h$ in $L_{2}(\Gamma)$ as $k\to\infty$, the sequence $(\varphi^{-1}(A)h_k)_{k=1}^{\infty}$ is Cauchy in $H^{\varphi}_{A}(\Gamma)$. Denoting its limit by $g$, we get
\begin{equation}\label{f5.27}
g=\lim_{k\to\infty}\varphi^{-1}(A)h_k=
\sum_{j=1}^{\infty}\varkappa_{j}(f)e_{j}\quad\mbox{in}\quad H^{\varphi}(\Gamma).
\end{equation}
Hence, $f=g\in H^{\varphi}(\Gamma)$ in view of \eqref{f5.21}.

Thus, a distribution $f\in\mathcal{D}'(\Gamma)$ belongs to $H^{\varphi}(\Gamma)$ if and only if \eqref{f5.22} is satisfied.
Besides, given $f\in H^{\varphi}(\Gamma)$, we have
\begin{align*}
\|f\|_{\varphi,\Gamma}^{2}&=
\lim_{k\to\infty}\|\varphi^{-1}(A)h_k\|_{\varphi,\Gamma}^{2}\asymp
\lim_{k\to\infty}\|h_k\|_{\Gamma}^{2}=\|h\|_{\Gamma}^{2}\\
&=\sum_{j=1}^{\infty}\varphi^{2}(\lambda_{j})|\varkappa_{j}(f)|^{2}\asymp
\|f\|_{\varphi,\Gamma,\mathcal{E}}^{2}
\end{align*}
by \eqref{f5.23}, \eqref{f5.26}, and \eqref{f5.27} where $g=f$ (as usual, the symbol $\asymp$ means equivalence of norms). The last assertion of the theorem is due to~\eqref{f5.27}.
\end{proof}


\begin{remark}\label{rem4.8}
Let $0<m\in\mathbb{R}$. Analogs of Theorems \ref{th4.1} and \ref{th4.3} hold true for PsDOs of order~$m$. Namely, suppose that a PsDO $\mathcal{A}$ belongs to $\Psi^{m}(\mathbb{R}^{n})$ or $\Psi^{m}(\Gamma)$ and satisfies conditions (b) and (c). Let $\varphi\in\mathrm{OR}$, and  put $\varphi_{m}(t):=\varphi(t^{m})$ whenever $t\geq1$ (evidently, $\varphi_{m}\in\mathrm{OR}$). Then the equality of spaces
\begin{equation}\label{f4.20}
H^{\varphi}_{A}(\mathbb{R}^{n}\;\mbox{or}\;\Gamma)=
H^{\varphi_{m}}(\mathbb{R}^{n}\;\mbox{or}\;\Gamma)
\end{equation}
holds in the sense that these spaces coincide as completions of $C^{\infty}_{0}(\mathbb{R}^{n})$ or $C^{\infty}(\Gamma)$ with respect to equivalent norms. This implies that Corollaries \ref{cor4.2}, \ref{cor4.4}, and \ref{cor4.5} remain true in this (more general) case. The proof of \eqref{f4.20} is very similar to the proofs of Theorems \ref{th4.1} and \ref{th4.3}. We only observe that $H^{k}_{A}(V)=\nobreak H^{km}(V)$ for every $k\in\mathbb{Z}$ whenever $V=\mathbb{R}^{n}$ or $V=\Gamma$ because $\mathrm{ord}\,\mathcal{A}=m$, which gives
\begin{equation}\label{f4.21}
H^{\varphi}_{A}(V)=
\bigl[H^{-k}_{A}(V),H^{k}_{A}(V)\bigr]_{\psi}=
\bigl[H^{-km}(V),H^{km}(V)\bigr]_{\psi}=
H^{\varphi_{m}}(V)
\end{equation}
with equivalence of norms; here the integer $k>0$ and the interpolation parameter $\psi$ are the same as those in the proof of Theorem~\ref{th4.1}. The first equality in \eqref{f4.21} is due to Theorem~\ref{th2.5}, whereas the third is a direct consequence of this theorem and Theorems \ref{th4.1} and \ref{th4.3}. Note if the PsDO $\mathcal{A}\in\Psi^{m}(\Gamma)$ is classical, then $A^{1/m}$ is the closure (in $L_{2}(\Gamma)$) of some classical PsDO $\mathcal{A}_{1}\in\Psi^{1}(\Gamma)$ as was established by Seeley  \cite{Seeley67}. Hence,
$$
H^{\varphi}_{A}(\Gamma)=H^{\varphi_{m}}_{A^{1/m}}(\Gamma)=
H^{\varphi_{m}}(\Gamma)
$$
immediately due to Theorem~\ref{th4.3}. Ending this remark, note that Theorem~\ref{th5.8} remains true if the order of the classical PsDO $\mathcal{A}$ is~$m$. It follows from the fact every eigenvector of $A$ is also an eigenvector of $A^{1/m}$.
\end{remark}

\section{Spectral expansions in spaces with two norms}\label{sec6a}

We will obtain some abstract results on the convergence of spectral expansions in a Hilbert space endowed with a second norm. In the next section, we will apply these results (together with results of Section~\ref{sec4}) to the investigation of the convergence of spectral expansions induced by normal elliptic operators.

\subsection{}\label{sec6.1a}
As in Section~\ref{sec2}, $H$ is a separable infinite-dimensional complex Hilbert space. Let $L$ be a normal (specifically, self-adjoint) unbounded linear operator in $H$. Let $E$ be the resolution of the identity (i.e., the spectral measure) generated by $L$, we considering $E$ as an operator-valued function $E=E({\delta})$ of $\delta\in\mathcal{B}(\mathbb{C})$. Here, as usual,  $\mathcal{B}(\mathbb{C})$  denotes the class of all Borel subsets of the complex plane~$\mathbb{C}$. Then
\begin{equation}\label{f6.1}
f=\int\limits_{\mathbb{C}}dEf
\end{equation}
for every $f\in H$. Besides, let $N$ be a normed space. (We use the standard notation $\|\cdot\|_{N}$ for the norm in $N$. As above, $\|\cdot\|$ and $(\cdot,\cdot)$ denote the norm and inner product in~$H$.) Suppose that $N$ and $H$ are embedded algebraically in a certain linear space. We find sufficient conditions for the convergence of the spectral expansion \eqref{f6.1} in the space~$N$. Put $\widetilde{B}_{\lambda}:=\{z\in\mathbb{C}:|z|\leq\lambda\}$ for every number $\lambda>0$.

\begin{definition}\label{def6.1}
Let $f\in H$. We say that the spectral expansion \eqref{f6.1} converges unconditionally in the space $N$ at the vector $f$ if $E(\delta)f\in N$ whenever $\delta\in\mathcal{B}(\mathbb{C})$ and if for an arbitrary number $\varepsilon>0$ there exists a bounded set $\gamma=\gamma(\varepsilon)\in\mathcal{B}(\mathbb{C})$ such that
\begin{equation}\label{f6.2}
\|f-E(\delta)f\|_{N}<\varepsilon\quad\mbox{whenever}\quad \gamma\subseteq\delta\in\mathcal{B}(\mathbb{C}).
\end{equation}
\end{definition}

Note that
\begin{equation*}
E(\delta)f=\int\limits_{\delta}dEf
\end{equation*}
for all $f\in H$ and $\delta\in\mathcal{B}(\mathbb{C})$. If the spectrum of $L$ is countable, say $\mathrm{Spec}\,L=\{z_{j}:1\leq j\in\mathbb{Z}\}$ where $j\neq k\Rightarrow z_{j}\neq z_{k}$, then \eqref{f6.1} becomes
\begin{equation}\label{f6.3}
f=\sum_{j=1}^{\infty}E(\{z_{j}\})f.
\end{equation}
If moreover $|z_{j}|\to\infty$ as $j\to\infty$, Definition~\ref{def6.1} will mean that the series \eqref{f6.3} converges to $f$ in $N$ under an arbitrary permutation of its terms.

Let $I$ stand for the identity operator on $H$, and let $\|\cdot\|_{H\to N}$ and $\|\cdot\|_{H\to H}$ denote the norms of bounded linear operators on the pair of spaces $H$ and $N$ and on the space $H$, respectively.

\begin{theorem}\label{th6.2}
Let $R$ and $S$ be bounded linear operators on (whole) $H$ such that they are commutative with $L$ and that
\begin{equation}\label{f6.4}
\mbox{$R$ is a bounded operator from $H$ to $N$.}
\end{equation}
Then the spectral expansion \eqref{f6.1} converges unconditionally in the space $N$ at every vector $f\in RS(H)$. Moreover, the degree of this convergence admits the estimate
\begin{equation}\label{f6.5}
\|f-E(\delta)f\|_{N}\leq\|R\|_{H\to N}\cdot\|g\|\cdot \|S(I-E(\delta))\|_{H\to H}\cdot r_{g}(\delta)
\end{equation}
for every $\delta\in\mathcal{B}(\mathbb{C})$ and with some decreasing function $r_{g}(\delta)\in[0,1]$ of $\delta\in\mathcal{B}(\mathbb{C})$ such that $r_{g}(\widetilde{B}_{\lambda})\to0$ as $\lambda\to\infty$. Here, $g\in H$ is an arbitrary vector satisfying $f=RSg$, and the function $r_{g}(\delta)$ does not depend on $S$ and~$R$.
\end{theorem}

Note, if $T$ is a bounded linear operator on $H$ and if $M$ is an unbounded linear operator in $H$, the phrase ``$T$ is commutative with $M$'' means that $TMf=MTf$ for every vector $f\in(\mathrm{Dom}\,M)\cap\mathrm{Dom}(MT)$ (see, e.g., \cite[Chapter~IV, \S~3, Section~4]{FunctionalAnalysis72}).

\begin{proof}[Proof of Theorem $\ref{th6.2}$]
Choose a vector $f\in RS(H)\subseteq N\cap H$ arbitrarily. If $f=0$, the conclusion of this theorem will be trivial; we thus suppose that $f\neq0$. Consider a nonzero vector $g\in H$ such that $f=RSg$. Choose a set $\delta\in\mathcal{B}(\mathbb{C})$ arbitrarily. Since the operators $R$ and $S$ are bounded on $H$ and commutative with $L$, they are also commutative with $E(\delta)$. Therefore,
\begin{equation*}
E(\delta)f=E(\delta)(RS)g=(RS)E(\delta)g\in N
\end{equation*}
due to \eqref{f6.4}. Hence,
\begin{equation}\label{f6.8}
\begin{aligned}
\|f-E(\delta)f\|_{N}&
=\|RS(I-E(\delta))g\|_{N}=\|RS(I-E(\delta))^{2}g\|_{N}\\
&\leq\|R\|_{H\to N}\cdot\|S(I-E(\delta))\|_{H\to H}\cdot \|(I-E(\delta))g\|.
\end{aligned}
\end{equation}
Put
\begin{equation}\label{f6.10}
r_{g}(\delta):=\|(I-E(\delta))g\|\cdot\|g\|^{-1};
\end{equation}
then \eqref{f6.8} yields the required estimate \eqref{f6.5}. It follows plainly from \eqref{f6.10} that $r_{g}(\delta)$ viewed as a function of  $\delta\in\mathcal{B}(\mathbb{C})$ is required.
\end{proof}

\begin{remark}\label{rem6.3}
Let $R$ be a bounded operator on $H$. If the norms in $N$ and $H$ are compatible, condition \eqref{f6.4} is equivalent to the inclusion $R(H)\subseteq N$. Indeed, assume that these norms are compatible and that $R(H)\subseteq N$, and show that $R$ satisfies~\eqref{f6.4}. According to the closed graph theorem, the operator $R:H\to\widetilde{N}$ is bounded if and only if it is closed; here, $\widetilde{N}$ is the completion of the normed space~$N$. Therefore, it is enough to prove that this operator is closable. Assume that a sequence $(f_{k})_{k=1}^{\infty}\subset H$ satisfies the following two conditions: $f_{k}\to0$ in $H$ and $Rf_{k}\to h$ in $\widetilde{N}$ for certain $h\in\widetilde{N}$, as $k\to\infty$. Then $Rf_{k}\to0$ in $H$ because $R$ is bounded on $H$. Hence, $h=0$ as the norms in $N$ and $H$ are compatible. Thus, the operator $R:H\to\widetilde{N}$ is closable.
\end{remark}

\begin{remark}\label{rem6.4}
Borel measurable bounded functions of $L$ are important examples of the bounded operators on $H$ commuting with~$L$. If $S=\eta(L)$ for a bounded Borel measurable function $\eta:\mathrm{Spec}\,L\to\mathbb{C}$, the third factor on the right of \eqref{f6.5} will admit the estimate
\begin{equation}
\begin{aligned}
\|S(I-E(\delta))\|_{H\to H}&\leq
\sup\bigl\{|\eta(z)|(1-\chi_{E(\delta)}(z)):z\in\mathrm{Spec}\,L\bigr\}\\
&\leq\sup\bigl\{|\eta(z)|:z\in(\mathrm{Spec}\,L)\setminus\delta\bigr\}.
\end{aligned}
\end{equation}
(As usual, $\chi_{E(\delta)}$ stands for the characteristic function of the set $E(\delta)$.) Hence, if $\eta(z)\to0$ as $|z|\to\infty$, then
\begin{equation*}
\lim_{\lambda\to\infty}\|S(I-E(\widetilde{B}_{\lambda}))\|_{H\to H}=0
\end{equation*}
(as well as the fourth factor $r_{g}(\delta)$ if $\delta=\widetilde{B}_{\lambda}$).
\end{remark}

\subsection{}\label{sec6.1b}
Assume now that the normal operator $L$ has pure point spectrum, i.e. the Hilbert space $H$ has an orthonormal basis $(e_{j})_{j=1}^{\infty}$ formed by some eigenvectors $e_{j}$ of $L$. Unlike the previous part of this subsection, we suppose that $L$ is either unbounded in $H$ or bounded on $H$. Thus,
\begin{equation}\label{f6.11}
f=\sum_{j=1}^{\infty}(f,e_j)e_j
\end{equation}
in $H$ for every $f\in H$. Let $\lambda_{j}$ denote the eigenvalue of $L$ such that $Le_j=\lambda_{j}e_j$.  Note that the expansions \eqref{f6.1} and \eqref{f6.3} become \eqref{f6.11} provided that all the proper subspaces of $L$ are one-dimensional. Let $P_k$ denote the orthoprojector on the linear span of the eigenvectors $e_1,\ldots,e_k$.

\begin{theorem}\label{th6.5}
Let $\omega,\eta:\mathrm{Spec}\,L\to\mathbb{C}\setminus\{0\}$ be Borel measurable bounded functions, and consider the bounded linear operators $R:=\omega(L)$ and $S:=\eta(L)$ on $H$. Assume that $R$ satisfies \eqref{f6.4}. Then  the series \eqref{f6.11} converges unconditionally (i.e. under each permutation of its terms) in the space $N$ at every vector $f\in RS(H)$. Moreover, the degree of this convergence admits the estimate
\begin{equation}\label{f6.12}
\biggl\|f-\sum_{j=1}^{k}(f,e_j)e_j\biggr\|_{N}\leq
\|R\|_{H\to N}\cdot\|g\|\cdot \|S(I-P_k)\|_{H\to H}\cdot r_{g,k}
\end{equation}
for every integer $k\geq1$ and with some decreasing sequence  $(r_{g,k})_{k=1}^{\infty}\subset[0,1]$ that tends to zero and does not depend on $S$ and~$R$. Here, $g:=(RS)^{-1}f\in H$.
\end{theorem}

\begin{proof}
Since $RSe_j=(\omega\eta)(\lambda_j)e_j$ for every integer $j\geq1$ and since $(\omega\eta)(t)\neq0$ for every $t\in\mathrm{Spec}\,L$, we conclude that each $e_j\in N$ in view of hypothesis \eqref{f6.4}. Thus, the left-hand side of \eqref{f6.12} makes sense. Besides, the operator $RS=(\omega\eta)(L)$ is algebraically reversible; hence, the vector $g:=(RS)^{-1}f\in H$ is well defined for every $f\in RS(H)$. We suppose that $f\neq0$ because the conclusion of this theorem is trivial in the  $f=0$ case. Choosing an integer $k\geq1$ arbitrarily, we get
\begin{align*}
(RS)P_{k}g&=RS\sum_{j=1}^{k}(g,e_j)e_j=\sum_{j=1}^{k}(g,e_j)RSe_j=
\sum_{j=1}^{k}(g,e_j)(\omega\eta)(\lambda_j)e_j\\
&=P_{k}(RS)\sum_{j=1}^{\infty}(g,e_j)e_j=P_{k}(RS)g.
\end{align*}
Hence,
\begin{equation}\label{f6.13}
\begin{aligned}
\biggl\|f-\sum_{j=1}^{k}(f,e_j)e_j\biggr\|_{N}&=\|f-P_{k}f\|_{N}=
\|RSg-P_{k}(RS)g\|_{N}\\
&=\|RS(I-P_{k})g\|_{N}=\|RS(I-P_{k})^{2}g\|_{N}\\
&\leq\|R\|_{H\to N}\cdot\|S(I-P_{k})\|_{H\to H}\cdot \|(I-P_{k})g\|.
\end{aligned}
\end{equation}
Putting
\begin{equation}\label{f6.14}
r_{g,k}:=\|(I-P_k)g\|\cdot\|g\|^{-1},
\end{equation}
we see that \eqref{f6.13} yields the required estimate \eqref{f6.12}. It follows plainly from \eqref{f6.14} that the sequence $(r_{g,k})_{k=1}^{\infty}$ is required. Hence, the series \eqref{f6.11} converges in $N$. This convergence is unconditional because the hypotheses of the theorem are invariant with respect to permutations of terms of this series.
\end{proof}

\begin{remark}\label{rem6.6}
The third factor on the right of \eqref{f6.12} admits the estimate
\begin{equation}\label{f6.15}
\|S(I-P_k)\|_{H\to H}\leq\sup_{j\geq k+1}|\eta(\lambda_j)|
\end{equation}
for each integer $k\geq1$. Indeed, since
\begin{equation*}
S(I-P_k)f=\eta(L)\sum_{j=k+1}^{\infty}(f,e_{j})e_{j}=
\sum_{j=k+1}^{\infty}(f,e_{j})\eta(\lambda_j)e_{j}
\end{equation*}
for every $f\in H$ (the convergence holds in $H$), we have
\begin{equation*}
\|S(I-P_k)f\|^{2}=\sum_{j=k+1}^{\infty}|(f,e_{j})\eta(\lambda_j)|^{2}\leq
\bigl(\sup_{j\geq k+1}|\eta(\lambda_j)|\bigr)^{2}\cdot\|f\|^{2},
\end{equation*}
which gives \eqref{f6.15}. Specifically, if $\eta(t)\to0$ as $|t|\to\infty$ and if $|\lambda_j|\to\infty$ as $j\to\infty$, then
\begin{equation*}
\lim_{k\to\infty}\|S(I-P_k)\|_{H\to H}=0
\end{equation*}
(as well as the fourth factor $r_{g,k}$).
\end{remark}

It is worthwhile to note that the hypotheses of Theorem~\ref{th6.5} do not depend on the choice of a basis of~$H$. They hence imply the unconditional convergence of the series \eqref{f6.11} in $N$ at every vector $f\in RS(H)$ for \emph{any} orthonormal basis of $H$ formed by eigenvectors of~$L$. Remark also that Theorem~\ref{th6.5} reinforces the conclusion of Theorem~\ref{th6.2} under the hypotheses of Theorem~\ref{th6.5}. Indeed, owing to Theorem~\ref{th6.2}, the series \eqref{f6.11} converges in $N$ at every $f\in RS(H)$ if its terms corresponding to equal eigenvalues are grouped together and if $|\lambda_{j}|\to\infty$ as $j\to\infty$.

Theorem~\ref{th6.5} contains M.~G.~Krein's theorem \cite{Krein47} according to which the series \eqref{f6.11} converges in $N$ at every $f\in L(H)$ if $L$ is a self-adjoint compact operator in $H$ obeying \eqref{f6.4}. The latter theorem generalizes (to abstract operators) the Hilbert\,--\,Schmidt theorem about the uniform decomposability of sourcewise representable functions with respect to eigenfunctions of a symmetric integral operator. If $L$ is a positive definite self-adjoint operator with discrete spectrum and if $R=L^{-\sigma}$ and $S=L^{-\tau}$ for certain $\sigma,\tau\geq0$ and if $R$ satisfies \eqref{f6.4}, Krasnosel'ski\u{\i} and Pustyl'nik \cite[Theorem~22.1]{KrasnoselskiiZabreikoPustylnikSobolevskii76} proved that the left-hand side of \eqref{f6.12} is $o(\lambda_{k}^{-\tau})$ as $k\to\infty$. This result follows from \eqref{f6.12} in view of \eqref{f6.15}.

\section{Applications to spectral expansions induced by elliptic operators}\label{sec6}

This section is devoted to applications of results of Sections \ref{sec4} and \ref{sec6a} to the investigation of the convergence (in the uniform metric) of spectral expansions induced by uniformly elliptic operators on $\mathbb{R}^{n}$ and by elliptic operators on a closed manifold $\Gamma\in C^{\infty}$. We find explicit criteria of the convergence of these expansions in the normed space $C^{q}$, with $q\geq0$, on the function class $H^{\varphi}$, with $\varphi\in\mathrm{OR}$, and evaluate the degree of this convergence. Besides, we consider applications of the spaces $H^{\varphi}(\Gamma)$ to the investigation of the almost everywhere convergence of the spectral expansions.

\subsection{}\label{sec6.2}
Let $1\leq n\in\mathbb{Z}$ and $0<m\in\mathbb{R}$. We suppose in this subsection that $L$ is a PsDO of class $\Psi^{m}(\mathbb{R}^{n})$ and that $L$ is uniformly elliptic on $\mathbb{R}^{n}$. We may and will consider $L$ as a closed unbounded operator in the Hilbert space $H:=L_{2}(\mathbb{R}^{n})$ with $\mathrm{Dom}\,L=H^{m}(\mathbb{R}^{n})$ (see \cite[Sections 2.3~d and 3.1~b]{Agranovich94}). We also suppose that $L$ is a normal operator in $L_{2}(\mathbb{R}^{n})$. Then $L$ generates a resolution of the identity $E=E({\delta})$, and the spectral expansion \eqref{f6.1} holds for every function $f\in L_{2}(\mathbb{R}^{n})$. Note that the spectrum of $L$ may be uncountable and may not have any eigenfunctions. Hence, the expansion \eqref{f6.1}
may not be represented in the form of the series \eqref{f6.3} or \eqref{f6.11}. For example, if $L=-\Delta$, then the spectrum of $L$ coincides with $[0,\infty)$ and is continuous.

\begin{definition}\label{def6.7}
Let a normed function space $N$ lie in $\mathcal{S}'(\mathbb{R}^{n})$. We say that the expansion \eqref{f6.1} (where $H=L_{2}(\mathbb{R}^{n})$) converges unconditionally in $N$ on a function class $\Upsilon$ if $\Upsilon\subset L_{2}(\mathbb{R}^{n})$ and if this expansion satisfies Definition~\ref{def6.1} for every $f\in\Upsilon$.
\end{definition}

We consider the important case where $N=C^{q}_{\mathrm{b}}(\mathbb{R}^{n})$ for an integer $q\geq0$ and use generalized Sobolev spaces $H^{\varphi}(\mathbb{R}^{n})$ as $\Upsilon$. Here, $C^{q}_{\mathrm{b}}(\mathbb{R}^{n})$ denotes the Banach space of
$q$ times continuously differentiable functions $\nobreak{f:\mathbb{R}^{n}\to\mathbb{C}}$ whose partial derivatives $\partial^{\alpha}f$ are bounded on $\mathbb{R}^{n}$ whenever $|\alpha|\leq q$. As usual, $\alpha=(\alpha_{1},\ldots,\alpha_{n})\in\mathbb{Z}_{+}^{n}$ and $|\alpha|=\alpha_{1}+\cdots+\alpha_{n}$. This space is endowed with the norm
\begin{equation*}
\|f\|_{C,q,\mathbb{R}^{n}}:=\sum_{|\alpha|\leq q}\,
\sup\bigl\{|\partial^{\alpha}f(x)|:x\in\mathbb{R}^{n}\bigr\}.
\end{equation*}

\begin{theorem}\label{th6.8}
Let $0\leq q\in\mathbb{Z}$ and $\varphi\in\mathrm{OR}$. The spectral expansion \eqref{f6.1} converges unconditionally in the normed space $C^{q}_{\mathrm{b}}(\mathbb{R}^{n})$ on the function class $H^{\varphi}(\mathbb{R}^{n})$ if and only if
\begin{equation}\label{f6.16}
\int\limits_{1}^{\infty}\frac{t^{2q+n-1}}{\varphi^2(t)}\,dt<\infty.
\end{equation}
\end{theorem}

\begin{remark}\label{rem6.9}
If we replace the lover limit $1$ in \eqref{f6.16} with an arbitrary number $k>1$, we will obtain an equivalent condition on the function $\varphi\in\mathrm{OR}$. This is due to the fact that every function $\varphi\in\mathrm{OR}$ is bounded together with $1/\varphi$ on each compact interval $[1,k]$ where $k>1$. This follows from property \eqref{f2.3}, in which we put $t=1$.
\end{remark}

The next result allows us to estimate the degree of the convergence stipulated by Theorem~\ref{th6.8}.

\begin{theorem}\label{th6.10}
Let $0\leq q\in\mathbb{Z}$ and $\phi_{1},\phi_{2}\in\mathrm{OR}$. Suppose that $\phi_{1}(t)\to\infty$ as $t\to\infty$ and that
\begin{equation}\label{f6.17}
\int\limits_{1}^{\infty}\frac{t^{2q+n-1}}{\phi_{2}^{2}(t)}\,dt<\infty.
\end{equation}
Consider the function $\varphi:=\phi_{1}\phi_{2}$, which evidently belongs to $\mathrm{OR}$ and satisfies \eqref{f6.16}. Then the degree of the convergence of the spectral expansion \eqref{f6.1} in the normed space $C^{q}_{\mathrm{b}}(\mathbb{R}^{n})$ on the class $H^{\varphi}(\mathbb{R}^{n})$ admits the estimate
\begin{equation}\label{f6.18}
\|f-E(\widetilde{B}_{\lambda})f\|_{C,q,\mathbb{R}^{n}}\leq
c\cdot\|f\|_{\varphi,\mathbb{R}^{n}}\cdot
\sup\bigl\{(\phi_{1}(t))^{-1}:t\geq\langle\lambda\rangle^{1/m}\bigr\}
\cdot\theta_{f}(\lambda)
\end{equation}
for every function $f\in H^{\varphi}(\mathbb{R}^{n})$ and each number $\lambda>0$. Here, $c$ is a certain positive number that does not depend on $f$ and $\lambda$, and $\theta_{f}(\lambda)$ is a decreasing function of $\lambda$ such that $\nobreak{0\leq\theta_{f}(\lambda)\leq1}$ whenever $\lambda>0$ and that $\theta_{f}(\lambda)\to0$ as $\lambda\to\infty$.
\end{theorem}

As to \eqref{f6.18}, recall that $\langle\lambda\rangle:=(1+|\lambda|^{2})^{1/2}$.

\begin{remark}\label{rem6.11}
Suppose that a function $\varphi\in\mathrm{OR}$ satisfies \eqref{f6.16}; then it may be represented in the form $\varphi=\phi_{1}\phi_{2}$ for some functions $\phi_{1},\phi_{2}\in\mathrm{OR}$ subject to the hypotheses of Theorem~\ref{th6.10}. Indeed, considering the function
\begin{equation*}
\eta(t):= \int\limits_t^\infty\frac{\tau^{2q+n-1}}{\varphi^2(\tau)}\,d\tau<\infty
\quad\mbox{of}\quad t\geq1
\end{equation*}
and choosing a number $\varepsilon\in(0,1/2)$, we put $\phi_{1}(t):=\eta^{-\varepsilon}(t)$ and $\phi_{2}(t):=\varphi(t)\eta^{\varepsilon}(t)$ whenever $t\geq1$. Then
$\phi_{1}(t)\to\infty$ as $t\to\infty$, and
\begin{equation*}
\int\limits_1^\infty\frac{t^{2q+n-1}}{\phi_2^2(t)}\,dt=
\int\limits_1^\infty\frac{t^{2q+n-1}}{\varphi^2(t)\eta^{2\varepsilon}(t)}
\,dt=-\int\limits_1^\infty\frac{d\eta(t)}{\eta^{2\varepsilon}(t)}=
\int\limits^{\eta(1)}_0\frac{d\eta}{\eta^{2\varepsilon}}<\infty.
\end{equation*}
To show that $\phi_{1},\phi_{2}\in\mathrm{OR}$, it suffices to prove the inclusion $\eta\in\mathrm{OR}$. Since $\varphi\in\mathrm{OR}$,
there exist numbers $a>1$ and $c\geq1$ such that $c^{-1}\leq\varphi(\lambda\zeta)/\varphi(\zeta)\leq c$ for all $\zeta\geq1$ and $\lambda\in[1,a]$. Assuming $t\geq1$ and $1\leq\lambda\leq a$, we therefore get
\begin{equation*}
\eta(\lambda t)=\int\limits_{\lambda t}^\infty
\frac{\tau^{2q+n-1}}{\varphi^2(\tau)}\,d\tau=
\lambda^{2q+n}\int\limits_{t}^\infty
\frac{\zeta^{2q+n-1}}{\varphi^2(\lambda\zeta)}\,d\zeta\leq
c^{2}\lambda^{2q+n}\int\limits_{t}^\infty
\frac{\zeta^{2q+n-1}}{\varphi^2(\zeta)}\,d\zeta
\leq c^2a^{2q+n}\eta(t)
\end{equation*}
and
\begin{equation*}
\eta(\lambda t)=\lambda^{2q+n}\int\limits_{t}^\infty
\frac{\zeta^{2q+n-1}}{\varphi^2(\lambda\zeta)}\,d\zeta\geq
c^{-2}\lambda^{2q+n}\int\limits_{t}^\infty
\frac{\zeta^{2q+n-1}}{\varphi^2(\zeta)}\,d\zeta\geq c^{-2}\eta(t);
\end{equation*}
i.e. $\eta\in\mathrm{OR}$.
\end{remark}

Before we prove Theorems \ref{th6.8} and \ref{th6.10}, we will illustrate them with three examples. As above, $0\leq q\in\mathbb{Z}$. As in Theorem~\ref{th6.10}, we let $c$ denote a positive number that does not depend on $f$ and $\lambda$.

\begin{example}\label{ex6.2.1}
Let us restrict ourselves to the Sobolev spaces $H^{s}(\mathbb{R}^{n})$, with $s\in\mathbb{R}$. Owing to Theorem~\ref{th6.8}, the spectral expansion \eqref{f6.1} converges unconditionally in $C^{q}_{\mathrm{b}}(\mathbb{R}^{n})$ on the class $H^{s}(\mathbb{R}^{n})$ if and only if $s>q+n/2$. Let $s>q+n/2$, and put $r:=s-q-n/2>0$. If  $0<\varepsilon<r/m$, then the degree of this convergence admits the following estimate:
\begin{equation*}
\|f-E(\widetilde{B}_{\lambda})f\|_{C,q,\mathbb{R}^{n}}\leq
c\,\|f\|_{s,\mathbb{R}^{n}}\langle\lambda\rangle^{\varepsilon-r/m}
\end{equation*}
for all $f\in H^{s}(\mathbb{R}^{n})$ and $\lambda>0$. Here,  $\|\cdot\|_{s,\mathbb{R}^{n}}$ is the norm in $H^{s}(\mathbb{R}^{n})$. This estimate follows from Theorem~\ref{th6.10}, in which we put
$\phi_{1}(t):=t^{r-m\varepsilon}$ and $\phi_{2}(t):=t^{s-r+m\varepsilon}$ for every $t\geq1$. Choosing a number $\varepsilon>0$ arbitrarily and putting
\begin{equation}\label{f6.18b}
\phi_{1}(t):=t^{r}\log^{-\varepsilon-1/2}(1+t)\quad\mbox{and}\quad
\phi_{2}(t):=t^{s-r}\log^{\varepsilon+1/2}(1+t)
\end{equation}
for every $t\geq1$ in this theorem, we obtain the sharper estimate
\begin{equation*}
\|f-E(\widetilde{B}_{\lambda})f\|_{C,q,\mathbb{R}^{n}}\leq
c\,\|f\|_{s,\mathbb{R}^{n}}\langle\lambda\rangle^{-r/m}
\log^{\varepsilon+1/2}(1+\langle\lambda\rangle)
\end{equation*}
for the same $f$ and $\lambda$.
\end{example}

Using the generalized Sobolev spaces $H^{\varphi}(\mathbb{R}^{n})$, with $\varphi\in\mathrm{OR}$, we may establish the unconditional convergence of \eqref{f6.1} in $C^{q}_{\mathrm{b}}(\mathbb{R}^{n})$ at some functions
\begin{equation*}
f\notin H^{q+n/2+}(\mathbb{R}^{n}):=
\bigcup_{s>q+n/2}H^{s}(\mathbb{R}^{n})
\end{equation*}
and evaluate its degree. (Note that this union is narrower than $H^{q+n/2}(\mathbb{R}^{n})$.)

\begin{example}\label{ex6.2.2}
Choosing a number $\varrho>0$ arbitrarily and putting
\begin{equation}\label{f6.19}
\varphi(t):=t^{q+n/2}\log^{\varrho+1/2}(1+t)
\quad\mbox{for every}\quad t\geq1,
\end{equation}
we conclude by Theorem~\ref{th6.8} that the spectral expansion \eqref{f6.1} converges unconditionally in $C^{q}_{\mathrm{b}}(\mathbb{R}^{n})$ on the class $H^{\varphi}(\mathbb{R}^{n})$. This class is evidently broader than $H^{q+n/2+}(\mathbb{R}^{n})$. If $0<\varepsilon<\varrho$, then the degree of this convergence admits the estimate
\begin{equation*}
\|f-E(\widetilde{B}_{\lambda})f\|_{C,q,\mathbb{R}^{n}}\leq
c\,\|f\|_{\varphi,\mathbb{R}^{n}}
\log^{\varepsilon-\varrho}(1+\langle\lambda\rangle)
\end{equation*}
for all $f\in H^{\varphi}(\mathbb{R}^{n})$ and $\lambda>0$. This estimate follows from Theorem~\ref{th6.10}, in which we represent $\varphi$ as the product of the functions
\begin{equation*}
\phi_{1}(t):=\log^{\varrho-\varepsilon}(1+t)\quad\mbox{and}\quad
\phi_{2}(t):=t^{q+n/2}\log^{\varepsilon+1/2}(1+t).
\end{equation*}
\end{example}

Using iterated logarithms, we may obtain weaker sufficient conditions for the unconditional convergence of \eqref{f6.1} in $C^{q}_{\mathrm{b}}(\mathbb{R}^{n})$. The next example involves the double logarithm.

\begin{example}\label{ex6.2.3}
Choose a number $\varrho>0$ arbitrarily, and consider the function
\begin{equation}\label{f6.20}
\varphi(t):=t^{q+n/2}\,(\log(1+t))^{1/2}\,(\log\log(2+t))^{\varrho+1/2}
\quad\mbox{of}\quad t\geq1.
\end{equation}
According to Theorem~\ref{th6.8}, the spectral expansion \eqref{f6.1} converges unconditionally in $C^{q}_{\mathrm{b}}(\mathbb{R}^{n})$ on the class $H^{\varphi}(\mathbb{R}^{n})$. If $0<\varepsilon<\varrho$, then the degree of this convergence admits the estimate
\begin{equation*}
\|f-E(\widetilde{B}_{\lambda})f\|_{C,q,\mathbb{R}^{n}}\leq
c\,\|f\|_{\varphi,\mathbb{R}^{n}}
\bigl(\log\log(2+\langle\lambda\rangle)\bigr)^{\varepsilon-\varrho}
\end{equation*}
for all $f\in H^{\varphi}(\mathbb{R}^{n})$ and $\lambda>0$. The estimate follows from Theorem~\ref{th6.10} provided that we represent $\varphi$ as the product of the functions $\phi_{1}(t):=(\log\log(2+t))^{\varrho-\varepsilon}$ and
\begin{equation*}
\phi_{2}(t):=t^{q+n/2}(\log(1+t))^{1/2}
(\log\log(2+t))^{\varepsilon+1/2}.
\end{equation*}
\end{example}

Let us turn to the proofs of Theorems \ref{th6.8} and \ref{th6.10}. The  proofs are based on the following version of H\"ormander's embedding theorem \cite[Theorem 2.2.7]{Hermander63}:

\begin{proposition}\label{prop6.12}
Let $0\leq q\in\mathbb{Z}$ and $\varphi\in\mathrm{OR}$. Then condition \eqref{f6.16} implies the continuous embedding $H^{\varphi}(\mathbb{R}^n)\hookrightarrow C^{q}_\mathrm{b}(\mathbb{R}^n)$. Conversely, if
\begin{equation}\label{f6.21}
\{w\in H^{\varphi}(\mathbb{R}^n): \mathrm{supp}\,w\subset G\}\subseteq C^q(\mathbb{R}^n)
\end{equation}
for an open nonempty set $G\subset\mathbb{R}^n$, then condition \eqref{f6.16} is satisfied.
\end{proposition}

\begin{proof}
We previously recall the definition of the H\"ormander space $\mathcal{B}_{p,k}$, which the embedding theorem deals with. Let $1\leq p\leq\infty$, and let a function $k:\mathbb{R}^{n}\to(0,\infty)$ satisfy the following condition: there exist positive numbers $c$ and $\ell$ such that
\begin{equation}\label{f6.22}
k(\xi+\zeta)\leq(1+c|\xi|)^{\ell}\,k(\zeta)\quad\mbox{for all}\quad
\xi,\zeta\in\mathbb{R}^{n}
\end{equation}
(the class of all such functions $k$ is denoted by $\mathcal{K}$). According to \cite[Definition 2.2.1]{Hermander63}, the complex linear space $\mathcal{B}_{p,k}$ consists of all distributions $w\in\mathcal{S}'(\mathbb{R}^{n})$ that their Fourier transform $\widehat{w}$ is locally Lebesgue integrable over $\mathbb{R}^{n}$ and that the product $k\widehat{w}$ belongs to the Lebesgue space $L_{p}(\mathbb{R}^{n})$. The space $\mathcal{B}_{p,k}$ is endowed with the norm $\|k\widehat{w}\|_{L_{p}(\mathbb{R}^{n})}$ and is complete with respect to it.

According to \cite[Theorem 2.2.7]{Hermander63} and its proof, the condition
\begin{equation}\label{f6.23}
\frac{(1+|\xi|)^{q}}{k(\xi)}\in L_{p'}(\mathbb{R}^{n})
\end{equation}
implies the inclusion $\mathcal{B}_{p,k}\subset C^{q}_\mathrm{b}(\mathbb{R}^n)$; here, as usual, the conjugate parameter $p'\in[1,\infty]$ is defined by $1/p+1/p'=1$. Moreover, if the set $\{w\in \mathcal{B}_{p,k}:\mathrm{supp}\,w\subset G\}$ lies in $C^{q}(\mathbb{R}^n)$ for an open nonempty set $G\subset\mathbb{R}^n$, then condition \eqref{f6.23} is satisfied. Note that the inclusion $\mathcal{B}_{p,k}\subset C^{q}_\mathrm{b}(\mathbb{R}^n)$ is continuous because its components are continuously embedded in a Hausdorff space, e.g. in $\mathcal{S}'(\mathbb{R}^{n})$.

The Hilbert space $H^{\varphi}(\mathbb{R}^n)$ is the H\"ormander space $\mathcal{B}_{2,k}$ provided that $k(\xi)=\varphi(\langle\xi\rangle)$ for every $\xi\in\mathbb{R}^n$ and that $k$ satisfies \eqref{f6.22}. Owing to \cite[Lemma~2.7]{MikhailetsMurach14}, the function $k(\xi):=\varphi(\langle\xi\rangle)$ of $\xi\in\mathbb{R}^n$ satisfies a weaker condition than \eqref{f6.22}; namely, there exist positive numbers $c_{0}$ and $\ell_0$ such that
\begin{equation*}
k(\xi+\zeta)\leq c_{0}(1+|\xi|)^{\ell_0}k(\zeta)
\quad\mbox{for all}\quad \xi,\zeta\in\mathbb{R}^{n}.
\end{equation*}
However, there exists a function $k_{1}\in\mathcal{K}$ that both  functions $k/k_{1}$ and $k_{1}/k$ are bounded on $\mathbb{R}^n$ (see \cite[the remark at the end of Section~2.1]{Hermander63}). Hence, the spaces $H^{\varphi}(\mathbb{R}^n)$ and $\mathcal{B}_{2,k_1}$ are equal with equivalence of norms. Thus, Proposition~\ref{prop6.12} holds true if we change \eqref{f6.16} for the condition $(1+|\xi|)^{q}/k_{1}(\xi)\in L_{2}(\mathbb{R}^{n})$. The latter is equivalent to
\begin{equation}\label{f6.24}
\int\limits_{\mathbb{R}^{n}}\,
\frac{\langle\xi\rangle^{2q}d\xi}{\varphi^{2}(\langle\xi\rangle)}<\infty
\end{equation}
It remains to show that $\eqref{f6.16}\Leftrightarrow\eqref{f6.24}$.

Passing to spherical coordinates with $r:=|\xi|$ and changing variables $t=\sqrt{1+r^{2}}$, we obtain
\begin{align*}
\int\limits_{\mathbb{R}^{n}}\,
\frac{\langle\xi\rangle^{2q}d\xi}{\varphi^{2}(\langle\xi\rangle)}&=
c_1\int\limits_{0}^{\infty}\,
\frac{(1+r^{2})^{q}\,r^{n-1}dr}{\varphi^{2}(\sqrt{1+r^{2}}\,)}=
c_1\int\limits_{1}^{\infty}\,
\frac{t^{2q+1}(t^{2}-1)^{n/2-1}dt}{\varphi^{2}(t)}\\
&=c_2+c_1\int\limits_{2}^{\infty}
\frac{t^{2q+1}(t^{2}-1)^{n/2-1}dt}{\varphi^{2}(t)}.
\end{align*}
Here, $c_1:=n\,\mathrm{mes}\,\widetilde{B}_{1}$, with the second factor being the volume of the unit ball in $\mathbb{R}^{n}$, and
$$
c_2:=c_1\int\limits_{1}^{2}
\frac{t^{2q+1}(t^{2}-1)^{n/2-1}dt}{\varphi^{2}(t)}<\infty
$$
because the function $1/\varphi$ is bounded on $[1,2]$ and because $n/2-1>-1$. Hence,
\begin{align*}
\eqref{f6.24}\,\Longleftrightarrow
\int\limits_{2}^{\infty}
\frac{t^{2q+1}(t^{2}-1)^{n/2-1}dt}{\varphi^{2}(t)}<\infty\,
\Longleftrightarrow
\int\limits_{2}^{\infty}\frac{t^{2q+n-1}dt}{\varphi^{2}(t)}<\infty
\,\Longleftrightarrow\,\eqref{f6.16}.
\end{align*}
\end{proof}

We systematically use the following auxiliary result:

\begin{lemma}\label{lem6.13}
Suppose that a function $\chi\in\mathrm{OR}$ is integrable over $[1,\infty)$. Then $\chi$ is bounded on $[1,\infty)$, and $t\chi(t)\to0$ as $t\to\infty$.
\end{lemma}

\begin{proof}
Let us prove by contradiction that $t\chi(t)\to0$ as $t\to\infty$. Assume the contrary; i.e., there exists a number $\varepsilon>0$ and a sequence $(t_{j})_{j=1}^{\infty}\subset[1,\infty)$ such that $t_{j}\to\infty$ as $j\to\infty$ and that $t_{j}\,\chi(t_{j})\geq\varepsilon$ for each $j\geq1$. Since $\chi\in\mathrm{OR}$, there are numbers $a>1$ and $c\geq1$ such that $c^{-1}\leq\chi(\lambda\tau)/\chi(\tau)\leq c$ for all $\tau\geq1$ and $\lambda\in[1,a]$. Since  $\chi$ is integrable over $[1,\infty)$, we have
\begin{equation}\label{f6.25}
\sum_{k=0}^{\infty}\int\limits_{a^k}^{a^{k+1}}\chi(t)dt<\infty.
\end{equation}
Choosing an integer $j\geq1$ arbitrarily, we find an integer $k(j)\geq0$ such that $a^{k(j)}\leq t_{j}<a^{k(j)+1}$ and observe that $c^{-1}\leq\chi(t)/\chi(t_{j})\leq c$ whenever $t\in[a^{k(j)},a^{k(j)+1}]$. Hence,
\begin{equation*}
\int\limits_{a^{k(j)}}^{a^{k(j)+1}}\chi(t)dt\geq
\int\limits_{a^{k(j)}}^{a^{k(j)+1}}c^{-1}\chi(t_{j})dt\geq
c^{-1}\varepsilon\,t_{j}^{-1}(a^{k(j)+1}-a^{k(j)})>
c^{-1}\varepsilon(1-a^{-1})
\end{equation*}
for each integer $j\geq1$, which contradicts \eqref{f6.25} because $c^{-1}\varepsilon(1-a^{-1})>0$ and $k(j)\to\infty$ as $j\to\infty$. Thus, our assumption is wrong; i.e., $t\chi(t)\to0$ as $t\to\infty$. It follows from this that the function $\chi\in\mathrm{OR}$ is bounded on $[1,\infty)$ because it is bounded
on each compact subinterval of $[1,\infty)$.
\end{proof}

\begin{proof}[Proof of Theorem $\ref{th6.8}$] \emph{Sufficiency.} Assume that $\varphi$ satisfies \eqref{f6.16} and prove that the spectral expansion \eqref{f6.1} converges unconditionally in the normed space  $C^{q}_{\mathrm{b}}(\mathbb{R}^{n})$ on the class $H^{\varphi}(\mathbb{R}^{n})$. Note first that $H^{\varphi}(\mathbb{R}^{n})\subset L_{2}(\mathbb{R}^{n})$ because the function $1/\varphi$ is bounded on $[1,\infty)$; the latter property follows from \eqref{f6.16} due to Lemma~\ref{lem6.13}.
We put $A:=I+L^{\ast}L$ and observe that $A$ is a positive definite self-adjoint unbounded linear operator in the Hilbert space $H=L_{2}(\mathbb{R}^{n})$ and that $\mathrm{Spec}\,A\subseteq[1,\infty)$. Here, $I$ is the identity operator in $L_{2}(\mathbb{R}^{n})$. It follows from the theorem on composition of PsDOs \cite[Theorem~1.2.4]{Agranovich94} that $A\in\Psi^{2m}(\mathbb{R}^{n})$ is uniformly elliptic on $\mathbb{R}^{n}$. Consider the functions $\chi(t):=\varphi(t^{1/(2m)})$ of $t\geq1$ and $\omega(z):=(\chi(1+|z|^{2}))^{-1}$ of $z\in\mathbb{C}$, and put $R:=\omega(L)=(1/\chi)(A)$ and $S:=I$ in Theorem~\ref{th6.2}. Since the function $1/\chi$ is bounded on $[1,\infty)$, the operator $R$ is bounded on $L_{2}(\mathbb{R}^{n})$, and $0\not\in\mathrm{Spec}\,\chi(A)$. It follows from the latter property that $H^{\chi}_{A}=\mathrm{Dom}\,\chi(A)$; hence, the operator $\chi(A)$ sets an isometric isomorphism between $H^{\chi}_{A}$ and $L_{2}(\mathbb{R}^{n})$. Thus,
\begin{equation*}
R(L_{2}(\mathbb{R}^{n}))=H^{\chi}_{A}=H^{\varphi}(\mathbb{R}^{n})\subset
C^{q}_{\mathrm{b}}(\mathbb{R}^{n})
\end{equation*}
due to \eqref{f4.20}, Proposition~\ref{prop6.12}, and our assumption~\eqref{f6.16}. Since the norms in the spaces $L_{2}(\mathbb{R}^{n})$ and $C^{q}_{\mathrm{b}}(\mathbb{R}^{n})$ are compatible, the operator $R$ acts continuously from $L_{2}(\mathbb{R}^{n})$ to $N=C^{q}_{\mathrm{b}}(\mathbb{R}^{n})$, as  was shown in Remark~\ref{rem6.3}. Thus, the operators $R$ and $S$ satisfy all the hypotheses of Theorem~\ref{th6.2}. According to this theorem, the spectral expansion \eqref{f6.1} converges unconditionally in the space $C^{q}_{\mathrm{b}}(\mathbb{R}^{n})$ at every vector $f\in (RS)(L_{2}(\mathbb{R}^{n}))=H^{\varphi}(\mathbb{R}^{n})$. The sufficiency is proved

\emph{Necessity.} Assume now that the spectral expansion \eqref{f6.1} converges unconditionally in $C^{q}_{\mathrm{b}}(\mathbb{R}^{n})$ on the class $H^{\varphi}(\mathbb{R}^{n})$. Then $f\in H^{\varphi}(\mathbb{R}^{n})$ implies $f=E(\mathbb{C})f\in C^{q}_{\mathrm{b}}(\mathbb{R}^{n})$ by Definition~\ref{def6.1}. Hence, $\varphi$ satisfies \eqref{f6.16} due to Proposition~\ref{prop6.12}. The necessity is also proved.
\end{proof}

\begin{proof}[Proof of Theorem $\ref{th6.10}$]
Consider the function $\chi_{j}(t):=\phi_{j}(t^{1/(2m)})$ of $t\geq1$ for each $j\in\{1,2\}$ and the functions $\eta(z):=(\chi_{1}(1+|z|^{2}))^{-1}$ and $\omega(z):=(\chi_{2}(1+|z|^{2}))^{-1}$ of $z\in\mathbb{C}$. Setting
$A:=I+L^{\ast}L$, we put $S:=\eta(L)=(1/\chi_{1})(A)$ and
$R:=\omega(L)=(1/\chi_{2})(A)$ in Theorem~\ref{th6.2}. The functions $\eta$ and $\omega$ are bounded on $\mathbb{C}$ by its hypotheses (note that the boundedness of $\omega$ follows from \eqref{f6.17} in view of Lemma~\ref{lem6.13}). Hence, the operators $R$ and $S$ are bounded on the Hilbert space $H=L_{2}(\mathbb{R}^{n})$. It follows from \eqref{f6.17} that $R$ acts continuously from $L_{2}(\mathbb{R}^{n})$ to $N=C^{q}_{\mathrm{b}}(\mathbb{R}^{n})$, as was shown in the proof of Theorem $\ref{th6.8}$ (the sufficiency). According to Theorem~\ref{th6.2} and Remark~\ref{rem6.4}, we have the estimate
\begin{equation}\label{f6.26}
\begin{aligned}
&\|f-E(\widetilde{B}_{\lambda})f\|_{C,q,\mathbb{R}^{n}}\\
&\leq
c'\cdot\|g\|_{\mathbb{R}^{n}}\cdot
\sup_{}\bigl\{(\phi_{1}(\langle z\rangle^{1/m})^{-1}:z\in\mathbb{C},|z|\geq\lambda\bigr\}
\cdot r_{g}(\widetilde{B}_{\lambda})
\end{aligned}
\end{equation}
for all $f\in RS(L_{2}(\mathbb{R}^{n}))$ and $\lambda>0$. Here, $c'$ denotes the norm of the bounded operator $R:L_{2}(\mathbb{R}^{n})\to C^{q}_{\mathrm{b}}(\mathbb{R}^{n})$, whereas $\|\cdot\|_{\mathbb{R}^{n}}$ stands for the norm in $L_{2}(\mathbb{R}^{n})$, and $g\in L_{2}(\mathbb{R}^{n})$ satisfies $f=RSg$. Note that $RS=(1/\chi)(A)$ where $\chi(t):=\chi_{1}(t)\chi_{2}(t)=\varphi(t^{1/(2m)})$ for every $t\geq1$. Since $0\not\in\mathrm{Spec}\,\chi(A)$, the operator $\chi(A)$ sets an isometric isomorphism between $H^{\chi}_{A}$ and $L_{2}(\mathbb{R}^{n})$. The inverse operator $RS$ sets an isomorphism between $L_{2}(\mathbb{R}^{n})$ and $H^{\varphi}(\mathbb{R}^{n})$ because the spaces $H^{\chi}_{A}$ and $H^{\varphi}(\mathbb{R}^{n})$ coincide up to equivalence of norms by \eqref{f4.20}. Hence, $c'\|g\|_{\mathbb{R}^{n}}\leq c\,\|f\|_{\varphi,\mathbb{R}^{n}}$
for some number $c>0$ that does not depend on $f$ and $\lambda$. Thus, formula \eqref{f6.26} yields the required estimate \eqref{f6.18} if we put $\theta_{f}(\lambda):=r_{g}(\widetilde{B}_{\lambda})$.
\end{proof}

\subsection{}\label{sec6.3}
As in Subsection~\ref{sec4.2}, let $\Gamma$ be a compact boundaryless   $C^{\infty}$-manifold of dimension $n\geq1$ endowed with a positive $C^{\infty}$-density $dx$. We suppose here that $L$ is a PsDO of class $\Psi^{m}(\Gamma)$ for some $m>0$ and that $L$ is elliptic on $\Gamma$. We may and will consider $L$ as a closed unbounded operator in the Hilbert space $H:=L_{2}(\Gamma)$ with $\mathrm{Dom}\,L=H^{m}(\Gamma)$ (see \cite[Sections 2.3~d and 3.1~b]{Agranovich94}). We also suppose that $L$ is a normal operator in $L_{2}(\Gamma)$. Then the Hilbert space $L_{2}(\Gamma)$ has an orthonormal basis $\mathcal{E}:=(e_{j})_{j=1}^{\infty}$ formed by some eigenvectors $e_{j}\in C^{\infty}(\Gamma)$ of $L$ (see, e.g., \cite[Section~15.2]{Shubin01}). Thus, the spectral expansion
\begin{equation}\label{f6.27}
f=\sum_{j=1}^{\infty}\varkappa_{j}(f)e_j,\quad\mbox{with}\quad \varkappa_{j}(f):=(f,e_j)_{\Gamma},
\end{equation}
holds in $L_{2}(\Gamma)$ for every $f\in L_{2}(\Gamma)$. (Recall that  $(\cdot,\cdot)_{\Gamma}$ and $\|\cdot\|_{\Gamma}$ respectively stand for the inner product and norm in $L_{2}(\Gamma)$.) These eigenvectors are enumerated so that $|\lambda_{j}|\leq|\lambda_{j+1}|$ whenever $j\geq1$, with $\lambda_{j}$ denoting the eigenvalue of $L$ such that $Le_j=\lambda_{j}e_j$. Note that $|\lambda_{j}|\to\infty$ as $j\to\infty$. Moreover, if $L$ is a classical PsDO, then
\begin{equation}\label{f6.28}
|\lambda_{j}|\sim\widetilde{c}\,j^{m/n}\quad\mbox{as}\quad j\to\infty,
\end{equation}
where $\widetilde{c}$ is a certain positive number that does not depend on~$j$.

As usual, $C^{q}(\Gamma)$ denotes the Banach space of all functions $u:\Gamma\to\mathbb{C}$ that are $q$ times continuously differentiable on $\Gamma$. The norm in this space is denoted by $\|\cdot\|_{C,q,\Gamma}$.

For the spectral expansion \eqref{f6.27}, the following versions of Theorems \ref{th6.8} and \ref{th6.10} hold true:

\begin{theorem}\label{th6.14}
Let $0\leq q\in\mathbb{Z}$ and $\varphi\in\mathrm{OR}$. The series \eqref{f6.27} converges unconditionally in the normed space $C^{q}(\Gamma)$ on the function class $H^{\varphi}(\Gamma)$ if and only if $\varphi$ satisfies \eqref{f6.16}.
\end{theorem}

\begin{theorem}\label{th6.15}
Let $0\leq q\in\mathbb{Z}$, and assume that the PsDO $L$ is classical. Suppose that certain functions $\phi_{1},\phi_{2}\in\mathrm{OR}$ satisfy the hypotheses of Theorem~$\ref{th6.10}$, and consider the function $\varphi:=\phi_{1}\phi_{2}\in\mathrm{OR}$ subject to \eqref{f6.16}. Then the degree of the convergence of \eqref{f6.27} in the normed space $C^{q}(\Gamma)$ on the class $H^{\varphi}(\Gamma)$ admits the estimate
\begin{equation}\label{f6.29}
\biggl\|f-\sum_{j=1}^{k}\varkappa_j(f)e_j\biggr\|_{C,q,\Gamma}
\leq c\cdot\|f\|_{\varphi,\Gamma}\cdot
\sup\bigl\{(\phi_{1}(j^{1/n}))^{-1}:k+1\leq j\in\mathbb{Z}\bigr\}
\cdot\theta_{f,k}
\end{equation}
for every function $f\in H^{\varphi}(\Gamma)$ and each integer $k\geq1$. Here, $c$ is a certain positive number that does not depend on $f$ and $k$, and $(\theta_{f,k})_{k=1}^{\infty}$ is a decreasing sequence that lies in $[0,1]$ and tends to zero.
\end{theorem}

We illustrate these theorems with analogous examples to those given in the previous subsection. Let $0\leq q\in\mathbb{Z}$, and let $c$ denote a positive number that does not depend on the function $f$ and integer $k$ from Theorem~\ref{th6.15}. Dealing with estimates of the form \eqref{f6.29}, we suppose that the PsDO $L$ is classical.

\begin{example}\label{ex6.3.1}
Owing to Theorem~\ref{th6.14}, the series \eqref{f6.27} converges unconditionally in $C^{q}(\Gamma)$ on the Sobolev class $H^{s}(\Gamma)$ if and only if $s>q+n/2$. This fact is known (see, e.g., \cite[Chapter~XII, Exercise~4.5]{Taylor81} in the $q=0$ case). Let $s>q+n/2$, and put $r:=s-q-n/2>0$. If  $0<\varepsilon<r/n$, then
\begin{equation*}
\biggl\|f-\sum_{j=1}^{k}\varkappa_j(f)e_j\biggr\|_{C,q,\Gamma}\leq
c\,\|f\|_{s,\Gamma}(k+1)^{\varepsilon-r/n}
\end{equation*}
for all $f\in H^{s}(\Gamma)$ and $k\geq1$, with $\|\cdot\|_{s,\Gamma}$ being the norm in $H^{s}(\Gamma)$. This estimate follows from Theorem~\ref{th6.15}, in which we put $\phi_{1}(t):=t^{r-n\varepsilon}$ and $\phi_{2}(t):=t^{s-r+n\varepsilon}$ for every $t\geq1$. The estimate admits the following refinement:
\begin{equation*}
\biggl\|f-\sum_{j=1}^{k}\varkappa_j(f)e_j\biggr\|_{C,q,\Gamma}\leq
c\,\|f\|_{s,\Gamma}(k+1)^{-r/n}\log^{\varepsilon+1/2}(k+1)
\end{equation*}
for the same $f$ and $k$, we choosing a real number $\varepsilon>0$ arbitrarily. This estimate follows from Theorem~\ref{th6.15} applied to the functions \eqref{f6.18b}.
\end{example}

\begin{example}\label{ex6.3.2}
We choose a number $\varrho>0$ arbitrarily and define a function $\varphi$ by formula \eqref{f6.19}. According to Theorem~\ref{th6.14}, the series \eqref{f6.27} converges unconditionally in $C^{q}(\Gamma)$ on the class $H^{\varphi}(\Gamma)$. This fact is known at least in the $q=0$ case (see \cite[Chapter~XII, Exercise~4.8]{Taylor81}). If $0<\varepsilon<\varrho$, then
\begin{equation*}
\biggl\|f-\sum_{j=1}^{k}\varkappa_j(f)e_j\biggr\|_{C,q,\Gamma}\leq
c\,\|f\|_{\varphi,\Gamma}\log^{\varepsilon-\varrho}(k+1)
\end{equation*}
for all $f\in H^{\varphi}(\Gamma)$ and $k\geq1$. This estimate follows from Theorem~\ref{th6.15} if we represent $\varphi$ in the form used in  Example~\ref{ex6.2.2}. Comparing this result with the previous example, we see that $H^{\varphi}(\Gamma)$ is broader than the union
\begin{equation*}
H^{q+n/2+}(\Gamma):=\bigcup_{s>q+n/2}H^{s}(\Gamma).
\end{equation*}
\end{example}

\begin{example}\label{ex6.3.3}
We choose a number $\varrho>0$ arbitrarily and define a function $\varphi$ by formula \eqref{f6.20}. Owing to Theorem~\ref{th6.14}, the series \eqref{f6.27} converges unconditionally in $C^{q}_{\mathrm{b}}(\mathbb{R}^{n})$ on the class $H^{\varphi}(\mathbb{R}^{n})$. This class is broader than that used in Example~\ref{ex6.3.2}. If $0<\varepsilon<\varrho$, then
\begin{equation*}
\biggl\|f-\sum_{j=1}^{k}\varkappa_j(f)e_j\biggr\|_{C,q,\Gamma}\leq
c\cdot\|f\|_{\varphi,\Gamma}\cdot(\log\log(k+2))^{\varepsilon-\varrho}
\end{equation*}
for all $f\in H^{\varphi}(\Gamma)$ and $k\geq1$. This bound follows from Theorem~\ref{th6.15} if we represent $\varphi$ in the form given in  Example~\ref{ex6.2.3}.
\end{example}

These results are applicable to multiple trigonometric series. Indeed, if $\Gamma=\mathbb{T}^{n}$ and $A=\Delta_{\Gamma}$, then \eqref{f6.27} becomes the expansion of $f$ into the $n$-multiple trigonometric series (as usual, $\mathbb{T}:=\{e^{i\tau}:0\leq\tau\leq 2\pi\}$). It is known \cite[Section~6]{Golubov84} that this series is unconditionally uniformly convergent (on $\Gamma$) on every H\"older class $C^{s}(\Gamma)$ of order $s>n/2$. The exponent $n/2$ is critical here; namely, there exists a function $f\in C^{n/2}(\Gamma)$ whose trigonometric series diverges at some point of $\mathbb{T}^{n}$. These results consist a multi-dimensional generalization of Bernstein's theorem for trigonometric series. Since $C^{s}(\Gamma)\subset H^{s}(\Gamma)$, Example \ref{ex6.3.1} gives a weaker sufficient condition for this convergent. The next Examples \ref{ex6.3.2} and \ref{ex6.3.3} treat the case of the critical exponent with the help of generalized Sobolev spaces.

The proofs of Theorems \ref{th6.14} and \ref{th6.15} are similar to the proofs of Theorems \ref{th6.8} and \ref{th6.10}, we using Theorem~\ref{th6.5} (instead of  Theorem~\ref{th6.2}) and the following analog of Proposition~\ref{prop6.12}:

\begin{proposition}\label{prop6.16}
Let $0\leq q\in\mathbb{Z}$ and $\varphi\in\mathrm{OR}$. Then condition \eqref{f6.16} is equivalent to the embedding $H^{\varphi}(\Gamma)\subseteq C^{q}(\Gamma)$. Moreover, this embedding is compact under condition \eqref{f6.16}.
\end{proposition}

\begin{proof}
Suppose first that $\varphi$ satisfies condition~\eqref{f6.16}. Then the continuous embedding $H^\varphi(\mathbb{R}^n)\hookrightarrow C^q_\mathrm{b}(\mathbb{R}^n)$ holds true by Proposition~\ref{prop6.12}.
Let $\varkappa$, $\chi_j$, and $\pi_j$ be the same as those in the definition of $H^{\varphi}(\Gamma)$. Choosing $f\in H^\varphi(\Gamma)$ arbitrarily, we get the inclusion
\begin{equation*}
(\chi_jf)\circ\pi_j\in H^\varphi(\mathbb{R}^n)\hookrightarrow C^q_\mathrm{b}(\mathbb{R}^n)
\end{equation*}
for each $j\in\{1,\ldots,\varkappa\}$. Hence, each $\chi_jf\in C^q(\Gamma)$, which implies that
\begin{equation*}
f=\sum_{j=1}^\varkappa\chi_j f\in C^q(\Gamma).
\end{equation*}
Thus, $H^\varphi(\Gamma)\subseteq C^q(\Gamma)$; this embedding is continuous because both the spaces are complete and continuously embedded in  $\mathcal{D}'(\Gamma)$. Let us prove that it is compact.

We showed in  Remark~\ref{rem6.11} that $\varphi=\phi_{1}\phi_{2}$ for some functions $\phi_{1}$ and $\phi_{2}$ satisfying the hypotheses of Theorem~\ref{th6.10}. Since $\phi_{2}(t)/\varphi(t)=1/\phi_{1}(t)\to0$ as $t\to\infty$, the compact embedding $H^\varphi(\Gamma)\hookrightarrow H^{\phi_{2}}(\Gamma)$ holds true. Indeed, let $T$ and $K$ be the bounded operators \eqref{f4.11} and \eqref{f4.18}. If a sequence $(f_{k})$ is bounded in $H^\varphi(\Gamma)$, then the sequence $(Tf_{k})$ is bounded
in $(H^{\varphi}(\mathbb{R}^n))^{\varkappa}$. It follows from this by \cite[Theorem~2.2.3]{Hermander63} that the latter sequence contains a convergent subsequence $(Tf_{k_\ell})$ in $(H^{\phi_{2}}(\mathbb{R}^n))^{\varkappa}$. Hence,
the subsequence of vectors $f_{k_\ell}=KTf_{k_\ell}$ is convergent in $H^{\phi_{2}}(\Gamma)$. Thus, the embedding $H^\varphi(\Gamma)\hookrightarrow H^{\phi_{2}}(\Gamma)$ is compact. As we showed in the previous paragraph, the continuous embedding $H^{\phi_{2}}(\Gamma)\hookrightarrow C^q(\Gamma)$ holds true because $\phi_{2}$ satisfies \eqref{f6.17}. Therefore, the embedding $H^\varphi(\Gamma)\hookrightarrow C^q(\Gamma)$ is compact.

Assume now that the embedding $H^\varphi(\Gamma)\subseteq C^q(\Gamma)$ holds true, and prove that $\varphi$ satisfies \eqref{f6.16}.
We suppose without loss of generality that $\Gamma_1$ is not contained in $\Gamma_2\cup\cdots\cup\Gamma_\varkappa$, choose an open nonempty set $U\subset\Gamma_1$ which satisfies $U\cap\Gamma_j=\emptyset$ whenever $j\neq1$, and put $G:=\pi^{-1}_1(U)$. Consider an arbitrary distribution $w\in H^{\varphi}(\mathbb{R}^n)$ subject to $\mathrm{supp}\,w\subset G$.
Owing to \eqref{f4.18} and our assumption, we have the inclusion
\begin{equation*}
u:=K(w,\underbrace{0,\ldots,0}_{\varkappa-1}\,)\in H^{\varphi}(\Gamma) \subseteq C^q(\Gamma).
\end{equation*}
Hence, $w=(\chi_1 u)\circ\pi_1\in C^q(\mathbb{R}^{n})$; note that the letter equality is true because $\chi_1=1$ on $U$. Thus, \eqref{f6.21} holds true, which implies \eqref{f6.16} due to Proposition~\ref{prop6.12}.
\end{proof}

\begin{proof}[Proof of Theorem $\ref{th6.14}$]
\emph{Sufficiency} is proved in the same manner as the proof of the sufficiency in Theorem~\ref{th6.8}. We only replace $\mathbb{R}^{n}$ with $\Gamma$ and use Theorem~\ref{th6.5} instead of Theorem~\ref{th6.2} and Proposition~\ref{prop6.16} instead of
Proposition~\ref{prop6.12}.

\emph{Necessity.} Assume that the series \eqref{f6.27} converges in $C^{q}(\Gamma)$ on the class $H^{\varphi}(\Gamma)$. Then $H^{\varphi}(\Gamma)\subseteq C^{q}(\Gamma)$, which implies \eqref{f6.16} by Proposition~\ref{prop6.16}.
\end{proof}

\begin{proof}[Proof of Theorem $\ref{th6.15}$.] It is very similar to the proof of Theorem $\ref{th6.10}$. Replacing $\mathbb{R}^{n}$ with $\Gamma$ in this proof and using Theorem~\ref{th6.5} and Remark~\ref{rem6.6} instead of Theorem~\ref{th6.2} and Remark~\ref{rem6.4}, we obtain the following analog of the estimate \eqref{f6.12}:
\begin{equation}\label{f6.30}
\biggl\|f-\sum_{j=1}^{k}\varkappa_j(f)e_j\biggr\|_{C,q,\Gamma}\leq
c'\cdot\|g\|_{\Gamma}\cdot\sup_{j\geq k+1}
\bigl\{(\phi_{1}(\langle\lambda_j\rangle^{1/m}))^{-1}\bigr\}
\cdot r_{g,k}
\end{equation}
for every function $f\in H^{\varphi}(\Gamma)$ and each integer $k\geq1$. Here, $c'$ denotes the norm of the bounded operator $R:L_{2}(\Gamma)\to C^{q}(\Gamma)$, and $g:=(RS)^{-1}f\in L_{2}(\Gamma)$. Reasoning in the same way as that given after formula \eqref{f6.26}, we arrive at the inequality $c'\|g\|_{\Gamma}\leq c''\|f\|_{\varphi,\Gamma}$ where the number $c''>0$ does not depend on $f$ and $k$. Besides, owing to the inclusion  $\phi_{1}\in\mathrm{OR}$ and asymptotic formula \eqref{f6.28}, the exist two positive numbers $c_{1}$ and $c_{2}$ such that
\begin{equation*}
c_{1}\phi_{1}(j^{1/n})\leq
\phi_{1}(\langle\lambda_j\rangle^{1/m})\leq
c_{2}\,\phi_{1}(j^{1/n})
\end{equation*}
for every integer $j\geq1$. Thus, formula \eqref{f6.30} yields the required estimate \eqref{f6.29} if we put $\theta_{f,k}:=r_{g,k}$ and $c:=c''/c_{1}$.
\end{proof}

\subsection{}\label{sec6.last}
We end Section~\ref{sec6} with two sufficient conditions under which
the spectral expansion \eqref{f6.27} converges a.e. (almost everywhere) on the manifold $\Gamma$ with respect to the measure induced by the $C^{\infty}$-density $dx$. These conditions are formulated in terms of belonging of $f$ to some generalized Sobolev spaces on $\Gamma$. Put
\begin{equation*}
S^{\ast}(f,x):=\sup_{1\leq k<\infty}\,
\biggl|\,\sum_{j=1}^{k}\;\varkappa_{j}(f)e_{j}(x)\,\biggr|
\end{equation*}
for all $f\in L_{2}(\Gamma)$ and $x\in\Gamma$; thus, $S^{\ast}(f,x)$ is the majorant of partial sums of \eqref{f6.27}. Consider the function $\log^{\ast}t:=\max\{1,\log t\}$ of $t\geq1$; it pertains to $\mathrm{OR}$. We suppose that the PsDO $L$ is classical.

\begin{theorem}\label{th6.17}
The series \eqref{f6.27} converges a.e. on $\Gamma$ on the function class $H^{\log^{\ast}}(\Gamma)$. Besides, there exists a number $c>0$ such that
\begin{equation*}
\|S^{\ast}(f,\cdot)\|_{\Gamma}\leq c\,\|f\|_{\log^{\ast},\Gamma}
\quad\mbox{for every}\quad f\in H^{\log^{\ast}}(\Gamma).
\end{equation*}
\end{theorem}

If $f\in H^{\log^{\ast}}(\Gamma)$, then the convergence of the series \eqref{f6.27} may be violated under a permutation of its terms. To ensure that the convergence does not depend on their order, we should subject $f$ to a stronger condition.

\begin{theorem}\label{th6.18}
Assume that a function $\varphi\in\mathrm{OR}$ (nonstrictly) increases and satisfies
\begin{equation}\label{f6.32}
\int\limits_{2}^{\infty}\frac{dt}{t\,(\log t)\,\varphi^{2}(t)}<\infty.
\end{equation}
Then the series \eqref{f6.27} converges unconditionally a.e. on $\Gamma$ on the function class $H^{\varphi\log^{\ast}}(\Gamma)$.
\end{theorem}

These theorems are proved in \cite[Section~2.3.2]{MikhailetsMurach14}, the second being demonstrated in the case where $\varphi$ varies slowly at infinity in the sense of Karamata. The proofs rely on Theorem~\ref{th5.8} and general forms of the classical Menshov--Rademacher \cite{Menschoff23, Rademacher22} and Orlicz \cite{Orlicz27} theorems about a.e. convergence of orthogonal series. We give these brief proofs for the sake of completeness.

\begin{proof}[Proof of Theorem~$\ref{th6.17}$.]
Note that the orthonormal basis $\mathcal{E}$ of $L_{2}(\Gamma)$ consists of eigenvectors of the operator $A:=(I+L^{\ast}L)^{1/m}$, to which Theorem~\ref{th5.8} is applicable. Owing to this theorem, we have
\begin{equation*}
\sum_{j=1}^{\infty}(\log^{2}(j+1))|\varkappa_{j}(f)|^{2}\asymp
\sum_{j=1}^{\infty}(\log^{\ast}(j^{1/n}))^{2}\,|\varkappa_{j}(f)|^{2}
\asymp\|f\|_{\log^{\ast},\Gamma}^{2}<\infty
\end{equation*}
whenever $f\in H^{\log^{\ast}}(\Gamma)$, with $\asymp$ meaning equivalence of norms. Now Theorem~\ref{th6.17} follows from the Menshov--Rademacher theorem, which remains true for general complex orthogonal series formed by square integrable functions (see, e.g., \cite{Meaney07,  MikhailetsMurach11MFAT4, MoriczTandori96}).
\end{proof}

\begin{proof}[Proof of Theorem~$\ref{th6.18}$.]
Let $f\in H^{\varphi\log^{\ast}}(\Gamma)$, and put $\omega_{j}:=\varphi^{2}(j^{1/n})$ for every integer $j\geq1$. Owing to Theorem~\ref{th5.8} applied to $A:=(I+L^{\ast}L)^{1/m}$, we have
\begin{equation}\label{f6.33}
\sum_{j=2}^{\infty}(\log^{2}j)\,\omega_{j}\,|\varkappa_{j}(f)|^{2}\asymp
\|f\|_{\varphi\log^{\ast},\Gamma}^{2}<\infty.
\end{equation}
Besides, condition \eqref{f6.32} implies that
\begin{equation}\label{f6.34}
\sum_{j=3}^{\infty}\frac{1}{j\,(\log j)\,\omega_{j}}\leq
\int\limits_{2}^{\infty}
\frac{d\tau}{\tau\,(\log\tau)\,\varphi^{2}(\tau^{1/n})}= \int\limits_{2^{1/n}}^{\infty}
\frac{n\,t^{n-1}\,dt}{t^{n}\,n\,(\log
t)\,\varphi^{2}(t)}<\infty.
\end{equation}
The conclusion of Theorem \ref{th6.18} follows from \eqref{f6.33} and \eqref{f6.34} due to the Orlicz theorem (in Ul'janov's equivalent statement \cite[Section~9, Subsection~1]{Uljanov64}), which remains true for general complex orthogonal series \cite[Theorem~2]{MikhailetsMurach12UMJ10} (see also \cite[Theorem~3]{MikhailetsMurach11MFAT4}).
\end{proof}

As to Theorems \ref{th6.17} and \ref{th6.18}, note the following: if we restrict ourselves to the Sobolev spaces, we will assert only that the series \eqref{f6.27} converges unconditionally a.e. on $\Gamma$ on the function class $H^{0+}(\Gamma):=\bigcup_{s>0}H^{s}(\Gamma)$ (cf. \cite{Meaney82}). This class is significantly narrower than the spaces used in these theorems. Using the extended Sobolev scale, we express in adequate forms the hypotheses of the Menshov--Rademacher and Orlicz theorems.

\end{document}